{\immediate\write18{bibtex \jobname}}{}

\documentclass[english,svgnames]{article}

\usepackage[a4paper,margin=1.5in]{geometry}

\usepackage[utf8]{inputenc}
\usepackage{hyperref}
\usepackage{amsmath,amssymb,amsthm}
\usepackage{mathtools}
\usepackage{thmtools, thm-restate}
\usepackage{enumerate}
\usepackage{csquotes}
\usepackage{tikz-cd}
\usetikzlibrary{fadings}
\usepackage[all]{xy}
\usepackage{graphicx}
\usepackage[capitalize]{cleveref}
\usepackage{import}

\declaretheorem[name=Theorem, numberwithin=section]{thm}
\declaretheorem[name=Proposition, sibling=thm]{prp}
\declaretheorem[name=Corollary, sibling=thm]{cor}
\declaretheorem[name=Lemma, sibling=thm]{lem}

\declaretheorem[name=Definition, sibling=thm, style=definition]{definition}
\declaretheorem[name=Remark, sibling=thm, style=remark]{remark}

\numberwithin{equation}{section}
\numberwithin{figure}{section}

\newcommand{\R}{\mathbb{R}}

\newcommand{\Z}{\mathbb{Z}}
\newcommand{\N}{\mathbb{N}}
\newcommand{\M}{\mathbb{M}}
\newcommand{\U}{\mathbb{U}}

\newcommand{\op}[1]{\operatorname{#1}}

\DeclarePairedDelimiterX\set[1]\lbrace\rbrace{\def\given{\;\delimsize\vert\;}#1}

\newcommand{\smashcap}{\sqcap}

\DeclareMathOperator{\pr}{pr}

\newcommand{\CSpine}{MediumSeaGreen}
\newcommand{\CHorn}{OrangeRed}
\newcommand{\CCyl}{RoyalBlue}

\begin{document}

\title{Universal Distances for Extended Persistence}

\author{Ulrich Bauer \and Magnus Bakke Botnan \and Benedikt Fluhr}

\maketitle

\begin{abstract}
The extended persistence diagram is an invariant of piecewise linear functions, which is known to be stable under perturbations of functions with respect to the bottleneck distance as introduced by Cohen-Steiner, Edelsbrunner, and Harer.
We address the question of universality, which asks for the largest possible stable distance on extended persistence diagrams, showing that a more discriminative variant of the bottleneck distance is universal.
Our result applies more generally to settings where persistence diagrams are considered only up to a certain degree.
We achieve our results by establishing a functorial construction and several characteristic properties of relative interlevel set homology, which mirror the classical Eilenberg--Steenrod axioms.
Finally, we contrast the bottleneck distance
with the interleaving distance of sheaves on the real line by showing that the latter is not intrinsic, let alone universal.
This particular result has the further implication
that the interleaving distance of Reeb graphs is not intrinsic either.
\end{abstract}

\section{Introduction}
The core idea of topological persistence is to construct invariants of continuous real-valued functions by considering preimages and applying homology with coefficients in a fixed field~$K$,
or any other functorial invariant from algebraic topology.
The most basic incarnation of this idea studies the homology of sublevel sets.
\emph{Sublevel set persistent homology} was introduced in \cite{MR1949898}
and is described up to isomorphism
by the \emph{sublevel set persistence diagram} \cite{MR2279866}.
An extension of this invariant
considers preimages of arbitrary intervals, where non-closed intervals are treated using relative homology
\cite{MR2352705,Carlsson:2009:ZPH:1542362.1542408}.
This is commonly referred to as the \emph{Mayer--Vietoris pyramid}, due to its connection to the relative Mayer--Vietoris sequence.

As shown by \cite{Carlsson:2009:ZPH:1542362.1542408},
interlevel set persistent homology
of a PL function $f \colon X \rightarrow \R$
on a finite connected simplicial complex,
is described up to isomorphism by the \emph{extended persistence diagram}
$\op{Dgm}(f)$
due to \cite{Cohen-Steiner2009}.
It is well-known that the operation ${f\mapsto \op{Dgm}(f)}$ is \emph{stable} \cite{Cohen-Steiner2009,MR3988214}, by an extension of the classical stability result for sublevel set persistence \cite{MR2279866} to the setting of extended persistence.
Specifically, for functions $f,g\colon X\to \R$ as above,
\[d_{B}(\op{Dgm}(f), \op{Dgm}(g)) \leq ||f-g||_\infty,\]
where $d_{B}(\op{Dgm}(f), \op{Dgm}(g))$ is the bottleneck distance
of $\op{Dgm}(f)$ and $\op{Dgm}(g)$.

In this paper we prove that this distance is \emph{universal}: it is the largest possible stable distance on extended persistence diagrams realized by functions.
As a note of caution, we point out
that universality is only achieved by a version of the bottleneck distance in which any vertex contained inside the \emph{extended subdiagram}
has to be matched, as is made precise in \cref{remark:matchings4universality}.

We obtain this universality result by exploring the structure behind extended persistence.
As observed already in \cite{Carlsson:2009:ZPH:1542362.1542408}, the connecting homomorphisms allow us to assemble the Mayer--Vietoris pyramids of all degrees into a single diagram, with the shape of an infinite strip.
We denote this strip by $\mathbb{M}$, alluding to the fact that it can naturally be interpreted as the universal cover of a Möbius strip whose points correspond to pairs of preimages of $f$ whose relative homology we want to consider, and we refer to the strip-shaped diagram as the \emph{relative interlevel set homology} of the function $f$ and denote it by $h(f) \colon \M \rightarrow \mathrm{Vect}_K$.
The resulting diagram now consists of all relative homology groups of those pairs, separated by homological degrees, and with all possible maps induced by either inclusions or connecting homomorphisms.
It is easily overlooked that establishing the commutativity of this strip-shaped diagram actually raises several technical challenges;
in the present paper, we give a rigorous construction, extending the definitions from the previous literature to also make the gluing of the Mayer--Vietoris pyramids explicit in a functorial way.

We also define the associated extended persistence diagram $\op{Dgm}(f)$
in terms of the functor $h(f)$ in \cref{sec:persDiagram}
as a multiset on the interior of $\M$.
Note that we write $\op{int} \M$ to denote this interior of $\M$ as usual in topology, and not the intervals in $\M$, which often play a role in persistence theory too.
While our definition turns out to be equivalent to the familiar one from \cite{Cohen-Steiner2009}, our approach is natural in the context of relative interlevel set persistence, clarifying the above mentioned conditions on matchings required to yield a universal bottleneck distance.
In \cref{sec:struct} we use a structure theorem for relative interlevel set homology to show that the ordinary, relative, and extended subdiagrams,
as originally defined in \cite{Cohen-Steiner2009},
can be obtained by restriction
to the corresponding regions depicted in \cref{fig:subdiagrams}.
\begin{figure}[t]
\centering
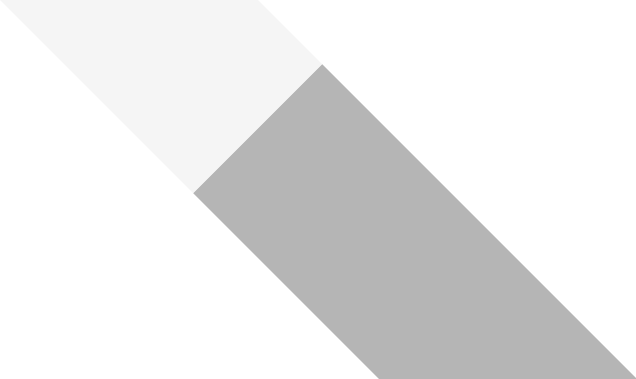
\caption{
Regions in the strip \(\M\) corresponding to the
ordinary, relative, and extended subdiagrams \cite{Cohen-Steiner2009}.
The support %
of extended persistence diagrams
is shaded in grey.
}
\label{fig:subdiagrams}
\end{figure}

In practice it may be undesirable or even intractable
to compute the entire extended persistence diagram
of $f \colon X \rightarrow \R$.
Instead, it might be preferable
to compute just the restriction of $\op{Dgm}(f) \colon \op{int} \M \rightarrow \N_0$
to some closed upset $\U \subseteq \M$.
Now in order for the bottleneck distance
of \cref{dfn:bottleneckDist} below
to be universal,
when applied to such restricted persistence diagrams,
we need to impose some restrictions on the upset $\U \subseteq \M$ (arising from the fact
that \cref{lem:matching} below is invalid
for arbitrary closed upsets $\U \subseteq \M$).
In \cref{fig:subdiagrams}
the grey shaded region is triangulated
by the solid and the dashed lines,
with each triangle corresponding to a face of a Mayer--Viatoris pyramid.
We say that a closed upset $\U \subseteq \M$ is \emph{admissible},
if it contains the region labeled $\op{Ext}_0$
and
if it is compatible with this triangulation
in the sense that
the interior of each triangle
is either fully contained in or disjoint from $\U$.
Moreover,
for an admissible upset $\U \subseteq \M$
we say that a multiset $\mu \colon \op{int} \U \rightarrow \N_0$
is a \emph{realizable persistence diagram}
if $\mu = \op{Dgm}(f) |_{\op{int} \U}$ for some function $f$ as above.
Roughly speaking, we may think of the restriction of
$\op{Dgm}(f) \colon \op{int} \M \rightarrow \N_0$
to the interior of an admissible upset $\U \subset \M$
as the graded barcode truncated at some upper bound for the degree
with respect to a certain kind of persistence.
Using the terminology introduced by
\cite{Carlsson:2009:ZPH:1542362.1542408},
the five kinds of persistence that may occur
for an admissible upset $\U \subset \M$
are levelset (zigzag), up-down, down-up, extended persistence,
and extended persistence of $-f$.
For extended persistence,
$\U$ is the union of all regions corresponding to subdiagrams
not exceeding a certain degree
as depicted in \cref{fig:subdiagrams}.
Universality now follows immediately from the following slightly stronger theorem.
\begin{restatable}{thm}{thmSuffCond}
  \label{thm:suffCond}
  \sloppy
  For an admissible upset $\U$
  and any two realizable persistence diagrams
  $\mu, \nu \colon \op{int} \U \rightarrow \N_0$
  with $d_B (\mu, \nu) < \infty$,
  there exists a finite simplicial complex $X$ and piecewise linear functions
  $f, g \colon X \rightarrow \R$ with
  \begin{equation*}
    \op{Dgm}(f) |_{\op{int} \U} = \mu,
    \quad
    \op{Dgm}(g) |_{\op{int} \U} = \nu,
    ~
    \text{and}
    \quad
    d_B (\mu, \nu)
    =
    \lVert f-g \rVert_{\infty}
    .
  \end{equation*}
\end{restatable}
Since the stability inequality $d_B (\mu, \nu) \leq \lVert f-g\rVert_{\infty}$ is attained for all $\mu, \nu$, the bottleneck distance is universal.
Note that this also immediately implies that the bottleneck distance is geodesic (and hence intrinsic), with an explicit geodesic between $\mu, \nu$ being given by the 1-parameter family of extended persistence diagrams of the convex combinations $(1-\lambda)f + \lambda g$ of $f$ and $g$;
see also \cref{cor:bottleneckGeodesic}.

An analogous realizability theorem in the context of sublevel set persistent homology has already been given in \cite{Lesnick2015}, and special cases have been considered in several other works \cite{DeSha-2021,MR4130976}.
Universality of metrics has further been studied for various other common constructions in topological data analysis, such as Reeb graphs \cite{MR4323622}, contour trees and merge trees \cite{MR4470893,MR4470903}, and the general setting of locally persistent categories \cite{scoccola-2020}.

We stress that the persistence diagrams considered here 
contain
information about homology in \emph{any degree} represented in $\U$.
It is worth noting that the realization of a \emph{pair} of extended persistence diagrams is more intricate than the realization of a single one.
For example, if $\U$ is the region corresponding to the $0$th homology of level sets (containing the triangles labeled $\mathrm{Ord}_0$, $\mathrm{Ext}_0$, and the adjacent triangles of $\mathrm{Ext}_0$ and $\mathrm{Ext}_1$),
it is possible to realize any single admissible persistence diagram as a $1$-dimensional complex (a \emph{Reeb graph}), for \emph{pairs} of admissible persistence diagrams this is not always possible in the presence of loops that are left unmatched.
Our construction yields a $2$-dimensional complex in this case.

In \cite{2021arXiv210809298B} we provide another construction
of the extended persistence diagram,
which is closely related to the construction we provide here
but requires fewer tameness assumptions.
More specifically, in \cite{2021arXiv210809298B} we use
(singular) cohomology in place of singular homology
and we take preimages of open subsets in place of closed subsets.
Moreover,
we show in \cref{sec:equivalenceOfDefs}
that the definition of extended persistence diagrams
we provide here is equivalent
to the original definition
and we show the same for the definition in terms of cohomology
in \cite[Section 3.2.2]{2021arXiv210809298B}.
Thus,
both constructions yield identical extended persistence diagrams.
As is implied by \cref{thm:suffCond},
it suffices to consider piecewise linear functions
on finite simplicial complexes
to prove universality for realizable persistence diagrams.
For this reason, singular homology and preimages of closed subsets
are sufficient in the present context.
Moreover, considering preimages of closed subsets,
the translation from singular homology to simplicial homology
is very straightforward.
We use this connection to provide several properties
in \cref{sec:props},
which are essential to the soundness of our computations.
We leave it to future work to generalize the results from \cref{sec:props}
to the weaker tameness assumptions in \cite{2021arXiv210809298B}.

A notable feature of the bottleneck distance is
that it can be defined generically over an admissible upset $\U \subseteq \M$
in a straightforward way
as in \cref{dfn:bottleneckDist}
while at the same time being universal.
In \cref{sec:contrastInterleavingDist}
we contrast this to interleaving distances of sheaves.
By \cite{Carlsson:2009:ZPH:1542362.1542408,Bendich-2013}
and as summarized in \cite[Section 3.2.1]{2021arXiv210809298B}
the extended persistence diagram
is equivalent to the \emph{graded level set barcode}.
Moreover, \cite{MR4355732} endow graded level set barcodes
with a bottleneck distance
and it is easy to see
that the correspondence
we describe in \cite[Section 3.2.1]{2021arXiv210809298B}
yields an isometry
between extended persistence diagrams
and graded level set barcodes,
each of the two sets endowed with their corresponding bottleneck distance.
Furthermore,
by the \mbox{\cite[Isometry Theorem 5.10]{MR4355732}}
the interleaving distance of \emph{derived level set persistence}
by \cite{MR3259939,MR3873181}
and the bottleneck distance
of corresponding graded level set barcodes are identical.
In particular,
the interleaving distance of derived level set persistence
is universal as a result of \cref{thm:suffCond}.
However,
\cref{cor:factorIntrinsic} implies that the interleaving distance of sheaves,
which can be seen as a counterpart
to derived level set persistence in degree $0$,
is not intrinsic,
let alone universal.
As a further consequence of \cref{cor:factorIntrinsic},
the interleaving distance of Reeb graphs is not intrinsic either,
which answers a question raised in \cite[Section 4]{MR3685697}.

\section{Preliminaries}
\label{sec:prelim}

In this section, we formalize the requisite notions of relative interlevel set homology and persistence diagrams, building on and extending several ideas that appear in the relevant literature, and aiming for an explicit description of those ideas.
In particular, we will assemble all relevant persistent homology in one single functor, which will turn out to be a helpful and elucidating perspective for studying the persistent homology associated to a function.

\subsection{Relative Interlevel Set Homology}
\label{sec:InterlevelHom}

For a piecewise linear function $f \colon X \rightarrow \R$,
the inverse image map $f^{-1}$ provides an order-preserving map
from the poset of compact intervals to $2^{X}$.
The image of this map consists of all \emph{interlevel sets} of $f$,
where an interlevel set is a preimage $f^{-1} (I)$ of a closed interval $I$.
Post-composing this map with homology
we obtain a functor from the poset of compact intervals
to the category of graded vector spaces over $K$:
\begin{equation}
  \label{eq:interlevelsetHomology}
  I \mapsto H_* (f^{-1} (I))
  .
\end{equation}
This invariant is commonly referred to as
\emph{interlevel set homology}.
As proposed by \cite{Carlsson:2009:ZPH:1542362.1542408}
for the discretely indexed setting,
we consider the following extension of this invariant:
\begin{enumerate}
\item
  For interlevel sets of the form $f^{-1} (I \cup J)$ with $I,J$ overlapping intervals,
  there exist connecting maps
  $H_d (f^{-1} (I \cup J)) \to H_{d-1} (f^{-1} (I \cap J))$
  from a Mayer--Vietoris sequence,
  which we want to include in our structure.
\item
  In addition to absolute homology groups we would
  like to include relative homology groups as well.
  More specifically, we include
  all homology groups
  of preimages of
  pairs of closed subspaces of
  $\overline{\R} := [-\infty, \infty]$ whose
  difference is an interval contained in $\R$.
  By excision, each such relative homology group can be written as
  $H_n (f^{-1} (I, C))$, where $I$ is a closed interval
  and $C$ is the complement of an open interval.
\end{enumerate}
Extending the interlevel set homology functor
\eqref{eq:interlevelsetHomology}
in the first direction leads to the
notion of a Mayer--Vietoris system as defined in
\cite[Definition 2.14]{2019arXiv190709759B}.
On the other hand, the construction by \cite{Bendich-2013}
gives an extension in the second direction.
The \emph{relative interlevel set homology}
combines both extensions
to obtain a continuously indexed version of the construction by
\cite{Carlsson:2009:ZPH:1542362.1542408}.
Specifically, we encode all information as a functor $h(f)$
on one large poset $\M$
(with Mayer--Vietoris systems arising as restrictions to
a subposet of $\M$).
Any point $u \in \M$ corresponds to a pair
$(I, C)$ as above
and a degree $n$ in such a way that $h(f)(u) = H_n (f^{-1} (I, C))$.
The natural symmetry of this parametrization is expressed by
an automorphism $T \colon \M \rightarrow \M$
such that given $h(f)(u) = H_n (f^{-1} (I, C))$ for some
$u \in \M$ we have $(h(f) \circ T)(u) = H_{n-1} (f^{-1} (I, C))$.
In other words, if $u \in \M$ corresponds to a pair $(I, C)$
and a degree $n$, then $T(u)$ corresponds to the pair
$(I, C)$ and the degree $n-1$.

\subsubsection{A Parametrization for Graded Relative Interlevel Sets}
\label{sec:parametrization}
Explicitly, $\M$ is given as the convex hull of two lines
$l_0$ and $l_1$ of slope $-1$ in $\R^{\circ} \times \R$
passing through $-\pi$ respectively $\pi$ on the $x$-axis,
as shown in \cref{fig:embeddingReals},
where $\R$ and $\R^{\circ}$ denote the posets given by
the orders $\leq$ and $\geq$ on $\R$, respectively.
This makes $\R^{\circ} \times \R$ the product poset
and $\M$ a subposet.

\begin{figure}[t]
  \centering
\begingroup%
  \makeatletter%
  \providecommand\color[2][]{%
    \errmessage{(Inkscape) Color is used for the text in Inkscape, but the package 'color.sty' is not loaded}%
    \renewcommand\color[2][]{}%
  }%
  \providecommand\transparent[1]{%
    \errmessage{(Inkscape) Transparency is used (non-zero) for the text in Inkscape, but the package 'transparent.sty' is not loaded}%
    \renewcommand\transparent[1]{}%
  }%
  \providecommand\rotatebox[2]{#2}%
  \newcommand*\fsize{\dimexpr\f@size pt\relax}%
  \newcommand*\lineheight[1]{\fontsize{\fsize}{#1\fsize}\selectfont}%
  \ifx\svgwidth\undefined%
    \setlength{\unitlength}{297.50001526bp}%
    \ifx\svgscale\undefined%
      \relax%
    \else%
      \setlength{\unitlength}{\unitlength * \real{\svgscale}}%
    \fi%
  \else%
    \setlength{\unitlength}{\svgwidth}%
  \fi%
  \global\let\svgwidth\undefined%
  \global\let\svgscale\undefined%
  \makeatother%
  \begin{picture}(1,0.42857141)%
    \lineheight{1}%
    \setlength\tabcolsep{0pt}%
    \put(0,0){\includegraphics[width=\unitlength,page=1]{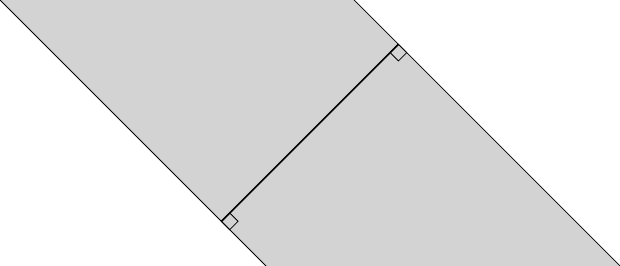}}%
    \put(0.40952368,0.27619057){\makebox(0,0)[t]{\lineheight{1.25}\smash{\begin{tabular}[t]{c}$\M$\end{tabular}}}}%
    \put(0.78571433,0.23809511){\makebox(0,0)[t]{\lineheight{1.25}\smash{\begin{tabular}[t]{c}$ \pi $\end{tabular}}}}%
    \put(0.2142857,0.1761905){\makebox(0,0)[t]{\lineheight{1.25}\smash{\begin{tabular}[t]{c}$-\pi~~$\end{tabular}}}}%
    \put(0,0){\includegraphics[width=\unitlength,page=2]{embedding-reals.pdf}}%
    \put(0.54761897,0.1761905){\makebox(0,0)[t]{\lineheight{1.25}\smash{\begin{tabular}[t]{c}$\op{Im} \blacktriangle$\end{tabular}}}}%
    \put(0.72924828,0.30884696){\makebox(0,0)[t]{\lineheight{1.25}\smash{\begin{tabular}[t]{c}$l_1$\end{tabular}}}}%
    \put(0.27945856,0.10149378){\makebox(0,0)[t]{\lineheight{1.25}\smash{\begin{tabular}[t]{c}$l_0$\end{tabular}}}}%
  \end{picture}%
\endgroup%

  \caption{
    The strip $\M$ and
    the image of the embedding
    $\blacktriangle \colon \overline{\R} \rightarrow \M$.
  }
  \label{fig:embeddingReals}
\end{figure}

The automorphism $T \colon \M \rightarrow \M$
has the following defining property
(also see \cref{fig:incidenceT}):
\begin{quote}
  Let $u \in \M$,
  $h_0$ be the horizontal line through $u$,
  let $g_0$ be the vertical line through $u$,
  let $h_1$ be the horizontal line through $T(u)$, and let
  $g_1$ be the vertical line through $T(u)$.
  Then the lines $l_0$, $h_0$, and $g_1$ intersect in a common point, and
  the same is true for the lines $l_1$, $g_0$, and $h_1$.  
\end{quote}
We also note that $T$ is a glide reflection along
the bisecting line between $l_0$ and $l_1$,
and the amount of translation is the distance
of $l_0$ and $l_1$.
Moreover, as a space, $\M / \langle T \rangle$ is a Möbius strip;
see also \cite{Carlsson:2009:ZPH:1542362.1542408}.

\begin{figure}[t]
  \centering
\begingroup%
  \makeatletter%
  \providecommand\color[2][]{%
    \errmessage{(Inkscape) Color is used for the text in Inkscape, but the package 'color.sty' is not loaded}%
    \renewcommand\color[2][]{}%
  }%
  \providecommand\transparent[1]{%
    \errmessage{(Inkscape) Transparency is used (non-zero) for the text in Inkscape, but the package 'transparent.sty' is not loaded}%
    \renewcommand\transparent[1]{}%
  }%
  \providecommand\rotatebox[2]{#2}%
  \newcommand*\fsize{\dimexpr\f@size pt\relax}%
  \newcommand*\lineheight[1]{\fontsize{\fsize}{#1\fsize}\selectfont}%
  \ifx\svgwidth\undefined%
    \setlength{\unitlength}{297.50001526bp}%
    \ifx\svgscale\undefined%
      \relax%
    \else%
      \setlength{\unitlength}{\unitlength * \real{\svgscale}}%
    \fi%
  \else%
    \setlength{\unitlength}{\svgwidth}%
  \fi%
  \global\let\svgwidth\undefined%
  \global\let\svgscale\undefined%
  \makeatother%
  \begin{picture}(1,0.42857141)%
    \lineheight{1}%
    \setlength\tabcolsep{0pt}%
    \put(0,0){\includegraphics[width=\unitlength,page=1]{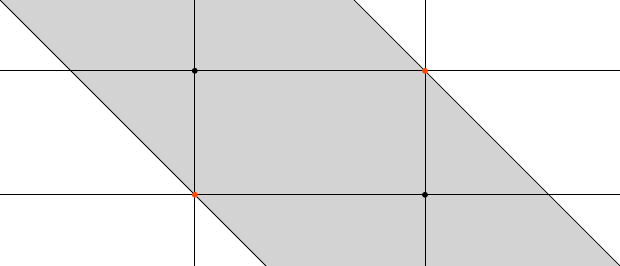}}%
    \put(0.34047628,0.23333343){\makebox(0,0)[t]{\lineheight{1.25}\smash{\begin{tabular}[t]{c}$g_1$\end{tabular}}}}%
    \put(0.47142855,0.32857133){\makebox(0,0)[t]{\lineheight{1.25}\smash{\begin{tabular}[t]{c}$h_1$\end{tabular}}}}%
    \put(0.66190484,0.18571436){\makebox(0,0)[t]{\lineheight{1.25}\smash{\begin{tabular}[t]{c}$g_0$\end{tabular}}}}%
    \put(0.51904762,0.08095235){\makebox(0,0)[t]{\lineheight{1.25}\smash{\begin{tabular}[t]{c}$h_0$\end{tabular}}}}%
    \put(0.27142864,0.32857133){\makebox(0,0)[t]{\lineheight{1.25}\smash{\begin{tabular}[t]{c}$T(u)$\end{tabular}}}}%
    \put(0.70714282,0.08333344){\makebox(0,0)[t]{\lineheight{1.25}\smash{\begin{tabular}[t]{c}$u$\end{tabular}}}}%
    \put(0.8235292,0.21456604){\makebox(0,0)[t]{\lineheight{1.25}\smash{\begin{tabular}[t]{c}$l_1$\end{tabular}}}}%
    \put(0.18517764,0.1957747){\makebox(0,0)[t]{\lineheight{1.25}\smash{\begin{tabular}[t]{c}$l_0$\end{tabular}}}}%
  \end{picture}%
\endgroup%

  \caption{Incidences defining $T$.}
  \label{fig:incidenceT}
\end{figure}

Now in order to specify a degree for each point in $\M$,
it suffices to specify all points corresponding to degree $0$.
More specifically, we will now specify a fundamental domain~$D$
with respect to the action of $\langle T \rangle$,
which consists of all points corresponding to degree~$0$.
To this end,
we embed the extended reals $\overline{\R}$ into the strip $\M$
by precomposing the diagonal map $\Delta \colon {\R} \to {\R}^2,
t \mapsto (t,t)$
with the homeomorphism
$\arctan: \overline{\R} \to [-\pi/2,\pi/2]$,
yielding a map
\[{\blacktriangle = \Delta \circ \arctan \colon \overline{\R} \rightarrow \M, ~
t \mapsto (\arctan t, \arctan t)}\]
such that $\op{Im} \blacktriangle$ is a perpendicular
line segment through the origin joining $l_0$ and $l_1$,
see \cref{fig:embeddingReals}.
With this we define our fundamental domain as
\[
  D := \,
  (\uparrow \op{Im} \blacktriangle) \setminus
  T (\uparrow \op{Im} \blacktriangle),
\]
see \cref{fig:defFundamentalDomain}.
Here $\uparrow \op{Im} \blacktriangle$ is the upset
of the image of $\blacktriangle$.

\begin{figure}[t]
  \centering
\begingroup%
  \makeatletter%
  \providecommand\color[2][]{%
    \errmessage{(Inkscape) Color is used for the text in Inkscape, but the package 'color.sty' is not loaded}%
    \renewcommand\color[2][]{}%
  }%
  \providecommand\transparent[1]{%
    \errmessage{(Inkscape) Transparency is used (non-zero) for the text in Inkscape, but the package 'transparent.sty' is not loaded}%
    \renewcommand\transparent[1]{}%
  }%
  \providecommand\rotatebox[2]{#2}%
  \newcommand*\fsize{\dimexpr\f@size pt\relax}%
  \newcommand*\lineheight[1]{\fontsize{\fsize}{#1\fsize}\selectfont}%
  \ifx\svgwidth\undefined%
    \setlength{\unitlength}{240.83333588bp}%
    \ifx\svgscale\undefined%
      \relax%
    \else%
      \setlength{\unitlength}{\unitlength * \real{\svgscale}}%
    \fi%
  \else%
    \setlength{\unitlength}{\svgwidth}%
  \fi%
  \global\let\svgwidth\undefined%
  \global\let\svgscale\undefined%
  \makeatother%
  \begin{picture}(1,0.52941176)%
    \lineheight{1}%
    \setlength\tabcolsep{0pt}%
    \put(0,0){\includegraphics[width=\unitlength,page=1]{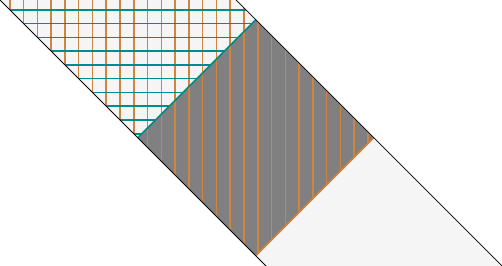}}%
    \put(0.87484017,0.16280688){\makebox(0,0)[t]{\lineheight{1.25}\smash{\begin{tabular}[t]{c}$l_1$\end{tabular}}}}%
    \put(0.36426394,0.11808903){\makebox(0,0)[t]{\lineheight{1.25}\smash{\begin{tabular}[t]{c}$l_0$\end{tabular}}}}%
  \end{picture}%
\endgroup%

  \caption{
    The fundametal domain
    $\textcolor{DimGrey}{D} :=
    \textcolor{Peru}{\uparrow \op{Im} \blacktriangle} \setminus
    \textcolor{DarkCyan}{T(\uparrow \op{Im} \blacktriangle)}$.
  }
  \label{fig:defFundamentalDomain}
\end{figure}

Now $D$ provides a tessellation of $\M$
as shown in \cref{fig:tessellation},
and we assign the degree $n$ to any point in $T^{-n} (D)$.
This convention for $T$ is chosen in analogy to the topological suspension, which also decreases the homological degree: a homology class of degree $d$ in the suspension
corresponds to a homology class of degree $d-1$
in the original space.

\begin{figure}[t]
  \centering
\begingroup%
  \makeatletter%
  \providecommand\color[2][]{%
    \errmessage{(Inkscape) Color is used for the text in Inkscape, but the package 'color.sty' is not loaded}%
    \renewcommand\color[2][]{}%
  }%
  \providecommand\transparent[1]{%
    \errmessage{(Inkscape) Transparency is used (non-zero) for the text in Inkscape, but the package 'transparent.sty' is not loaded}%
    \renewcommand\transparent[1]{}%
  }%
  \providecommand\rotatebox[2]{#2}%
  \newcommand*\fsize{\dimexpr\f@size pt\relax}%
  \newcommand*\lineheight[1]{\fontsize{\fsize}{#1\fsize}\selectfont}%
  \ifx\svgwidth\undefined%
    \setlength{\unitlength}{191.25bp}%
    \ifx\svgscale\undefined%
      \relax%
    \else%
      \setlength{\unitlength}{\unitlength * \real{\svgscale}}%
    \fi%
  \else%
    \setlength{\unitlength}{\svgwidth}%
  \fi%
  \global\let\svgwidth\undefined%
  \global\let\svgscale\undefined%
  \makeatother%
  \begin{picture}(1,0.66666667)%
    \lineheight{1}%
    \setlength\tabcolsep{0pt}%
    \put(0,0){\includegraphics[width=\unitlength,page=1]{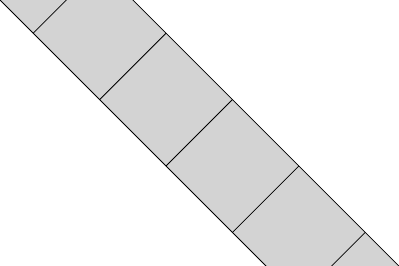}}%
    \put(0.75,0.08333333){\makebox(0,0)[t]{\lineheight{1.25}\smash{\begin{tabular}[t]{c}$\Sigma^{-2}(D)$\end{tabular}}}}%
    \put(0.58333333,0.25){\makebox(0,0)[t]{\lineheight{1.25}\smash{\begin{tabular}[t]{c}$\Sigma^{-1}(D)$\end{tabular}}}}%
    \put(0.25,0.58333333){\makebox(0,0)[t]{\lineheight{1.25}\smash{\begin{tabular}[t]{c}$\Sigma(D)$\end{tabular}}}}%
    \put(0.41666667,0.41666667){\makebox(0,0)[t]{\lineheight{1.25}\smash{\begin{tabular}[t]{c}$D$\end{tabular}}}}%
    \put(0.68190079,0.35809922){\makebox(0,0)[t]{\lineheight{1.25}\smash{\begin{tabular}[t]{c}$l_1$\end{tabular}}}}%
    \put(0.30524236,0.31142431){\makebox(0,0)[t]{\lineheight{1.25}\smash{\begin{tabular}[t]{c}$l_0$\end{tabular}}}}%
  \end{picture}%
\endgroup%

  \caption{The tessellation of $\M$ induced by $T$ and $D$.}
  \label{fig:tessellation}
\end{figure}

It remains to assign a pair $(I, C)$
to any point $u \in D$ of the fundamental domain.
The following proposition provides such an assignment;
a schematic image is shown in \cref{fig:rho}.

\begin{prp}
  \label{prp:rho}
  Let $\mathcal{P}$ denote the set of pairs of closed subspaces
  of $\overline{\R}$.
  Then there is a unique order-preserving map
  \begin{equation*}
    \rho \colon D \rightarrow \mathcal{P}
  \end{equation*}
  with the following three properties:
  \begin{enumerate}
  \item
    For any $t \in \R$ we have
    $(\rho \circ \blacktriangle)(t) = (\{t\}, \emptyset)$.
  \item
    For any $u \in D \cap \partial \M$
    the two components of $\rho(u)$ are identical.
  \item
    For any axis-aligned rectangle contained in $D$
    the corresponding joins and meets are preserved by $\rho$.
  \end{enumerate}
\end{prp}

\begin{figure}[t]
  \centering
  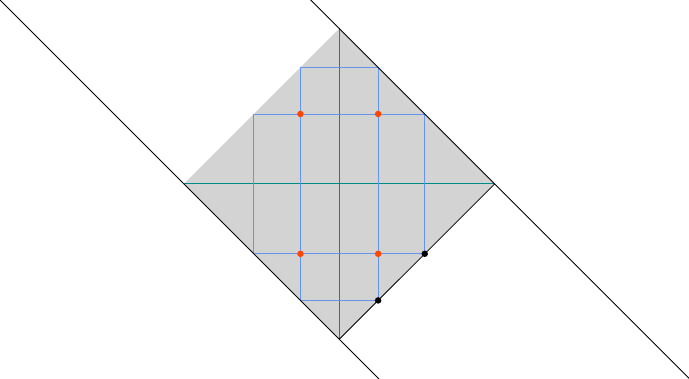
  \caption{A schematic image of $\rho$.}
  \label{fig:rho}
\end{figure}

For $u \in D$ we may describe $(I, C) := \rho(u)$ more explicitly as follows.
The interval $I$ is given by taking the downset
of $u$ in the poset $\M$,
denoted by $\downarrow u$,
and taking the preimage under the embedding
$\blacktriangle \colon \overline{\R} \rightarrow \M$.
Similarly, $C$ is given by starting with the transformed point $T^{-1} (u)$,
taking its upset, forming the complement of this upset in $\M$,
and taking the preimage of the closure under $\blacktriangle$.
Thus, we have the formula
\begin{equation*}
  \rho(u) =
  \blacktriangle^{-1} \big(
  \downarrow u, \, \overline{\M \setminus \uparrow T^{-1} (u)} \,
  \big)
  .
\end{equation*}
We also note that, if $I$ is a bounded interval, then $C$ must be empty; if $I$ is a proper downset (upset), then so is $C$; and if $\overline{\R} \setminus C \subseteq \R$, then $I$ must be $\overline{\R}$.
Any other point in the orbit of the point $u$ is assigned the same pair,
but a different degree.

\subsubsection{Assembling the Relative Interlevel Set Homology Functor}
Now suppose $f \colon X \rightarrow \R$ is a piecewise linear function
with $X$ a finite simplicial complex.
For any $n \in \Z$, we obtain a functor
\begin{equation}
  \label{eq:nthLayer}
  D \rightarrow \mathrm{vect}_K,
  ~
  u \mapsto H_n (f^{-1}(\rho(u)))
  .
\end{equation}
Here $\mathrm{vect}_K$ denotes the category of finite-dimensional
vector spaces over $K$.
This functor describes the $n$-th layer
of the Mayer--Vietoris pyramid of $f$, as introduced in
\cite{Carlsson:2009:ZPH:1542362.1542408} and extensively studied in
\cite{Bendich-2013} and \cite{MR3924175}.
As pointed out in \cite{Carlsson:2009:ZPH:1542362.1542408},
the different layers can be glued
together to form one large diagram.
More specifically, we can move the functor in \eqref{eq:nthLayer}
down by $n$ tiles in the tessellation of $\M$ shown in
\cref{fig:tessellation} by precomposition with
$T^{n}$:
\begin{equation}
  \label{eq:movedLayer}
  T^{-n}(D) \rightarrow \mathrm{vect}_K,
  ~
  u \mapsto H_n (f^{-1}(\rho( T^{n}(u)))
  .
\end{equation}
This way we obtain a single functor on each tile $T^{-n} (D)$.
We can further extend these functors into one single functor
\begin{equation*}
  h(f) := h (f; K) \colon \M \rightarrow \mathrm{vect}_K
  .
\end{equation*}
We will refer to this functor as the
\emph{relative interlevel set homology of $f$
  with coefficients in $K$}.
For better readability we suppress the field $K$ as an argument of $h$.
We define the restriction $h (f) |_{T^{-n} (D)}$
as the functor specified in \eqref{eq:movedLayer}.

It remains to specify
the linear maps in $h (f)$ between comparable elements
from different tiles.
In the following we will define these maps 
as either the zero map or as the boundary operator
of a Mayer--Vietoris sequence.
To this end, let $u \preceq v \in \M$, with $u$ and $v$ lying in different tiles, and consider the interval $[u, v]$ in the poset $\M \subseteq \R^{\circ} \times \R$, which is the intersection of a closed axis-aligned rectangle with the closed strip $\M$.
If there is a point $x \in [u, v] \cap \partial \M$,
then $h (f)(x) \cong \{0\}$,
since $\rho$ assigns to any point in $D \cap \partial \M$
a pair of identical intervals.
In this case, we thus have to set
$h (f)(u \preceq v) = 0$,
as this map factors through $\{0\}$.

Now consider the case $[u, v] \cap \partial \M = \emptyset$,
and let $w := T(u)$.
As illustrated by \cref{fig:squareRects},
we have $v \preceq w$,
and thus $u$ and $v$ have to lie in adjacent tiles.
For simplicity, we first consider the case $v \in D$,
which implies $u \in T^{-1} (D)$ and $w \in D$.
In particular, the poset interval $[v, w]$ is an axis aligned rectangle with corners $v=(v_1, v_2)$, $w=(w_1, w_2)$, $(w_1, v_2)$, and $(v_1, w_2)$.
We have the meet ${v=(w_1, v_2)\wedge (v_1, w_2)}$
and the join ${w = (w_1, v_2) \vee (v_1, w_2)}$.
Since this rectangle is contained in $D$, join and meet 
are preserved by $\rho$ by \cref{prp:rho}.3.
Moreover, since taking preimages
$f^{-1} \colon 2^{\overline{\R}} \rightarrow 2^X$
is a homomorphism of Boolean algebras,
$f^{-1}$ also preserves joins and meets, which in this case are the componentwise unions and intersections.
Writing $F = f^{-1} \circ \rho$,
we get a Mayer--Vietoris sequence
\begin{equation*}
\dots
\to
\underbrace
{H_1 (F(w))}_
{h (f)(u)}
\stackrel{\partial}\to 
\underbrace
{H_0 (F(v))}
_
{h (f)(v)}
\to
H_0 (F(w_1, v_2))
\oplus
H_0 (F(v_1, w_2))
\to
H_0 (F(w))
\to
0
\end{equation*}
and define
$h (f)(u \preceq v) := \partial .$

\begin{figure}[t]
  \centering
  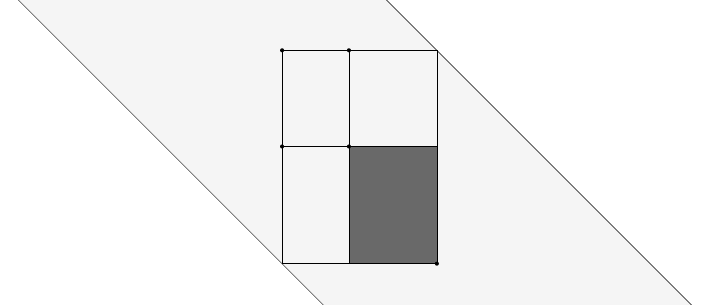
  \caption{The interval $[u, v]$ and $w = T(u)$.}
  \label{fig:squareRects}
\end{figure}

We note that the existence of the
Mayer--Vietoris sequence above is ensured by
a very general criterion, given in
\cite[Theorem 10.7.7]{MR2456045}.
This criterion requires certain triads to be excisive, which is satisfied here because
$f$ is piecewise linear.
In any subsequent applications of
the Mayer--Vietoris sequence,
we will continue to use this criterion, omitting the straightforward proof that the corresponding triads are excisive.

In the general situation where $v \in T^{-n} (D)$ for some $n \in \Z$,
the points ${v' := T^{n}(v)}$ and $w' := T^{n+1}(u)$
lie in $D$.
Using the above arguments with
$v'$ and $w'$
in place of $v$ respectively $w$,
we obtain the Mayer--Vietoris sequence
\begin{equation*}
\dots
\to
\underbrace
{H_{n+1} (F(w'))}_
{h (f)(u)}
\stackrel{\tilde\partial}\to 
\underbrace
{H_n(F(v'))}
_
{h (f)(v)}
\to
H_n (F(w'_1, v'_2))
\oplus
H_n (F(v'_1, w'_2))
\to
H_n (F(w'))
\to
\dots
\end{equation*}
and define
\(h (f)(u \preceq v) := \tilde\partial .\)

We refer to $h(f)$ as the
\emph{relative interlevel set homology of $f$
  with coefficients in $K$}.
A generalization of the construction
of $h(f)$ to pairs $(X,A)$ and a proof of its functoriality can be found in
\cref{sec:constr_pers}.

\subsection{The Extended Persistence Diagram}
\label{sec:persDiagram}

Having defined the relative interlevel set homology as a functor
\(h(f) \colon \M \rightarrow \mathrm{vect}_K\),
we now formalize the notion of an \emph{extended persistence diagram},
originally due to \cite{Cohen-Steiner2009},
as an invariant of functors
$F \colon \M \rightarrow \mathrm{vect}_K$
vanishing on $\partial \M$.
The \emph{persistence diagram} of $F$ is a multiset
$\op{Dgm} (F) = \mu \colon \op{int} \M \rightarrow \N_0$, which counts,
for each point $v = (v_1, v_2) \in \M$, the maximal number $\mu(v)$
of linearly independent vectors in $F(v)$ born at $v$;
for the functor $h(f)$, these are homology classes.
More precisely, we define
\begin{equation*}
  \op{Dgm} (F) : \op{int} \M \rightarrow \N_0,
  \,
  v \mapsto
  \dim_K F(v) -
  \dim_K \sum_{u \prec v} \op{Im} F(u \preceq v)
  .
\end{equation*}
In the last term, $u$ ranges over all $u \in \M$ with $u \prec v$.
Moreover, note that
\begin{equation*}
  \sum_{u \prec v} \op{Im} F(u \preceq v)
  =
  \left(\bigcup_{x > v_1} \op{Im} F((x, v_2) \preceq v)\right) +
  \left(\bigcup_{y < v_2} \op{Im} F((v_1, y) \preceq v)\right).
\end{equation*}

Now let $f \colon X \rightarrow \R$ be a piecewise linear function
with $X$ a finite simplicial complex.

\begin{definition}[Extended Persistence Diagram]
  \label{dfn:diagramFunction}
  The \emph{extended persistence diagram of $f$ (over $K$)} is
  $\op{Dgm}(f) := \op{Dgm}(f; K) := \op{Dgm}(h(f)).$
\end{definition}

\sloppy
See \cref{fig:booklet} for an example.
In \cref{sec:struct} we show that the restriction
of $\op{Dgm}(f) \colon \op{int} \M \rightarrow \N_0$
to any of the regions shown in \cref{fig:subdiagrams}
yields the corresponding ordinary, relative, or extended subdiagram
as defined in \cite{Cohen-Steiner2009}
up to reparametrization.
Moreover, we note that $\op{Dgm}(f)$ is supported in the downset
${\downarrow \op{Im} \blacktriangle \subseteq \M}$;
in \cref{fig:subdiagrams} this region is shaded in dark gray.
As $f$ is bounded, $\op{Dgm}(f)$ is also supported in the union of open squares
$\left(-\frac{\pi}{2}, \frac{\pi}{2}\right)^2 + \pi \Z^2$.

\begin{figure}[t]
  \centering
  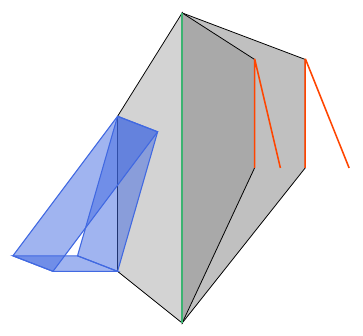
  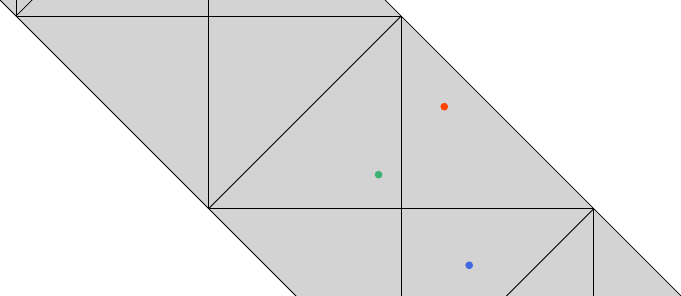
  \caption{
    A simplicial complex (top) and the extended persistence diagram
    (bottom)
    of the associated height function.
    The numbers next to the dots indicate their multiplicities
    in the persistence diagram.
    The complex shown on top is an instance of the construction
    we use to realize extended persistence diagrams
    in the proof of \cref{thm:suffCond}.
    Subcomplexes shaded in color correspond to dots in the persistence diagram.
    The rectangular regions show the support
    of the corresponding indecomposables in
    relative interlevel set homology, generically spanning two consecutive degrees in homology.
  }
  \label{fig:booklet}
\end{figure}

We now define the \emph{bottleneck distance} of extended persistence diagrams.
To this end, we first define an extended metric
$d \colon \M \times \M \rightarrow [0, \infty]$ on~$\M$.
Following the approach of \cite{MR4099800},
we then define an extended metric on multisets of $\op{int} \M$ in terms of~$d$.
Let $d_0 \colon \R \times \R \rightarrow [0, \infty]$
be the unique extended metric on $\R$ such that for all $s < t$
we have the equation
\begin{equation*}
  d_0 (s, t) =
  \begin{cases}
    \tan t - \tan s &
    [s, t] \cap \left(\frac{\pi}{2} + \pi \Z\right) = \emptyset,
    \\
    \infty & \text{otherwise}.
  \end{cases}
\end{equation*}

\begin{definition}
  \label{dfn:d}
  Writing $u = (u_1, u_2)$ and $v = (v_1, v_2)$
  for points $u, v \in \M$ we define
  \begin{equation*}
    d \colon \M \times \M \rightarrow [0, \infty], \,
    (u, v) \mapsto \max \{d_0(u_1, v_1) , d_0(u_2, v_2)\}
  \end{equation*}
  to be the extended metric on $\M$
  given by the maximum 
  of $d_0$ on each copy of $\R$.
\end{definition}

Using the extended metric
$d \colon \M \times \M \rightarrow [0, \infty]$
we can now express how a perturbation of a function
$f \colon X \rightarrow \R$ as above
may affect its persistence diagram.
If we think of the vertices of $\op{Dgm}(f)$
as \enquote{features} of $f \colon X \rightarrow \R$,
then a $\delta$-perturbation of $f \colon X \rightarrow \R$
may cause the corresponding vertices of the persistence diagram
to be moved by up to a distance of $\delta$ away from their original position
with respect to $d \colon \M \times \M \rightarrow [0, \infty]$.
As the persistence diagram
$\op{Dgm}(f) \colon \op{int} \M \rightarrow \N_0$
is undefined at the boundary $\partial \M$,
vertices that are $\delta$-close to the boundary $\partial \M$
may disappear altogether and moreover,
new vertices within a distance of $\delta$ from $\partial \M$ may appear
in the persistence diagram of the perturbation.
This intuition is made precise
by the \cite[Stability Theorem]{Cohen-Steiner2009}
for the bottleneck distance of extended persistence diagrams.
Completely analogously,
when $\U \subseteq \M$ is an admissible upset of $\M$
and when we merely compute $\op{Dgm}(f) |_{\op{int} \U}$,
then perturbations may cause vertices to disappear in $\partial \U$
and other vertices to appear in close proximity to $\partial \U$.
We follow \cite{MR4099800} to provide a combinatorial description
of the bottleneck distance.

\begin{definition}[Matchings relative to the boundary and the Bottleneck Distance]
  \label{dfn:bottleneckDist}
  \sloppy
  Let $\U \subseteq \M$ be an admissible upset
  and let $\mu, \nu \colon \op{int} \U \rightarrow \N_0$
  be finite multisets.
  A \emph{matching (relative to the boundary) between $\mu$ and $\nu$} is
  a multiset \enquote{of pairs}
  ${M \colon \U \times \U \rightarrow \N_0}$
  such that
  \begin{equation*}
    \op{pr}_1 (M) |_{\op{int} \U} = \mu, \quad
    \op{pr}_2 (M) |_{\op{int} \U} = \nu, ~
    \text{and} \quad
    M |_{
      \partial \U \times \partial \U
    } \equiv 0
    ,
  \end{equation*}
  where $\op{pr}_1 (M), \op{pr}_2 (M) \colon \U \rightarrow \N_0$
  are the projections of ${M \colon \U \times \U \rightarrow \N_0}$
  to the respective components, 
  $\op{pr}_1 (M) \colon u \mapsto \sum_{v \in \U} M(u, v)$
  and
  $\op{pr}_2 (M) \colon v \mapsto \sum_{u \in \U} M(u, v)$.
  In this paper, we tacitly assume that all matchings are relative to the boundary.
  The \emph{norm of a matching $M$} is defined as
  \begin{equation*}
    \lVert M \rVert := \sup d \left(M^{-1} (\N \setminus \{0\})\right)
  \end{equation*}
  and the
  \emph{bottleneck distance of $\mu$ and $\nu$} is
  \begin{equation*}
    d_B (\mu, \nu) :=
    \inf
    \{
    \lVert M \rVert
    \mid
    \text{$M$ is a matching between $\mu$ and $\nu$}
    \}
    .
  \end{equation*}
\end{definition}

\begin{remark}
\label{remark:matchings4universality}
We point out an important difference of this definition of bottleneck distance to the one commonly found in the literature on extended persistence.
Consider the regions in \cref{fig:subdiagrams}
in conjunction with the extended metric
${d \colon \M \times \M \rightarrow [0, \infty]}$
defined in \cref{dfn:d} above.

Assume that we either consider the unrestricted persistence diagram (with admissible upset $\U = \M$), or restrict to a region $\U$ such that the regions corresponding to any extended subdiagram is either completely contained in $\U$ or disjoint from $\U$.
While all interior points of regions
corresponding to ordinary and relative subdiagrams
have finite distance to the boundary $\partial \M$,
all of the interior points of regions
corresponding to extended subdiagrams have distance $\infty$ to $\partial \U$.
As a result, any matching of finite norm between extended persistence diagrams
has to match any vertex contained in the extended subdiagram
to a vertex of the other extended persistence diagram.
In contrast, most of the literature concerned with extended or level set persistence \cite{Cohen-Steiner2009,Carlsson:2009:ZPH:1542362.1542408,MR3924175} either explicitly admits points in the extended subdiagram to be matched to the diagonal, or is somewhat ambiguous about this aspect.
These phantom matchings do not affect stability,
as the resulting bottleneck distance is only getting smaller.
However, not allowing these phantom matchings is crucial
for the universality result \cref{thm:suffCond} to hold.

On the other hand, if the admissible upset $\U$ contains only part of some extended subdiagram, and hence the diagonal of the extended subdiagram is contained in $\partial \U$, our definition does allow points in this extended subdiagram to be matched to the diagonal, and this is even necessary to obtain stability.
It is notable that we obtain universality in all scenarios with a single general theorem.

\end{remark}

\section{Properties of Relative Interlevel Set Homology}
\label{sec:props}

We think of
the relative interlevel set
homology of a piecewise linear function $f \colon X \rightarrow \R$
(with $X$ a finite complex)
as an analogue
to the homology of a space.
There are various tools to reduce the computation of the homology of a space
to smaller subcomputations,
and one of them is the relative homology of a pair of spaces.
In order to compute the relative interlevel set
homology of a function, we develop a counterpart to
relative homology in \cref{sec:constr_pers}.
More specifically,
given a finite simplicial pair $(X, A)$ and a piecewise linear function
$f \colon X \rightarrow \R$
we construct a functor
\[h(X, A; f) := h(X, A; f; K) \colon \M \rightarrow \mathrm{vect}_K,\]
which we refer to as the
\emph{relative interlevel set homology of the function
  ${f \colon A \subseteq X \to \R}$}.
Similarly, we write
\({\op{Dgm}(X, A; f) := \op{Dgm}(h(X, A; f))}\)
for the corresponding persistence diagram.
In the case ${A = \emptyset}$, we also write ${h(X; f) := h(X, \emptyset; f)}$ and \({\op{Dgm}(X; f) := \op{Dgm}(X, \emptyset; f)}\).
When the pair $(X, A)$ is clear from the context, we may suppress it
as an argument of $h$ or $\op{Dgm}$ and simply write $h(f)$
and $\op{Dgm}(f)$.

One of the most useful properties of relative homology
is functoriality.
Before we continue, we explain in what sense
$h = h(-; K)$ is a functor.
To this end, we describe the category of functions $f \colon A \subseteq X \to \R$ that provides the natural domain for this functor.
Denoting this category by $\mathcal{F}_0$,
its objects are triples $(X, A; f)$
with $(X, A)$ a finite simplicial pair and
$f \colon X \rightarrow \R$ a piecewise linear function.
Now let $(X, A; f)$ and $(Y, B; g)$ be objects of $\mathcal{F}_0$.
A \emph{morphism $\varphi$ from $(X, A; f)$ to $(Y, B; g)$}
is a continuous map $\varphi \colon X \rightarrow Y$
with $\varphi(A) \subseteq B$ and the property that the diagram
\begin{equation}
  \label{eq:commHom}
  \begin{tikzcd}
    X
    \arrow[rr, bend left, "\varphi"]
    \arrow[dr, "f"']
    &
    &
    Y
    \arrow[dl, "g"]
    \\
    &
    \R
  \end{tikzcd}
\end{equation}
commutes.
We may also write this as
$\varphi \colon (X, A; f) \rightarrow (Y, B; g)$.
The subscript $0$ in the notation $\mathcal{F}_0$ indicates the equivalent condition that $f$ and $g \circ \phi$ have distance $0$ in the supremum norm.
The composition and the identities of $\mathcal{F}_0$
are defined in the obvious way.
Moreover, as homology is a functor on topological spaces
and $h = h(-; K)$ is defined in terms of homology
(see \cref{sec:constr_pers}),
this makes $h$ a functor from $\mathcal{F}_0$
to the category \(\mathrm{vect}_K^{\M}\)
of functors from $\M$ to $\mathrm{vect}_K$.
As it turns out, this relative interlevel set homology functor
$h \colon \mathcal{F}_0 \rightarrow \mathrm{vect}_K^{\M}$
satisfies certain characteristic properties analogous to the Eilenberg--Steenrod axioms
for homology theories of topological spaces.

\sloppy
For the first property, let $(X, A; f)$ and $(Y, B; g)$ be objects of $\mathcal{F}_0$,
and let 
$\varphi, \psi \colon (X, A; f) \rightarrow (Y, B; g)$
be morphisms.
Then $(X \times [0, 1], A \times [0, 1]; f \circ \pr_1)$
is an object of $\mathcal{F}_0$, and
we define a \emph{fiberwise homotopy $\eta$ from $\varphi$ to $\psi$}
to be a morphism
\[\eta \colon (X \times [0, 1], A \times [0, 1]; f \circ \pr_1)
\rightarrow (Y, B; g)\] such that
$\eta(-, 0) = \varphi$ and
$\eta(-, 1) = \psi$.
If such a homotopy exists, we say that $\varphi$ and $\psi$
\emph{are fiberwise homotopic}.

\begin{lem}[Homotopy Invariance]
  \label{lem:homotopyInv}
  If $\varphi$ and $\psi$ are fiberwise homotopic,
  then ${h(\varphi) \cong h(\psi)}$.
\end{lem}

\begin{proof}
  As $h = h(-; K)$ is defined in terms of
  homology,
  this follows from the homotopy invariance of homology.
\end{proof}

\begin{restatable}[Mayer--Vietoris Sequence]{lem}{lemMVseq}
  \label{lem:MVseq}
  Let $X$ be a finite simplicial complex with
  subcomplexes $A_0 \subseteq X_0 \subseteq X$, $A_1 \subseteq X_1 \subseteq X$, and $A \subseteq X$,
  such that
  \begin{equation*}
    X = X_0 \cup X_1
    \quad \text{and} \quad
    A = A_0 \cup A_1
    .
  \end{equation*}
  Moreover, let $f \colon X \rightarrow \R$
  be a piecewise linear map.
  Then there is an exact sequence of functors $\M \to \mathrm{vect}_K$ given by
  \begin{equation}
    \label{eq:mv_seq}
    \begin{tikzcd}
      h(X_0 \cap X_1, A_0 \cap A_1; f |_{X_0 \cap X_1})
      \arrow[d, "\small\begin{pmatrix} 1 \\ 1 \end{pmatrix}"]
      \\[2em]
      h(X_0, A_0; f |_{X_0})
      \oplus
      h(X_1, A_1; f |_{X_1})
      \arrow[d, "(1 ~ -1)"]
      \\
      h(X, A; f)
      \arrow[d, "\partial"]
      \\
      h(X_0 \cap X_1, A_0 \cap A_1; f |_{X_0 \cap X_1}) \circ T
      \arrow[d, "\small\begin{pmatrix} 1 \\ 1 \end{pmatrix}"]
      \\[2em]
      (h(X_0, A_0; f |_{X_0}) \circ T)
      \oplus
      (h(X_1, A_1; f |_{X_1}) \circ T)
      ,
    \end{tikzcd}
  \end{equation}
  where $1$ denotes a map induced by inclusion.
\end{restatable}
  We name this the
  \emph{Mayer--Vietoris sequence for %
    $(A; A_0, A_1) \subseteq (X; X_0, X_1)$
    relative to~$f$}, following the notation used in \cite[Theorem 10.7.7]{MR2456045}.
    
\begin{proof}
  Let $u \in D$ and $n \in \Z$.
  By \cref{lem:smashcap_dist}, the operation
  $- \smashcap f^{-1}(\rho(u))$ preserves
  componentwise unions and intersections.
  Thus, $(X_0 \cap X_1, A_0 \cap A_1) \smashcap f^{-1} (\rho(u))$
  is the componentwise intersection
  of $(X_0, A_1) \smashcap f^{-1} (\rho(u))$
  and $(X_1, A_1) \smashcap f^{-1} (\rho(u))$,
  while $(X, A) \smashcap f^{-1} (\rho(u))$ is their union.
  We obtain the sequence
  \begin{equation*}
    \begin{tikzcd}
      h(X_0 \cap X_1, A_0 \cap A_1; f |_{X_0 \cap X_1})\left(T^{n}(u)\right)
      \arrow[d, "\begin{pmatrix} 1 \\ 1 \end{pmatrix}"]
      \\[2em]
      h(X_0, A_0; f |_{X_0})\left(T^{n}(u)\right)
      \oplus
      h(X_1, A_1; f |_{X_1})\left(T^{n}(u)\right)
      \arrow[d, "(1 ~ -1)"]
      \\
      h(X, A; f)\left(T^{n}(u)\right)
      \arrow[d, "(-1)^n \partial_{T^n(u)}"]
      \\
      h(X_0 \cap X_1, A_0 \cap A_1; f |_{X_0 \cap X_1})\left(T^{n+1}(u)\right)
      \arrow[d, "\begin{pmatrix} 1 \\ 1 \end{pmatrix}"]
      \\[2em]
      h(X_0, A_0; f |_{X_0})\left(T^{n+1}(u)\right)
      \oplus
      h(X_1, A_1; f |_{X_1})\left(T^{n+1}(u)\right)
    \end{tikzcd}    
  \end{equation*}
  as a portion from the corresponding Mayer--Vietoris sequence.
  As we will see,
  the sign change of the boundary operator
  is needed for naturality.
  Since $D$ is a fundamental domain,
  we get \eqref{eq:mv_seq} pointwise, at each index of $\M$.
  Moreover, the maps denoted as
  \begin{equation*}
    \begin{pmatrix}
      1 & -1
    \end{pmatrix}
    \quad \text{and} \quad
    \begin{pmatrix}
      1 \\ 1
    \end{pmatrix}    
  \end{equation*}
  are natural transformations by the functoriality of
  homology and the naturality of the boundary operator
  of the Mayer--Vietoris sequence.
  
  It remains to be shown that $\partial$ is a natural transformation.
  In part, this follows from the naturality of the above Mayer--Vietoris
  sequence.
  However, recall from our construction of $h$ in \cref{sec:constr_pers}
  that some of the internal maps of
  $h(X, A; f)$ and
  ${h(X_0 \cap X_1, A_0 \cap A_1; f |_{X_0 \cap X_1})}$
  are boundary operators as well.
  Therefore, we need to check whether certain squares with all maps
  boundary operators commute.
  Specifically, given $T^n (v) \preceq T^{n+1} (u)$ for some
  $u = (u_1, u_2) \preceq v = (v_1, v_2) \in D$ and $n \in \Z$,
  we have to show that the diagram
  \begin{equation*}
    \begin{tikzcd}[row sep=35pt, column sep=68pt]
      h(X, A; f)\left(T^{n+1}(u)\right)
      \arrow[r, "(-1)^{n+1} \partial_{T^{n+1}(u)}"]
      &
      h \left(
        X_0 \cap X_1, A_0 \cap A_1; f |_{X_0 \cap X_1}
      \right)\left(T^{n+2}(u)\right)
      \\
      h(X, A; f)\left(T^{n}(v)\right)
      \arrow[r, "(-1)^n \partial_{T^{n}(v)}"']
      \arrow[u]%
      &
      h \left(
        X_0 \cap X_1, A_0 \cap A_1; f |_{X_0 \cap X_1}
      \right)\left(T^{n+1}(v)\right)
      \arrow[u]%
    \end{tikzcd}
  \end{equation*}
  commutes.
  By unraveling the definition of the relative interlevel set homology $h$
  we may rewrite this square to a more concrete form
  with all maps boundary operators of some Mayer--Vietoris sequence:
  \begin{equation}
    \label{eq:DsquareConcrete}
    \begin{tikzcd}[row sep=25pt, column sep=55pt]
      H_{n-1}
      \left((X, A) \smashcap f^{-1} (\rho(u))\right)
      \arrow[r, "(-1)^{n-1} \partial_{n-1}"]
      &
      H_{n-2}
      \left((X_0 \cap X_1, A_0 \cap A_1) \smashcap f^{-1} (\rho(u))\right)
      \\
      H_{n}
      \left((X, A) \smashcap f^{-1} (\rho(v))\right)
      \arrow[u]
      \arrow[r, "(-1)^{n} \partial_n"']
      &
      H_{n-1}
      \left((X_0 \cap X_1, A_0 \cap A_1) \smashcap f^{-1} (\rho(v))\right)
      .
      \arrow[u]
    \end{tikzcd}
  \end{equation}
  As $f$ is piecewise linear, there are subdivisions of
  $\overline{\R}$ and $X$ such that both components of
  $\rho(v_1, u_2)$ and $\rho(u_1, v_2)$ are subcomplexes of
  $\overline{\R}$ and such that
  $f \colon X \rightarrow \overline{\R}$ is a simplicial map.
  Thus, we may use the isomorphism from simplicial homology
  to singular homology to show
  that the square \eqref{eq:DsquareConcrete} commutes.
  As it turns out,
  the boundary operator for the Mayer--Vietoris sequence
  in simplicial homology,
  which is defined in terms of the zig-zag lemma,
  and the boundary operator from \cite[Theorem 10.7.7]{MR2456045},
  which is defined in terms of the long exact sequence of a triple
  and the suspension isomorphism,
  commute with the corresponding isomorphisms
  from simplicial to singular homology
  of domain and codomain.
  Thus,
  we may think of each of the arrows in \eqref{eq:DsquareConcrete}
  as boundary operators of a Mayer--Vietoris sequence
  in simplicial homology.
  Now in order to show the commutativity of \eqref{eq:DsquareConcrete},
  we consider the commutative diagram
    \begin{equation}\tiny
      \label{eq:nineComp}
      \begin{tikzcd}[row sep=25pt]
        C_{\bullet}
        \left((X, A) \smashcap f^{-1} (\rho(u))\right)
        \arrow[d, "
        \begin{pmatrix}
          1 \\ 1
        \end{pmatrix}
        "']
        &
        \begin{matrix}
          C_{\bullet}
          \left((X_0, A_0) \smashcap f^{-1} (\rho(u))\right)
          \\
          \oplus
          \\
          C_{\bullet}
          \left((X_1, A_1) \smashcap f^{-1} (\rho(u))\right)
        \end{matrix}
        \arrow[l, "(1 ~ -1)"']
        \arrow[d]
        &
        C_{\bullet}
        \left((X_0 \cap X_1, A_0 \cap A_1) \smashcap f^{-1} (\rho(u))\right)
        \arrow[l, "
        \begin{pmatrix}
          1 \\ 1
        \end{pmatrix}
        "']
        \arrow[d, "
        \begin{pmatrix}
          1 \\ 1
        \end{pmatrix}
        "]
        \\
        \begin{matrix}
          C_{\bullet}
          \left((X, A) \smashcap f^{-1} (\rho(v_1, u_2))\right)
          \\
          \oplus
          \\
          C_{\bullet}
          \left((X, A) \smashcap f^{-1} (\rho(u_1, v_2))\right)
        \end{matrix}
        \arrow[d, "(1 ~ -1)"']
        &
        \begin{matrix}
          C_{\bullet}
          \left((X_0, A_0) \smashcap f^{-1} (\rho(v_1, u_2))\right)
          \\
          \oplus
          \\
          C_{\bullet}
          \left((X_1, A_1) \smashcap f^{-1} (\rho(v_1, u_2))\right)
          \\
          \oplus
          \\
          C_{\bullet}
          \left((X_0, A_0) \smashcap f^{-1} (\rho(u_1, v_2))\right)
          \\
          \oplus
          \\
          C_{\bullet}
          \left((X_1, A_1) \smashcap f^{-1} (\rho(u_1, v_2))\right)
        \end{matrix}
        \arrow[l]
        \arrow[d]
        &
        \begin{matrix}
          C_{\bullet}
          \left(
            (X_0 \cap X_1, A_0 \cap A_1) \smashcap f^{-1} (\rho(v_1, u_2))
          \right)
          \\
          \oplus
          \\
          C_{\bullet}
          \left(
            (X_0 \cap X_1, A_0 \cap A_1) \smashcap f^{-1} (\rho(u_1, v_2))
          \right)        
        \end{matrix}
        \arrow[l]
        \arrow[d, "(1 ~ -1)"]
        \\
        C_{\bullet}
        \left((X, A) \smashcap f^{-1} (\rho(v))\right)
        &
        \begin{matrix}
          C_{\bullet}
          \left((X_0, A_0) \smashcap f^{-1} (\rho(v))\right)
          \\
          \oplus
          \\
          C_{\bullet}
          \left((X_1, A_1) \smashcap f^{-1} (\rho(v))\right)        
        \end{matrix}
        \arrow[l, "(1 ~ -1)"]
        &
        C_{\bullet}
        \left((X_0 \cap X_1, A_0 \cap A_1) \smashcap f^{-1} (\rho(v))\right)
        \arrow[l, "
        \begin{pmatrix}
          1 \\ 1
        \end{pmatrix}
        "]
      \end{tikzcd}
    \end{equation}
  of simplicial chain complexes with coefficients in $K$.
  Note that each of the arrows of the outer square
  of this large diagram \eqref{eq:nineComp} point in the opposite direction
  when compared to the corresponding arrows in \eqref{eq:DsquareConcrete}.
  Now each row and each column of \eqref{eq:nineComp}
  is a short exact sequence of simplicial chain complexes
  and moreover,
  up to a change of sign
  each of the boundary operators from \eqref{eq:DsquareConcrete}
  arises as a boundary map from the zig-zag lemma
  of the corresponding short exact sequence of simplicial chain complexes
  in \eqref{eq:nineComp}.
  Thus,
  the commutativity of \eqref{eq:DsquareConcrete}
  follows from \cref{lem:anticommutingBoundaries}.
\end{proof}

\begin{cor}[Excision]
  \label{cor:exc}
  Let $X$ be a finite simplicial complex with subcomplexes
  $A$ and $B$ such that $X = A \cup B$.
  Moreover, let $f \colon X \rightarrow \R$
  be a piecewise linear map.
  Then the inclusion
  \[(A, A \cap B; f |_A) \hookrightarrow (X, B; f)\]
  induces a natural isomorphism
  $h(A, A \cap B; f |_A) \xrightarrow{\cong} h(X, B; f)$
  in relative interlevel set homology.
\end{cor}

\begin{proof}
  We consider the Mayer--Vietoris sequence for
  $(X; B, A) \subseteq (X; X, A)$ relative to $f$.
  Since both $h(A, A; f |_A)$ and $h(X, X; f)$
  are constantly zero,
  the isomorphism follows from exactness.
\end{proof}

\begin{cor}[Exact Sequence of a Pair]
  \label{cor:pair}
  Let $(X, A; f)$ be an object in $\mathcal{F}_0$.
  Then we have an exact sequence
  \begin{equation*}
    h(A; f |_A)
    \xrightarrow{\,1\,}
    h(X; f)
    \xrightarrow{\,1\,}
    h(X, A; f)
    \xrightarrow{\,\partial\,}
    h(A; f |_A) \circ T
    \xrightarrow{\,1\,}
    h(X; f) \circ T
    ,
  \end{equation*}
  where $1$ denotes a map induced by inclusion.
\end{cor}

\begin{proof}
  This is equivalent to the Mayer--Vietoris sequence for
  $(A; \emptyset, A) \subseteq (X; X, A)$ relative to $f$.
\end{proof}

The last corollary implies yet another corollary, which will be useful later.

\begin{cor}
  \label{cor:retract}
  Let $(X, A; f)$ be an object in $\mathcal{F}_0$
  such that there is a fiberwise retraction
  ${r \colon X \rightarrow A}$
  to the inclusion $A \subseteq X$:
   $r |_A = \op{id}_A$ and $f \circ r = f$.
  Then
  \begin{equation*}
    h(X; f) \cong
    h(A; f |_A) \oplus h(X, A; f)
    .
  \end{equation*}
  In particular, we have
  $\op{Dgm} (X; f) =
  \op{Dgm} (A; f |_A) + \op{Dgm} (X, A; f)$.
\end{cor}

\begin{proof}
  We consider the exact sequence from the previous corollary.
  The map $r$ yields a left inverse to the map induced by inclusion,
  \begin{equation*}
    h(A; f |_A)
    \xrightarrow{~1~}
    h(X; f),
  \end{equation*}
  so by exactness $\partial \circ T^{-1} = 0 = \partial$.
  Thus we obtain the split exact sequence
  \begin{equation*}
    0
    \rightarrow
    h(A; f |_A)
    \xrightarrow{~1~}
    h(X; f)
    \xrightarrow{~1~}
    h(X, A; f)
    \rightarrow
    0
    ,
  \end{equation*}
  and hence
  \begin{equation*}
    h(X; f) \cong h(A; f |_A) \oplus h(X, A; f)
    .
    \qedhere
  \end{equation*}
\end{proof}

\begin{lem}[Additivity]
  \label{lem:add}
  \sloppy
  Let $(X, A; f)$ be an object of $\mathcal{F}_0$,
  and let ${\{(X_i, A_i) \subseteq (X, A) \mid i = 1, \dots, n\}}$
  be a family of pairs of subcomplexes
  with $X = \coprod_{i=1}^n X_i$ and $A = \coprod_{i=1}^n A_i$.
  Then the inclusions induce a natural isomorphism
  \begin{equation*}
    \bigoplus_{i=1}^n h(X_i, A_i; f |_{X_i})
    \xrightarrow{~\cong~}
    h(X, A; f)
    .
  \end{equation*}
\end{lem}

\begin{proof}
  This follows from the additivity of homology
  or by induction from the Mayer--Vietoris sequence \cref{lem:MVseq}.
\end{proof}

\begin{lem}[Dimension]
  \label{lem:dim}
  Let $f \colon [0, 1] \rightarrow \R$
  be an order-preserving affine map.
  Then 
  \begin{equation*}
    \op{Dgm}(f) = \mathbf{1}_u,
  \end{equation*}
  where $\rho(u) = ([f(0), f(1)], \emptyset)$.
\end{lem}

Here $\mathbf{1}_u$, for $u \in \M$, is the indicator function
$\mathbf{1}_u \colon v \mapsto
  \begin{cases}
    1 & v = u \\
    0 & v \neq u .
  \end{cases}$

\section{Universality of the Bottleneck Distance}

In this section, we prove our main result, by providing a construction that realizes any given extended persistence diagram as the relative interlevel set homology of a function.
We then further extend this construction to realize any $\delta$-matching between extended persistence diagrams as a pair of functions on a common domain and with distance $\delta$ in the supremum norm.

\subsection{Lifting Points in $\M$}
\label{sec:liftingPoints}

We start by defining the notion of a \emph{lift}
of a point $u \in \M$, which is the construction of a function whose extended persistence diagram consists only of the single point $u$.

\begin{definition}
  A \emph{lift} of a point $u \in \M$ is
  a function $f \colon A \subseteq X \rightarrow \R$
  in the category $\mathcal{F}_0$ with
  \[\op{Dgm} (X, A; f) = \mathbf{1}_u \big|_{\op{int} \M}.\]
\end{definition}

We note that for any point $u \in \partial \M$ the inclusion
of the empty set $\emptyset \subset \R$ is a lift of $u$.
For technical reasons, we also allow
for boundary points in the definition, in order to avoid some case distinctions.
As already noted in \cref{sec:persDiagram},
the extended persistence diagram $\op{Dgm}(X, A; f)$
for a function $f \colon A \subseteq X \rightarrow \R$
in $\mathcal{F}_0$ is supported
on the intersection $\mathbb{S}$ of
the downset $\downarrow \op{Im} \blacktriangle \subseteq \M$
with the union of open squares
$\left(-\frac{\pi}{2}, \frac{\pi}{2}\right)^2 + \pi \Z^2$.
In particular, there are no lifts for any points contained in
$\op{int} \M \setminus \mathbb{S}$.

\begin{figure}[t]
  \centering
\begingroup%
  \makeatletter%
  \providecommand\color[2][]{%
    \errmessage{(Inkscape) Color is used for the text in Inkscape, but the package 'color.sty' is not loaded}%
    \renewcommand\color[2][]{}%
  }%
  \providecommand\transparent[1]{%
    \errmessage{(Inkscape) Transparency is used (non-zero) for the text in Inkscape, but the package 'transparent.sty' is not loaded}%
    \renewcommand\transparent[1]{}%
  }%
  \providecommand\rotatebox[2]{#2}%
  \newcommand*\fsize{\dimexpr\f@size pt\relax}%
  \newcommand*\lineheight[1]{\fontsize{\fsize}{#1\fsize}\selectfont}%
  \ifx\svgwidth\undefined%
    \setlength{\unitlength}{327.34755707bp}%
    \ifx\svgscale\undefined%
      \relax%
    \else%
      \setlength{\unitlength}{\unitlength * \real{\svgscale}}%
    \fi%
  \else%
    \setlength{\unitlength}{\svgwidth}%
  \fi%
  \global\let\svgwidth\undefined%
  \global\let\svgscale\undefined%
  \makeatother%
  \begin{picture}(1,0.34367142)%
    \lineheight{1}%
    \setlength\tabcolsep{0pt}%
    \put(0.99999998,0.03730095){\makebox(0,0)[t]{\lineheight{1.25}\smash{\begin{tabular}[t]{c}$0$\end{tabular}}}}%
    \put(0,0){\includegraphics[width=\unitlength,page=1]{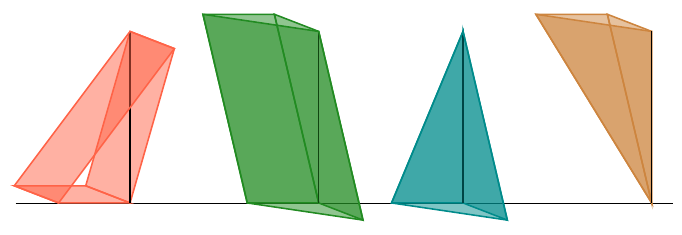}}%
    \put(0.99999998,0.28876785){\makebox(0,0)[t]{\lineheight{1.25}\smash{\begin{tabular}[t]{c}$1$\end{tabular}}}}%
    \put(0,0){\includegraphics[width=\unitlength,page=2]{point-lifts.pdf}}%
  \end{picture}%
\endgroup%

\begingroup%
  \makeatletter%
  \providecommand\color[2][]{%
    \errmessage{(Inkscape) Color is used for the text in Inkscape, but the package 'color.sty' is not loaded}%
    \renewcommand\color[2][]{}%
  }%
  \providecommand\transparent[1]{%
    \errmessage{(Inkscape) Transparency is used (non-zero) for the text in Inkscape, but the package 'transparent.sty' is not loaded}%
    \renewcommand\transparent[1]{}%
  }%
  \providecommand\rotatebox[2]{#2}%
  \newcommand*\fsize{\dimexpr\f@size pt\relax}%
  \newcommand*\lineheight[1]{\fontsize{\fsize}{#1\fsize}\selectfont}%
  \ifx\svgwidth\undefined%
    \setlength{\unitlength}{305.25bp}%
    \ifx\svgscale\undefined%
      \relax%
    \else%
      \setlength{\unitlength}{\unitlength * \real{\svgscale}}%
    \fi%
  \else%
    \setlength{\unitlength}{\svgwidth}%
  \fi%
  \global\let\svgwidth\undefined%
  \global\let\svgscale\undefined%
  \makeatother%
  \begin{picture}(1,0.59459459)%
    \lineheight{1}%
    \setlength\tabcolsep{0pt}%
    \put(0,0){\includegraphics[width=\unitlength,page=1]{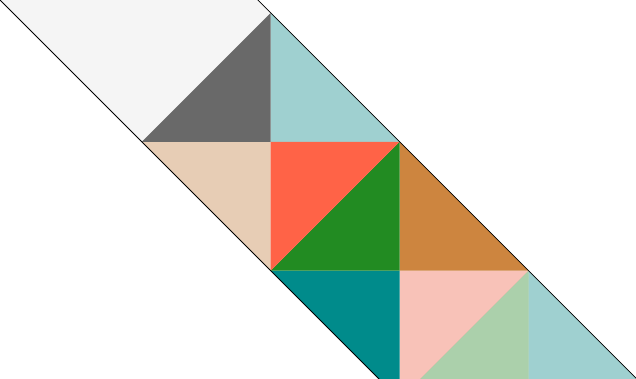}}%
    \put(0.8277027,0.01689189){\makebox(0,0)[t]{\lineheight{1.25}\smash{\begin{tabular}[t]{c}$T^{-3}(D)$\end{tabular}}}}%
    \put(0.56418919,0.18581081){\makebox(0,0)[t]{\lineheight{1.25}\smash{\begin{tabular}[t]{c}$T^{-2}(D)$\end{tabular}}}}%
    \put(0.48986486,0.33445946){\makebox(0,0)[t]{\lineheight{1.25}\smash{\begin{tabular}[t]{c}$T^{-1}(D)$\end{tabular}}}}%
    \put(0.24662162,0.54054054){\makebox(0,0)[t]{\lineheight{1.25}\smash{\begin{tabular}[t]{c}$D$\end{tabular}}}}%
    \put(0.35135135,0.42567568){\makebox(0,0)[lt]{\lineheight{1.25}\smash{\begin{tabular}[t]{l}$\textcolor{White}{R}$\end{tabular}}}}%
    \put(0,0){\includegraphics[width=\unitlength,page=2]{lift-regions.pdf}}%
    \put(0.76407617,0.26295086){\makebox(0,0)[t]{\lineheight{1.25}\smash{\begin{tabular}[t]{c}$l_1$\end{tabular}}}}%
    \put(0.32074693,0.24006388){\makebox(0,0)[t]{\lineheight{1.25}\smash{\begin{tabular}[t]{c}$l_0$\end{tabular}}}}%
  \end{picture}%
\endgroup%

  \caption{
    The region $R \subset \mathbb{S}$ is shaded in dark gray.
    The remaining regions in $\mathbb{I} := \mathbb{S} \setminus R$
    are colored and the corresponding lifts are shown at the top.
    All of the points of the extended subdiagram of degree $1$
    representing \emph{vertical homology classes}
    in the sense of \cite[Section 5]{Cohen-Steiner2009}
    are contained in the dark red region,
    while those representing \emph{horizontal homology classes}
    are contained in the dark green region.
    The dark cyan region corresponds to the ordinary subdiagram
    of degree $1$ and the brown region to the relative subdiagram
    of degree $2$;
    see also \cref{fig:subdiagrams}.
  }
  \label{fig:liftRegions}
\end{figure}

We now construct a lift
$f_u \colon A_u \subseteq X_u \rightarrow \R$
for any point $u \in \mathbb{S}$.
To this end,
we partition $\mathbb{S}$ into regions shown in \cref{fig:liftRegions}.
We start with the region $R \subset \mathbb{S}$,
which is the connected component of $\mathbb{S}$ containing the origin,
shaded in dark gray in \cref{fig:liftRegions}.
For any point $u \in R$, a lift $f_u \colon [0, 1] \rightarrow \R$
is provided by \cref{lem:dim}.
All remaining points of $\mathbb{S}$
are contained in $\mathbb{I} := \mathbb{S} \setminus R$.
We use a construction that is sketched in \cref{fig:liftRegions}
and formalized in \cref{sec:formal_lifts_points}.
This figure shows four pairs of geometric simplicial complexes
with ambient space $\R^3$ and beneath them the strip $\M$
with some of the regions colored.
Moreover,
we identify the first and the second simplicial pair
from the left in \cref{fig:liftRegions}
in the obvious way.
The relative part of each simplicial pair is given by the solid black line.
Four of the regions of $\M$ are shaded with saturated colors;
for a point $u$ in one of these four regions, the corresponding lift
is given by the height function
$r_u \colon A_u \subset X_u \rightarrow [0, 1]$
of the simplicial pair shaded
in the same color, post-composed with the appropriate order-preserving
affine map $b_u \colon [0, 1] \rightarrow \R$;
see \eqref{eq:bu} for an explicit formula for $b_u$.
This way we obtain a lift $f_u := b_u \circ r_u$
for each point $u$ in one of these four regions.
Moreover, for each such $u$ there is a level-preserving isomorphism
$j_u \colon [0, 1] \rightarrow A_u$,
where $A_u$ is the solid black line.
Furthermore, the red simplicial pair and the green simplicial pair
admit an isomorphism preserving the levels of the solid black line.
Thus, we may think of the domains of our choices of lifts
for the saturated red and the saturated green region as being identical.
The other regions of $\mathbb{I}$ are shaded in a pale color.
For a point in one of these regions, the corresponding lift
is given by a higher- or lower-dimensional version
of a lift for the region pictured in a more saturated version
of the same color.
We provide a formal construction of the maps
$r_u \colon X_u \rightarrow [0, 1]$,
$b_u \colon [0, 1] \rightarrow \R$, and
$j_u \colon [0, 1] \xrightarrow{\cong} A_u \subset X_u$
for each $u \in \mathbb{I}$
in \cref{sec:formal_lifts_points}.
Finally, we note that our choices of lifts satisfy the following
essential property.

\begin{lem}
  \label{lem:isoFam}
  The family $\{(X_u, A_u; f_u) \mid u \in \mathbb{S}\}$
  of lifts
  is \emph{isometric} in the sense that any points
  $u, v \in \mathbb{S}$ of finite distance $d(u, v) < \infty$
  satisfy
  \begin{equation*}
    (X_u, A_u) = (X_v, A_v)
    \quad \text{and} \quad
    \left\lVert f_u - f_v \right\rVert_{\infty} = d(u, v)
    .
  \end{equation*}
\end{lem}

\begin{proof}
  Let ${u, v \in \mathbb{S}}$ be of finite distance ${d(u, v) < \infty}$.
  As we identified the first two simplicial pairs of \cref{fig:liftRegions}
  or by the construction of \cref{sec:formal_lifts_points}
  we have
  ${(X_u, A_u) = (X_v, A_v)}$ and ${j_u = j_v}$.
  Now if $u$ and $v$ are contained in a region
  corresponding to an ordinary or a relative subdiagram
  as shown in \cref{fig:subdiagrams},
  then we have
  ${\left\lVert b_u - b_v \right\rVert_{\infty} = d(u, v)}$.
  As the retractions ${r_u \colon A_u \subset X_u \rightarrow [0, 1]}$
  and $r_v$ are level-preserving,
  we have
  ${\left\lVert f_u - f_v \right\rVert_{\infty} = d(u, v)}$
  as well.
  Moreover,
  if $u$ and $v$ are contained in a region
  corresponding to an extended subdiagram,
  then the distance
  $d(u, v)$ is realized by the restrictions of $f_u$ and $f_v$
  to the vertices of ${X_u = X_v}$
  that are not in ${A_u = A_v}$.
\end{proof}

\subsection{Lifting Multisets}
\label{sec:liftingMultisets}

We proceed to extend our construction of lifts from single points to arbitrary persistence diagrams.
In order to motivate our next definition, we first consider the height function shown in \cref{fig:booklet}
and its persistence diagram.
Let ${a = (a_1, a_2)}$ be the green vertex
of the persistence diagram shown in \cref{fig:booklet}.
Letting $\sigma \colon \R \rightarrow \R, \ t \mapsto \pi - t$
be the reflection at $\frac{\pi}{2}$,
we consider the region
${([a_2, a_1] \cup \sigma [a_2, a_1])^2 + 2 \pi \Z^2}$.
The intersection of $\M$ and this region is the green shaded area
in \cref{fig:booklet_diagram_adm}.
As we can see, both the blue and the red vertex are contained
in this green area.
As it turns out, this is true for any realizable persistence diagram.
More specifically,
any realizable persistence diagram
$\mu \colon \op{int} \M \rightarrow \N_0$
satisfies both conditions
of the following definition.

\begin{figure}[t]
  \centering
\begingroup%
  \makeatletter%
  \providecommand\color[2][]{%
    \errmessage{(Inkscape) Color is used for the text in Inkscape, but the package 'color.sty' is not loaded}%
    \renewcommand\color[2][]{}%
  }%
  \providecommand\transparent[1]{%
    \errmessage{(Inkscape) Transparency is used (non-zero) for the text in Inkscape, but the package 'transparent.sty' is not loaded}%
    \renewcommand\transparent[1]{}%
  }%
  \providecommand\rotatebox[2]{#2}%
  \newcommand*\fsize{\dimexpr\f@size pt\relax}%
  \newcommand*\lineheight[1]{\fontsize{\fsize}{#1\fsize}\selectfont}%
  \ifx\svgwidth\undefined%
    \setlength{\unitlength}{360.6818161bp}%
    \ifx\svgscale\undefined%
      \relax%
    \else%
      \setlength{\unitlength}{\unitlength * \real{\svgscale}}%
    \fi%
  \else%
    \setlength{\unitlength}{\svgwidth}%
  \fi%
  \global\let\svgwidth\undefined%
  \global\let\svgscale\undefined%
  \makeatother%
  \begin{picture}(1,0.47826087)%
    \lineheight{1}%
    \setlength\tabcolsep{0pt}%
    \put(0,0){\includegraphics[width=\unitlength,page=1]{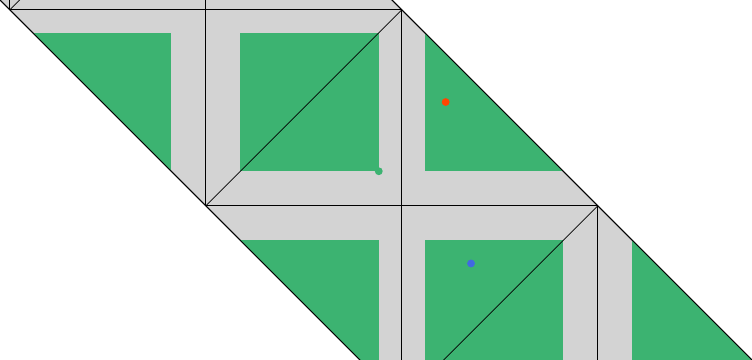}}%
    \put(0.873913,0.03043486){\makebox(0,0)[t]{\lineheight{1.25}\smash{\begin{tabular}[t]{c}$T^{-2}(D)$\end{tabular}}}}%
    \put(0.43913041,0.16956531){\makebox(0,0)[t]{\lineheight{1.25}\smash{\begin{tabular}[t]{c}$T^{-1}(D)$\end{tabular}}}}%
    \put(0.25217386,0.43478259){\makebox(0,0)[t]{\lineheight{1.25}\smash{\begin{tabular}[t]{c}$D$\end{tabular}}}}%
    \put(0.51713559,0.21994885){\makebox(0,0)[t]{\lineheight{1.25}\smash{\begin{tabular}[t]{c}\textcolor{\CSpine}{$a$}\end{tabular}}}}%
    \put(0.77352915,0.26125354){\makebox(0,0)[t]{\lineheight{1.25}\smash{\begin{tabular}[t]{c}$l_1$\end{tabular}}}}%
    \put(0.33825645,0.09652614){\makebox(0,0)[t]{\lineheight{1.25}\smash{\begin{tabular}[t]{c}$l_0$\end{tabular}}}}%
  \end{picture}%
\endgroup%

  \caption{
    The extended persistence diagram from \cref{fig:booklet}
    is admissible.
  }
  \label{fig:booklet_diagram_adm}
\end{figure}

\begin{definition}
  \label{dfn:admissibleMultiset}
  Let $C \subseteq \M$ be a convex subset
  containing $R$
  and let 
  $\mu \colon C \rightarrow \N_0$ be a finite multiset.
  We say that $\mu$ is \emph{admissible}
  if the following two conditions are satisfied:
  \begin{enumerate}
  \item
    The multiset $\mu$ contains exactly one point
    $a = (a_1, a_2) \in R$.
  \item
    All other points of $\mu$ are contained in $\mathbb{I}$
    as well as $([a_2, a_1] \cup \sigma [a_2, a_1])^2 + 2 \pi \Z^2$.
  \end{enumerate}
\end{definition}

In particular, any realizable persistence diagram
is an admissible multiset.
As we will see with \cref{lem:multisetLift} below, the converse
is true as well.
Now the notion of a multiset does not distinguish between individual instances
of the same element of the underlying set.
The following definition enables us to make this distinction.

\begin{figure}[t]
  \centering
\begingroup%
  \makeatletter%
  \providecommand\color[2][]{%
    \errmessage{(Inkscape) Color is used for the text in Inkscape, but the package 'color.sty' is not loaded}%
    \renewcommand\color[2][]{}%
  }%
  \providecommand\transparent[1]{%
    \errmessage{(Inkscape) Transparency is used (non-zero) for the text in Inkscape, but the package 'transparent.sty' is not loaded}%
    \renewcommand\transparent[1]{}%
  }%
  \providecommand\rotatebox[2]{#2}%
  \newcommand*\fsize{\dimexpr\f@size pt\relax}%
  \newcommand*\lineheight[1]{\fontsize{\fsize}{#1\fsize}\selectfont}%
  \ifx\svgwidth\undefined%
    \setlength{\unitlength}{357.32144165bp}%
    \ifx\svgscale\undefined%
      \relax%
    \else%
      \setlength{\unitlength}{\unitlength * \real{\svgscale}}%
    \fi%
  \else%
    \setlength{\unitlength}{\svgwidth}%
  \fi%
  \global\let\svgwidth\undefined%
  \global\let\svgscale\undefined%
  \makeatother%
  \begin{picture}(1,0.4827586)%
    \lineheight{1}%
    \setlength\tabcolsep{0pt}%
    \put(0,0){\includegraphics[width=\unitlength,page=1]{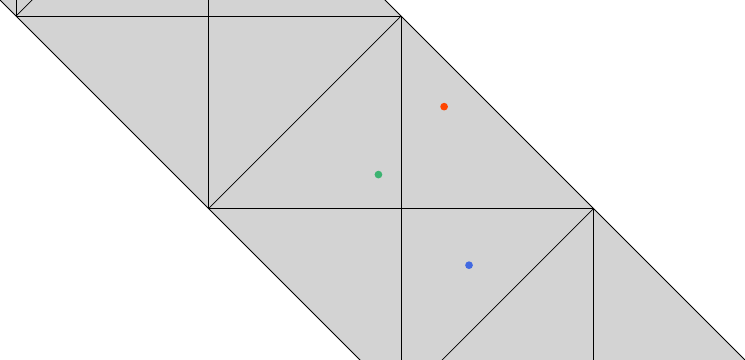}}%
    \put(0.73060351,0.02586201){\makebox(0,0)[t]{\lineheight{1.25}\smash{\begin{tabular}[t]{c}$T^{-2}(D)$\end{tabular}}}}%
    \put(0.4784482,0.15517247){\makebox(0,0)[t]{\lineheight{1.25}\smash{\begin{tabular}[t]{c}$T^{-1}(D)$\end{tabular}}}}%
    \put(0.30172413,0.43103438){\makebox(0,0)[t]{\lineheight{1.25}\smash{\begin{tabular}[t]{c}$D$\end{tabular}}}}%
    \put(0.59989852,0.09203856){\makebox(0,0)[t]{\lineheight{1.25}\smash{\begin{tabular}[t]{c}\textcolor{\CCyl}{$w(3)$}\end{tabular}}}}%
    \put(0.69142997,0.33303753){\makebox(0,0)[t]{\lineheight{1.25}\smash{\begin{tabular}[t]{c}\textcolor{\CHorn}{$w(1)=w(2)$}\end{tabular}}}}%
    \put(0.46526366,0.23960443){\makebox(0,0)[t]{\lineheight{1.25}\smash{\begin{tabular}[t]{c}\textcolor{\CSpine}{$w(0)$}\end{tabular}}}}%
    \put(0.77548138,0.25900125){\makebox(0,0)[t]{\lineheight{1.25}\smash{\begin{tabular}[t]{c}$l_1$\end{tabular}}}}%
    \put(0.3439611,0.09569409){\makebox(0,0)[t]{\lineheight{1.25}\smash{\begin{tabular}[t]{c}$l_0$\end{tabular}}}}%
  \end{picture}%
\endgroup%

  \caption{
    A multiset representation $w \colon \{0, 1, 2, 3\} \rightarrow \M$.
  }
  \label{fig:booklet_diagram_repr}
\end{figure}

\begin{definition}
  A \emph{multiset representation} is a map
  $w \colon S \rightarrow \M$ with finite domain $S$.
  Its associated multiset is
  \begin{equation*}
    \op{Dgm} (w) \colon \M \rightarrow \N_0,
    u \mapsto \# w^{-1} (u)
    .
  \end{equation*}
  We say $w \colon S \rightarrow \M$
  is \emph{admissible}
  if $\op{Dgm} (w)$ is.
  Moreover,
  if $\U \subseteq \M$ is an admissible upset
  containing the support of $\op{Dgm}(w)$
  and if $\mu \colon \U \rightarrow \N_0$ is a multiset
  with $\mu = \op{Dgm} (w) |_{\U}$,
  then we say that $w$
  is a \emph{representation of $\mu$}.
\end{definition}

\cref{fig:booklet_diagram_repr} shows a multiset representation
$w \colon \{0, 1, 2, 3\} \rightarrow \M$
of the extended persistence diagram shown in \cref{fig:booklet}.
For multiset representations $w_i \colon S_i \rightarrow \M$
with $i = 1, 2$ and a map $\zeta \colon S_1 \rightarrow S_2$,
the \emph{norm of $\zeta$} is
\begin{equation*}
  \lVert \varphi \rVert :=
  \sup_{s \in S_1} d(w_1(s), (w_2 \circ \zeta)(s))
  \in [0, \infty]
  .
\end{equation*}
Admissible multiset representations and maps of finite norm
form a category $\mathcal{A}$.
Such categories are also known as normed or weighted categories,
see \cite[Section 2.2]{2017arXiv170706288B}.

We now extend the isometric family of lifts
of points in $\M$ from \cref{lem:isoFam}
to admissible multisets,
by gluing lifts of the points contained.
When gluing these lifts, we have to make certain choices,
and we keep track of these choices using representations of multisets
and another category $\mathcal{F}$ of functions on simplicial pairs.
The category $\mathcal{F}$ is defined verbatim the same way
as the category $\mathcal{F}_0$ from \cref{sec:props},
except that the commutativity of the diagram in \eqref{eq:commHom}
is not required.
For a morphism ${\varphi \colon (X, A; f) \rightarrow (Y, B; g)}$
in $\mathcal{F}$ we define the \emph{norm of $\varphi$} as
\begin{equation*}
  \lVert \varphi \rVert :=
  \lVert f - g \circ \varphi \rVert_{\infty}
  \in [0, \infty)
  .  
\end{equation*}
This makes $\mathcal{F}_0$ the subcategory of $\mathcal{F}$
of all morphisms with vanishing norm.
The main goal of this subsection is to specify a norm-preserving functor
\begin{equation*}
  F \colon \mathcal{A} \rightarrow \mathcal{F}
\end{equation*}
such that for any admissible multiset representation
$w \colon S \rightarrow \M$
we have ${F(w) = (X, \emptyset; f)}$ for some function
$f \colon X \rightarrow \R$, and moreover,
\begin{equation*}
  \beta_0 (X) = 1
  \quad
  \text{and}
  \quad
  \op{Dgm} (f)
  =
  \op{Dgm} (w) |_{\op{int} \M}
  .
\end{equation*}

We now describe the construction of $F(w)$
for any admissible multiset representation $w \colon S \rightarrow \M$.
If $w$ is the multiset representation
shown in \cref{fig:booklet_diagram_repr},
then $F(w)$ will be the height function of the simplicial complex
shown in \cref{fig:booklet}.

We start by defining an auxiliary simplicial complex $A$,
which we may think of as a \emph{blank booklet} whose pages
are indexed or \enquote{numbered} by $S \setminus \{s_0\}$,
where $s_0$ is the single element of $w^{-1}(R)$.
More specifically, we define~$A$ as the mapping cylinder $M (\op{pr}_1)$
of the projection
\[
  \op{pr}_1 \colon [0, 1] \times (S \setminus \{s_0\}) \rightarrow [0, 1], ~
  (t, s) \mapsto t
  .
\]
In \cref{fig:booklet}, the complex $A$ corresponds to the subcomplex
shaded in gray.
With some abuse of notation, we view $[0, 1] \times (S \setminus \{s_0\})$
as the subspace of $A = M (\op{pr}_1)$
that is the top of the mapping cylinder.
Moreover, we refer to $[0, 1] \times (S \setminus \{s_0\})$
as the \emph{fore edges of $A$}.
Similarly, we view $[0, 1]$ as the subspace of $A$ that is the bottom
of the mapping cylinder and refer to it as \emph{the spine of $A$}.
In \cref{fig:booklet}, the spine $[0, 1]$
is shaded in green.

We now extend the booklet $A$ to construct the complex $X$, i.e., the domain of $F(w)$.
For each index ${s \in S \setminus \{s_0\}}$,
we glue the complex $X_{w(s)}$,
which is associated to the specified lift of $w(s)$,
to the fore edge ${[0, 1] \times \{s\}}$ of the booklet $A$
along the map
${j_{w(s)} \colon [0, 1] \times \{s\} \cong [0, 1] \longrightarrow X_{w(s)}}$.
This way we obtain the simplicial pair $(X, A)$.
For the example from \cref{fig:booklet_diagram_repr},
this amounts to gluing the blue simplicial cylinder
as well as the two red horns to the gray shaded booklet
in \cref{fig:booklet}.

We now define a simplicial retraction
$r \colon X \rightarrow A$ of $X$ onto $A$.
Again, we flip through the pages of $A$
and when we are at the page with index $s \in S \setminus \{s_0\}$,
we retract $X_{w(s)}$,
which we glued to this page in the construction of $X$,
to the fore edge ${[0, 1] \times \{s\}}$ via
${r_{w(s)} \colon X_{w(s)} \rightarrow [0, 1] \cong [0, 1] \times \{s\}}$
from \cref{sec:liftingPoints,sec:formal_lifts_points}.
In other words, the family of retractions
$\{r_{w(s)} \colon X_{w(s)} \rightarrow [0, 1] \mid s \in S \setminus \{s_0\}\}$
assembles a retraction $r \colon X \rightarrow A$.
In our example this amounts to orthogonally projecting
the blue simplicial cylinder and the two red horns in \cref{fig:booklet}
to the corresponding fore edge of the gray shaded booklet $A$.

Finally, we construct the function $f \colon X \rightarrow \R$.
We start with defining an affine function
$b \colon A \rightarrow \R$
by specifying its restrictions to the spine $[0, 1]$
and the fore edges ${[0, 1] \times (S \setminus \{s_0\})}$
of the booklet $A$.
The function $f \colon X \rightarrow \R$ is then defined as the composition
$f = b \circ r$.
The restriction of $b \colon A \rightarrow \R$ to the spine $[0, 1]$
of the booklet $A$ is given by
$f_{w(s_0)} \colon [0, 1] \rightarrow \R$
from \cref{sec:liftingPoints},
which is characterized by $\op{Dgm}(f_{w(s_0)}) = \mathbf{1}_{w(s_0)}$
and the assumptions of \cref{lem:dim}.
For each $s \in S \setminus \{s_0\}$
the restriction of $b \colon A \rightarrow \R$
to the fore edge ${[0, 1] \times \{s\}}$
is given by the function
${b_{w(s)} \colon [0, 1] \times \{s\} \cong [0, 1] \longrightarrow \R}$
from \cref{sec:liftingPoints} or rather \eqref{eq:bu}.
In the example from \cref{fig:booklet}
the function $b \colon A \rightarrow \R$
is the restriction of the height function
to the subcomplex shaded in gray.
Finally we set
\[f := b \circ r \quad \text{and} \quad
  F(w) := (X, \emptyset; f) .\]
Since the above constructions are functorial, this defines the desired functor
$F \colon \mathcal{A} \rightarrow \mathcal{F}$.
The norms of morphisms are preserved by construction.
It remains to show that $F$ has the desired property:
\begin{lem}
  \label{lem:multisetLift}
  \(\op{Dgm}(f) = \op{Dgm}(w) |_{\op{int} \M} =: \mu .\)
\end{lem}
\begin{proof}
By \cref{cor:retract}, we have
\[\op{Dgm}(f) =
  \op{Dgm}(b) + \op{Dgm}(X, A; f).\]
By homotopy invariance (\cref{lem:homotopyInv}), we have
\begin{equation*}
  h(b) \cong h\left(f_{w(s_0)}\right)
  \quad
  \text{and thus}
  \quad
  \op{Dgm}(b) = \mathbf{1}_{w(s_0)}
  .
\end{equation*}
Moreover, by excision (\cref{cor:exc})
and additivity (\cref{lem:add}), we have
\begin{equation*}
  h(X, A; f) \cong
  \bigoplus_{s \in S \setminus \{s_0\}}
  h\left(X_{w(s)}, A_{w(s)}; f_{w(s)}\right)
\end{equation*}
and thus
\begin{align*}
  \op{Dgm}(X, A; f)
  & =
    \sum_{s \in S \setminus \{s_0\}}
    \op{Dgm}\left(X_{w(s)}, A_{w(s)}; f_{w(s)}\right)
  \\
  & =
    \sum_{s \in S \setminus \{s_0\}}
    \mathbf{1}_{w(s)} |_{\op{int} \M}
  \\
  & =
    \mu - \mathbf{1}_{w(s_0)}
    .
\end{align*}
Altogether, we obtain
$\op{Dgm}(f) = \op{Dgm}(b) + \op{Dgm}(X, A; f) =
\mathbf{1}_{w(s_0)} + \mu - \mathbf{1}_{w(s_0)} = \mu$.
\end{proof}

\subsection{Universality of the Bottleneck Distance}

We are now equipped with the machinery to realize arbitrary matchings of persistence diagrams, proving \cref{thm:suffCond}.
We start with the following auxiliary result.

\begin{lem}[Matchings with Admissible Projections]
  \label{lem:matching}
  Let
  $\U \subseteq \M$ be an admissible upset and let
  $\mu, \nu \colon \op{int} \U \rightarrow \N_0$
  be two realizable persistence diagrams.
  Suppose ${M \colon \U \times \U \rightarrow \N_0}$
  is a matching between $\mu$ and $\nu$
  of finite norm ${\lVert M \rVert < \infty}$.
  Then there is a matching
  ${N \colon \U \times \U \rightarrow \N_0}$
  (between $\mu$ and $\nu$) such that
  ${\lVert N \rVert \leq \lVert M \rVert}$
  and both
  $\op{pr}_1(N)$ and $\op{pr}_2(N)$
  are admissible multisets.
\end{lem}

\begin{proof}
  Suppose $\op{pr}_1(M) \colon \U \rightarrow \N_0$
  is not admissible.
  By symmetry it suffices to show that there is a matching
  $N \colon \U \times \U \rightarrow \N_0$
  between $\mu$ and $\nu$
  with $\op{pr}_1(N) \colon \U \rightarrow \N_0$ admissible,
  $\lVert N \rVert \leq \lVert M \rVert$,
  and
  $\op{pr}_2(N) = \op{pr}_2(M)$.
  Now in order to construct $N$ from $M$
  we consider the \enquote{elements} of $\op{pr}_1(M)$
  violating the second condition from \cref{dfn:admissibleMultiset}
  and we replace the corresponding edges in $M$ one by one.
  Suppose that
  $u \in (\op{pr}_1(M))^{-1} (\N \setminus \{0\})$ is a vertex violating
  the second condition from \cref{dfn:admissibleMultiset}
  and let
  $a = (a_1, a_2) \in R$ be the unique vertex with $\mu(a) = 1$.
  Then we have $u \in \partial \U$ necessarily
  and there is a vertex $v \in \op{int} \U$
  with $M(u, v) \geq 1$.
  We assume that $u \in l_0$,
  the other two cases
  $u \in l_1$ and $u \in \partial \U \setminus \partial \M$
  are similar.
  Now let $\nabla \subset \M$ be the triangular region of all points in
  ${([a_2, a_1] \cup \sigma [a_2, a_1])^2 + 2 \pi \Z^2}$ and $\M$
  that are of finite distance to $u$ (as well as $v$).
  We have to show that
  there is a point $w \in \nabla \cap \partial \U = \nabla \cap l_0$
  such that ${d(w, v) \leq \lVert M \rVert}$,
  since then we may set
  \[N := M - \mathbf{1}_{(u, v)} + \mathbf{1}_{(w, v)}\]
  to obtain the matching $N \colon \U \times \U \rightarrow \N_0$ with
  ${\lVert N \rVert \leq \lVert M \rVert}$ and
  ${\op{pr}_1(N)(u) = \op{pr}_1(M)(u) - 1}$
  and we may continue by induction.
  Thus, we are done in case we have
  ${d(v, \nabla \cap \partial \U) \leq d(v, u) \leq \lVert M \rVert}$.
  Now suppose we have ${d(v, \nabla \cap \partial \U) \geq d(v, u)}$.
  Without loss of generality
  we assume that $u$ is an upper bound for any of the points in $\nabla$.
  \begin{figure}[t]
    \centering
\begingroup%
  \makeatletter%
  \providecommand\color[2][]{%
    \errmessage{(Inkscape) Color is used for the text in Inkscape, but the package 'color.sty' is not loaded}%
    \renewcommand\color[2][]{}%
  }%
  \providecommand\transparent[1]{%
    \errmessage{(Inkscape) Transparency is used (non-zero) for the text in Inkscape, but the package 'transparent.sty' is not loaded}%
    \renewcommand\transparent[1]{}%
  }%
  \providecommand\rotatebox[2]{#2}%
  \newcommand*\fsize{\dimexpr\f@size pt\relax}%
  \newcommand*\lineheight[1]{\fontsize{\fsize}{#1\fsize}\selectfont}%
  \ifx\svgwidth\undefined%
    \setlength{\unitlength}{246.60326385bp}%
    \ifx\svgscale\undefined%
      \relax%
    \else%
      \setlength{\unitlength}{\unitlength * \real{\svgscale}}%
    \fi%
  \else%
    \setlength{\unitlength}{\svgwidth}%
  \fi%
  \global\let\svgwidth\undefined%
  \global\let\svgscale\undefined%
  \makeatother%
  \begin{picture}(1,0.83636363)%
    \lineheight{1}%
    \setlength\tabcolsep{0pt}%
    \put(0,0){\includegraphics[width=\unitlength,page=1]{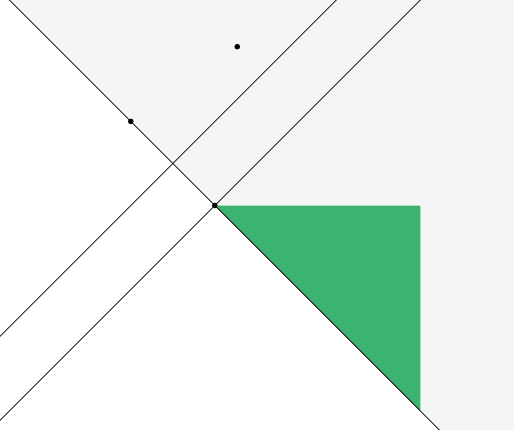}}%
    \put(0.23832507,0.58377961){\makebox(0,0)[rt]{\lineheight{1.25}\smash{\begin{tabular}[t]{r}\(u\)\end{tabular}}}}%
    \put(0.4455978,0.74545454){\makebox(0,0)[rt]{\lineheight{1.25}\smash{\begin{tabular}[t]{r}\(v\)\end{tabular}}}}%
    \put(0.39790634,0.42825343){\makebox(0,0)[rt]{\lineheight{1.25}\smash{\begin{tabular}[t]{r}\(w\)\end{tabular}}}}%
    \put(0.56755922,0.22615977){\makebox(0,0)[lt]{\lineheight{1.25}\smash{\begin{tabular}[t]{l}\(l_0\)\end{tabular}}}}%
    \put(0.74924517,0.33498621){\makebox(0,0)[rt]{\lineheight{1.25}\smash{\begin{tabular}[t]{r}{\fontsize{13.0pt}{13.0pt}\selectfont \(\nabla\)}\end{tabular}}}}%
  \end{picture}%
\endgroup%

    \caption{
      The triangular region
      $\nabla \subset ([a_2, a_1] \cup \sigma [a_2, a_1])^2 + 2 \pi \Z^2$
      is shaded in green.
      This figure is drawn in tangential coordinates,
      so that the bisector between $u$ and $\nabla \cap \partial \U$
      with respect to $d \colon \M \times \M \rightarrow [0, \infty]$
      appears as a straight line.
    }
    \label{fig:auxLemma}
  \end{figure}
  In \cref{fig:auxLemma} we show the bisector
  of $u$ and $\nabla \cap \partial \U = \nabla \cap l_0$
  with respect to $d \colon \M \times \M \rightarrow [0, \infty]$
  in tangential coordinates.
  Now let $w$ be the upper left vertex of $\nabla$.
  As $v$ is on or to the upper left of the bisecting line
  between $u$ and $\nabla \cap \partial \U$
  it is also to the upper left of the line through $w$
  that is parallel to the bisector
  and hence
  \begin{equation}
    \label{eq:closestPoint}
    d(v, w) = d(v, \nabla).
  \end{equation}
  Now let $a'$ be the unique point in $R$ with $\nu(a') = 1$.
  Then we have $d(a', a) \geq d(v, \nabla)$
  as well as $M(a, a') = 1$.
  In conjunction with \eqref{eq:closestPoint}
  we obtain
  \begin{equation*}
    d(v, w) = d(v, \nabla) \leq d(a', a) \leq \lVert M \rVert
    .
    \qedhere
  \end{equation*}
\end{proof}

\thmSuffCond*

\begin{proof}
  As $d_B (\mu, \nu) < \infty$,
  there is a matching between $\mu$ and $\nu$,
  and as these multisets are finite,
  the infimum
  \[d_B (\mu, \nu) =
    \inf
    \{ \lVert M \rVert \mid \text{$M$ is a matching between $\mu$ and $\nu$}\}
  \]
  is attained by some matching $M$.
  Moreover,
  we may assume that $\op{pr}_1 (M)$ and $\op{pr}_2 (M)$ are admissible
  by the previous \Cref{lem:matching}.

  We choose representations
  $w_i \colon S_i \rightarrow \M$
  of $\op{pr}_i (M)$ for $i = 1, 2$,
  noting that both $S_i$ are in bijection with the pairs in the multiset $M$. 
  We choose a bijection
  $\zeta \colon S_1 \rightarrow S_2$
  with $\lVert \zeta \rVert = \lVert M \rVert$.
  Now let $(X, \emptyset; f) := F(w_1)$, and let
  $g'$ be the function of $F(w_2)$.
  Since $F$ is a functor, $F(\zeta)$ is a homeomorphism,
  and so the persistence diagrams of $g'$ and
  $g := g' \circ F(\zeta)$ are identical.
  In summary, we obtain
  \begin{multline*}
    \op{Dgm} (f) |_{\op{int} \U} = \mu,
    ~
    \op{Dgm} (g) |_{\op{int} \U} = \nu,
    \quad \text{and}
    \\
    \lVert f - g \rVert_{\infty}
    =
    \lVert F(\zeta) \rVert
    =
    \lVert \zeta \rVert
    =
    \lVert M \rVert
    =
    d_B (\mu, \nu)
    .
    \qedhere
  \end{multline*}
\end{proof}

\begin{cor}
  \label{cor:bottleneckGeodesic}
  Let $\U \subseteq \M$ be an admissible upset
  and consider the set of realizable persistence diagrams
  $\mu \colon \op{int} \U \rightarrow \N_0$
  together with the bottleneck distance $d_B$
  as an extended metric space.
  This is then a geodesic extended metric space.
\end{cor}

\begin{proof}
  Let
  $\mu, \nu \colon \op{int} \U \rightarrow \N_0$
  be two realizable persistence diagrams
  with $d_B (\mu, \nu) < \infty$.
  Then
  there exists a finite simplicial complex $X$ and piecewise linear functions
  $f, g \colon X \rightarrow \R$ with
  \begin{equation*}
    \op{Dgm}(f) |_{\op{int} \U} = \mu,
    \quad
    \op{Dgm}(g) |_{\op{int} \U} = \nu,
    ~
    \text{and}
    \quad
    \lVert f-g \rVert_{\infty}
    =
    d_B (\mu, \nu)
  \end{equation*}
  by \cref{thm:suffCond}.
  Thus, a geodesic from $\mu$ to $\nu$ can be constructed by considering the convex combinations
  \begin{equation*}
    \gamma \colon [0, 1] \rightarrow \N_0^{\op{int} \U},\,
    t \mapsto \op{Dgm}((1 - t) f + t g) |_{\op{int} \U},
  \end{equation*}
  proving the claim.
\end{proof}

\section{Contrasting Bottleneck and Interleaving Distances}
\label{sec:contrastInterleavingDist}

For a PL function $f \colon X \rightarrow \R$
on a finite simplicial complex $X$
the \emph{derived levelset persistence} of $f$
as introduced by \cite{MR3259939}
is fully determined by the extended persistence diagram $\op{Dgm}(f)$
(up to isomorphism) and vice versa,
see for example
\mbox{\cite[Theorem 1.5, Proposition 3.34]{2022arXiv220515275B}},
and \cite[Section 3.2.2]{2021arXiv210809298B}.
Moreover,
the derived interleaving distance
by \cite{MR3873181,MR3259939}
and the bottleneck distance
coincide in this setting
by \cite{MR4355732}.

In the following we show that the situation is more subtle
when we restrict to an admissible upset ${\U \subset \M}$.
We illustrate this for the choice of $\U$ for which $\op{Dgm}(f) |_{\op{int} \U}$
contains the same information
as the level set barcode of $f \colon X \rightarrow \R$
in degree $0$; explicitly, this is the upset of $T^{-1}(D)$ (see also \cite[Section 3.2.1]{2021arXiv210809298B}).
Moreover, the level set barcode fully classifies
the pushforward $f_* K_X$
of the sheaf $K_X$ of locally constant $K$-valued functions on $X$
by \cite[Section 15.3]{MR3259939} and \cite{MR332887}.
(Other sources include \cite[Corollary 7.3]{2016arXiv160307876G},
\cite[Proposition 1.16]{MR3873181}, as well as
\cite[Section 1.1]{2022arXiv220515275B} and
\cite[Section 3.2.1]{2021arXiv210809298B}.)
One might therefore be tempted to expect
that the interleaving distance of sheaves on the reals
in the sense of \mbox{\cite[Definition 15.2.3]{MR3259939}}
shares certain properties with the bottleneck distance.
However,
in contrast to \cref{cor:bottleneckGeodesic}
the interleaving distance of sheaves is not intrinsic,
let alone geodesic.

Before we prove this, we fix some notation
which we adopt in part from \cite{MR3413628}.
For a topological space $X$ let $\mathrm{lcConst}(X, K)$
be the $K$-vector space of locally constant functions
on $X$ with values in $K$.
Then we obtain a contravariant functor
\begin{equation*}
  \mathrm{lcConst}(-, K)
  \colon
  X \mapsto \mathrm{lcConst}(X, K)
\end{equation*}
from the category of topological spaces
to the category of vector spaces over $K$.
If $X$ is a locally connected space,
then we write $K_X$ for the restriction
of $\mathrm{lcConst}(-, K)$
to the subcategory of open subsets of $X$
and inclusions.
This contravariant functor $K_X$
is also known as the sheaf of locally constant functions on $X$
with values in $K$.
Moreover,
we denote the poset of open intervals on the real numbers $\R$
by $\mathcal{I}$.
Suppose $f \colon X \rightarrow \R$
is a continuous function
on a locally connected space $X$.
Then we have the monotone map
\begin{equation*}
  f^{-1} \colon I \mapsto f^{-1}(I)
\end{equation*}
sending $I \in \mathcal{I}$
to the open subset $f^{-1}(I) \subseteq X$.
With this we may define the pushforward
as the composition
\begin{equation*}
  f_* K_X := \mathrm{lcConst}(-, K) \circ f^{-1}
  \colon
  \mathcal{I}^{\circ} \rightarrow \mathrm{Vect}_K,\,
  I \mapsto \mathrm{lcConst}(f^{-1}(I), K)
  ,
\end{equation*}
where $\mathcal{I}^{\circ}$ denotes the poset of open intervals
endowed with the superset-order $\supseteq$.
As $\mathcal{I}$ is closed under intersections,
the usual sheaf condition
makes sense for presheaves on  $\mathcal{I}$.
Moreover, as $\mathcal{I}$ is a basis
for the Euclidean topology on $\R$,
we will henceforth refer to (co)sheaves on $\mathcal{I}$
as (co)sheaves on $\R$.
Furthermore,
as $f^{-1} \colon I \mapsto f^{-1}(I)$
preserves unions and intersections,
the presheaf
$f_* K_X \colon \mathcal{I} \rightarrow \mathrm{Vect}_K$
is a sheaf on the reals $\R$.
Completely analogously,
if $X$ is locally path-connected,
then the restriction of $\pi_0$
to the subcategory of open subsets of $X$
is a cosheaf on $X$.
Precomposing this cosheaf by
$f^{-1} \colon I \mapsto f^{-1}(I)$
we obtain the Reeb cosheaf
\begin{equation*}
  \mathrm{RcS}(f) := \pi_0 \circ f^{-1} \colon
  \mathcal{I} \rightarrow \mathrm{Set}
\end{equation*}
as defined by
\cite{MR3505333}.

\begin{lem}
  \label{lem:ReebToSheaf}
  If $f \colon X \rightarrow \R$
  is a continuous function
  with $X$ locally path-connected,
  then there is a sheaf isomorphism
  \begin{equation*}
    \eta^f \colon
    f_* K_X \longrightarrow
    \mathrm{Map}(-, K) \circ \mathrm{RcS}(f)
    .
  \end{equation*}
\end{lem}

\begin{proof}
  Suppose
  $a \colon X \rightarrow K$
  is a locally constant function on $X$.
  By the universal property of the quotient set
  there is a unique \enquote{extension}
  of $a \colon X \rightarrow K$
  along the projection to $\pi_0 (X)$:
  \begin{equation*}
    \begin{tikzcd}
      X
      \arrow[r, "a"]
      \arrow[d, two heads]
      &
      K
      .
      \\
      \pi_0 (X)
      \arrow[ur, "\bar{a}"']
    \end{tikzcd}
  \end{equation*}
  This yields a natural transformation
  \begin{equation*}
    \eta \colon
    \mathrm{lcConst}(-, K)
    \Longrightarrow
    \mathrm{Map}(-, K) \circ \pi_0
    ,
  \end{equation*}
  which restricts to a natural isomorphism
  on locally path-connected spaces.
  By whiskering $\eta$ from the right with
  $f^{-1} \colon I \mapsto f^{-1}(I)$,
  we obtain a natural transformation
  \begin{equation*}
    \eta^f := \eta \circ f^{-1} \colon
    f_* K_X \Longrightarrow
    \mathrm{Map}(-, K) \circ \mathrm{RcS}(f)
    ,
  \end{equation*}
  which is a sheaf isomorphism
  if $X$ is locally path-connected.  
\end{proof}

For $\delta \geq 0$ we define the order-preserving map
\[\Omega_\delta \colon \mathcal{I} \rightarrow \mathcal{I},\,
  (a, b) \mapsto (a-\delta, a+\delta).\]
This definition by itself defines an order-preserving map
\begin{equation*}
  [0, \infty) \rightarrow \mathrm{Map}(\mathcal{I}, \mathcal{I}),\,
  \delta \mapsto \Omega_{\delta}
  .
\end{equation*}
Now considering $\mathcal{I}$ as a poset category,
we may consider the relation
${\Omega_0 = \op{id} \leq \Omega_{\delta}}$
as a natural transformation,
which we denote by
${\Omega_{0 \leq \delta} \colon \Omega_0 \rightarrow \Omega_{\delta}}$.

\begin{definition}[$\delta$-Interleaving]
  Let $\delta \geq 0$ and let $F$ and $G$ be functors
  from $\mathcal{I}$
  to some category $\mathcal{C}$.
  Then a $\delta$-interleaving of $F$ and $G$
  is a pair of natural transformations
  \[\varphi \colon F \rightarrow G \circ \Omega_\delta
    \quad \text{and} \quad
    \psi \colon G \rightarrow F \circ \Omega_\delta
  \]
  such that both triangles in the diagram
  \begin{equation}
    \label{eq:interleaving}
    \begin{tikzcd}[row sep=large, column sep=10ex]
      F
      \arrow[dd, bend right=55,
      "F \circ \Omega_{0 \leq 2\delta}"']
      &
      G
      \arrow[dl, "\psi"' near start]
      \arrow[dd, bend left=55,
      "G \circ \Omega_{0 \leq 2 \delta}"]
      \\
      F \circ \Omega_\delta
      &
      G \circ \Omega_\delta
      \arrow[from=ul, crossing over, "\varphi" near start]
      \arrow[dl, "\psi \circ \Omega_\delta"' near end]
      \\
      F \circ \Omega_{2 \delta}
      &
      G \circ \Omega_{2 \delta}
      \arrow[from=ul, crossing over, "\varphi \circ \Omega_\delta" near end]
    \end{tikzcd}    
  \end{equation}
  commute.
\end{definition}

As presheaves could be described as functors taking values
in an opposite category,
the previous definition can be applied to presheaves as well.
However, as a matter of convention we prefer to think of presheaves
as functors on $\mathcal{I}^{\circ}$.
Transcribing the previous definition into this setting
a $\delta$-interleaving of presheaves $F$ and $G$
is a pair of presheaf homomorphisms
\[\varphi \colon G \circ \Omega_\delta \rightarrow F
  \quad \text{and} \quad
  \psi \colon F \circ \Omega_\delta \rightarrow G
\]
such that both triangles in the diagram
\begin{equation}
  \label{eq:interleavingPresheaves}
  \begin{tikzcd}[row sep=large, column sep=10ex]
    F \circ \Omega_{2 \delta}
    \arrow[dd, bend right=55,
    "F \circ \Omega_{0 \leq 2\delta}"']
    &
    G \circ \Omega_{2 \delta}
    \arrow[dl, "\varphi \circ \Omega_\delta"' near end]
    \arrow[dd, bend left=55,
    "G \circ \Omega_{0 \leq 2 \delta}"]
    \\
    F \circ \Omega_\delta
    &
    G \circ \Omega_\delta
    \arrow[from=ul, crossing over, "\psi \circ \Omega_\delta" near end]
    \arrow[dl, "\varphi"' near end]
    \\
    F
    &
    G
    \arrow[from=ul, crossing over, "\psi" near end]
  \end{tikzcd}    
\end{equation}
commute.

\begin{lem}
  \label{lem:ReebToSheavesInterleaving}
  Suppose
  $f \colon X \rightarrow \R$
  and
  $g \colon Y \rightarrow \R$
  are continuous functions
  with $X$ and $Y$ locally path-connected.
  If
  $\mathrm{RcS}(f)$ and $\mathrm{RcS}(g)$
  are $\delta$-interleaved,
  then
  $f_* K_X$ and $g_* K_Y$
  are $\delta$-interleaved as well.
\end{lem}

\begin{proof}
  By \cref{lem:ReebToSheaf}
  it suffices to show that
  ${\mathrm{Map}(-, K) \circ \mathrm{RcS}(f)}$
  and
  ${\mathrm{Map}(-, K) \circ \mathrm{RcS}(g)}$
  are $\delta$-interleaved.
  This in turn follows from
  \mbox{\cite[Proposition 2.2.11]{MR3413628}}.  
\end{proof}

Now let $F$ be a sheaf on the reals
and let $I = (a, b) \in \mathcal{I}$ be an open interval.
Then we have the open cover $(-\infty, b) \cup (a, \infty) = \R$.
Using the same notation as in \cref{sec:sectIntersect},
there is a naturally induced map
\begin{equation*}
  \nabla_{(-\infty, b), (a, \infty)} (F) \colon
  F((-\infty, b)) \oplus F((a, \infty)) \rightarrow F(I)
\end{equation*}
with cokernel
${\kappa_I (F) := \op{coker} \big(\nabla_{(-\infty, b), (a, \infty)} (F)\big)}$,
which is natural in $F$.
Now suppose we have ${\delta \geq 0}$,
then we may apply the functor
${\kappa_I = \op{coker} \big(\nabla_{(-\infty, b), (a, \infty)} (-)\big)}$
to the sheaf homomorphism
\begin{equation*}
  F \circ \Omega_{0 \leq \delta} \colon F \circ \Omega_{\delta} \rightarrow F
\end{equation*}
to obtain the map
\begin{equation}
  \label{eq:cokerCoverSmoothing}
  \kappa_I (F \circ \Omega_{0 \leq \delta}) \colon
  \kappa_I (F \circ \Omega_{\delta}) \rightarrow
  \kappa_I (F)
  .
\end{equation}

\begin{lem}
  \label{lem:cokerCoverSmoothing}
  The map ${\kappa_I (F \circ \Omega_{0 \leq \delta})}$
  from \eqref{eq:cokerCoverSmoothing} is injective.
\end{lem}

\begin{proof}
  We have
  \[\kappa_I (F \circ \Omega_{\delta})
    =
    \op{coker} \big(
    \nabla_{(-\infty, b), (a, \infty)} (F \circ \Omega_{\delta})
    \big)
    =
    \op{coker} \big(
    \nabla_{(-\infty, b+\delta), (a-\delta, \infty)} (F \circ \Omega_{\delta})
    \big)
    ,
  \]
  hence the result follows from \cref{lem:cokerSuperCover}.
\end{proof}

\begin{cor}
  \label{cor:cokerCoverInterleaving}
  Let $\delta \geq 0$, let $F$ and $G$ be sheaves on the real numbers,
  and let
  \[\varphi \colon G \circ \Omega_\delta \rightarrow F
    \quad \text{and} \quad
    \psi \colon F \circ \Omega_\delta \rightarrow G
  \]
  be a $\delta$-interleaving of $F$ and $G$.
  Moreover,
  let $I \in \mathcal{I}$ be an open interval.
  Then the map
  \begin{equation*}
    \kappa_I (\psi \circ \Omega_{\delta}) \colon
    \kappa_I (F \circ \Omega_{2 \delta}) \rightarrow
    \kappa_I (G \circ \Omega_{\delta})
  \end{equation*}
  is injective.
\end{cor}

\begin{proof}
  This follows in conjunction with the commutativity
  of the left triangle in \eqref{eq:interleaving}.
\end{proof}

Now we harness the example by \cite[Proposition 3.6]{MR4323622}
as well as \cref{cor:cokerCoverInterleaving}
to show that the interleaving distance of sheaves is not intrinsic.
Even more, it differs from the induced intrinsic version of the interleaving distance by a factor of $2$ in this example.
To this end,
let
\[C := \{(x, y, z) \in \R^3 \mid x^2 + z^2 = 1, |2y - x| \leq 1\}\]
and let
\begin{equation}
\label{eq:exampleReeb}
  f \colon C \rightarrow \R, (x, y, z) \mapsto x
  \quad \text{and} \quad
  g \colon C \rightarrow \R, (x, y, z) \mapsto y
\end{equation}
be the projections to the $x$- and the $y$-axis respectively,
see also \cref{fig:skewedCylinder}.
\begin{figure}[t]
  \centering
  \includegraphics{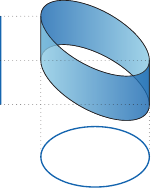}
  \caption{
    The skewed cylinder $C$ by \cite[Proposition 3.6]{MR4323622}.
    The functions $f, g \colon C \rightarrow \R$
    are the projections to the horizontal and vertical axes, respectively.
  }
  \label{fig:skewedCylinder}
\end{figure}
We consider the two pushforwards
$f_* K_C$ and $g_* K_C$
of the sheaf of locally constant $K$-valued functions $K_C$ on $C$.
By \cite[Proposition 3.9]{MR4323622}
the interleaving distance of the Reeb graphs
associated to $f$ and $g$ is $\frac12$.
Moreover,
the interleaving distance of these Reeb graphs
is the same as the interleaving distance of the Reeb cosheaves
associated to $f$ and $g$
by \cite{MR3505333}.
Thus,
the interleaving distance of $f_* K_C$ and $g_* K_C$
is at most $\frac12$ by \cref{lem:ReebToSheavesInterleaving}.
Since ${|f(p) - g(p)| \leq 1}$ for all ${p \in C}$,
the induced intrinsic interleaving distance
of $f_* K_C$ and $g_* K_C$ is at most $1$.
In the following we show that the induced intrinsic interleaving distance
is at least $1$ and hence equal to $1$.
Thus,
the interleaving distance of sheaves on the reals
and the induced intrinsic interleaving distance
differ globally by at least a factor of $2$.

\begin{prp}
  \label{prp:intrinsicDistSheaves}
  The induced intrinsic interleaving distance
  of $f_* K_C$ and $g_* K_C$ is $1$.
\end{prp}

\begin{proof}
  Assume for a contradiction that the induced intrinsic interleaving distance
  of $f_* K_C$ and $g_* K_C$ is less than $1$.
  Then there exists an $\varepsilon > 0$ such that there is a 1-parameter family of sheaves
  \[\{\gamma(t) \mid 0 \leq t \leq 1 - 2\varepsilon\}\]
  with
  $\gamma(t)$ and $\gamma(t')$ being ${|t - t'|}$-interleaved
  for all ${t, t' \in [0, 1 - 2 \varepsilon]}$
  and
  ${\gamma(0) = f_* K_C}$ and ${\gamma(1 - 2 \varepsilon) = g_* K_C}$.
  By the monotonicity of the interleaving property, we may assume without loss of generality  that \[\varepsilon = \frac{1}{2^{N}+1}\] for some $N \in \N$.
    We consider the values of $\gamma$
  at a finite sequence of consecutive times $t_0 , \dots, t_N$, where
    \[t_n := (1-\varepsilon)\sum_{i=1}^n \frac1{2^n} = (1-\varepsilon)\left(
    1-\frac1{2^n}\right)\]
  is chosen such that 
  $\delta_n := t_n - t_{n-1} = \frac{1-\varepsilon}{2^n}=2^{N-n}\varepsilon$
  and
  $t_{N} = 1 - 2 \varepsilon$.
  For any two consecutive times
  $t_{n-1}$ and $t_n$ with $n = 1, \dots, N$,
  choose a \mbox{$\delta_n$-interleaving} of
  $\gamma(t_n)$ and $\gamma(t_{n-1})$
  with interleaving homomorphisms
  \[
    \varphi_n  \colon
                \gamma(t_n) \circ \Omega_{\delta_n} \rightarrow
                \gamma(t_{n-1})
    \quad
    \text{and} \quad
    \psi_n  \colon
             \gamma(t_{n-1}) \circ \Omega_{\delta_n} \rightarrow
             \gamma(t_n)
             .
  \]
  
  We illustrate this for the example $N=3$.
  By gluing the left triangles of the corresponding interleaving diagrams
  \eqref{eq:interleaving},
  we obtain the commutative diagram
  \begin{equation}
    \label{eq:gluedInterleavingTriangles}
    \begin{tikzcd}[column sep={6ex,between origins},row sep={3ex,between origins}]
      f_* K_C \circ \Omega_{8 \varepsilon}
      \arrow[dddd, "f_* K_C \circ \Omega_{0 \leq 8 \varepsilon}"']
      \arrow[dr, "\psi_0 \circ \Omega_{8 \varepsilon}"]
      \\[18ex]
      &[15ex]
      \gamma(4 \varepsilon) \circ \Omega_{4 \varepsilon}
      \arrow[ddd,
      "\gamma(4 \varepsilon) \circ \Omega_{0 \leq 4 \varepsilon}" description]
      \arrow[dr, "\psi_1 \circ \Omega_{2 \varepsilon}"]
      \arrow[dddl, "\varphi_0"']
      \\[7ex]
      & &[4ex]
      \gamma(6 \varepsilon) \circ \Omega_{2 \varepsilon}
      \arrow[dd,
      "\gamma(6 \varepsilon) \circ \Omega_{0 \leq 2 \varepsilon}" description]
      \arrow[dr, "\psi_2 \circ \Omega_{ \varepsilon}"]
      \arrow[ddl, "\varphi_1"']
      \\[4ex]
      & & &
      g_* K_C \circ \Omega_{ \varepsilon}
      \arrow[dl, "\varphi_2"]
      \\[4ex]
      f_* K_C
      &
      \gamma(4 \varepsilon)
      &
      \gamma(6 \varepsilon)
      .
    \end{tikzcd}
  \end{equation}
  
  Now let $I := (-\varepsilon, \varepsilon)$.
  By \cref{cor:cokerCoverInterleaving}
  we have the injective map
  \begin{equation*}
    \kappa_I (\psi_n \circ \Omega_{\delta_n}) \colon
    \kappa_I (\gamma(t_{n-1}) \circ \Omega_{2 \delta_n}) \rightarrow
    \kappa_I (\gamma(t_n) \circ \Omega_{\delta_n})
  \end{equation*}
  for any $n = 1, \dots, N$.
  Moreover,
  as is illustrated by \eqref{eq:gluedInterleavingTriangles},
  any two consecutive homomorphisms
  ${\psi_n \circ \Omega_{\delta_n} \colon
    \gamma(t_{n-1}) \circ \Omega_{2 \delta_n} \rightarrow
    \gamma(t_n) \circ \Omega_{\delta_n}
  }$
  are composable.
  Setting $\delta_0 := 1 - \varepsilon = 2^{N}\varepsilon$, we have $\delta_n = 2 \delta_{n-1}$ for any $n = 0, \dots, N$, and thus we get the composition
  \begin{multline*}
  f_* K_C \circ \Omega_{\delta_0}
  \xrightarrow{\psi_1 \circ \Omega_{\delta_1}}
  \gamma(t_1) \circ \Omega_{\delta_1}
  \xrightarrow{\psi_2 \circ \Omega_{\delta_2}}
  \gamma(t_2) \circ \Omega_{\delta_2}
  \xrightarrow{\psi_3 \circ \Omega_{\delta_3}}
  \dots
  \\
  \dots
  \xrightarrow{\psi_{N-1} \circ \Omega_{\delta_{N-1}}}
  \gamma(t_{N-1}) \circ \Omega_{\delta_{N-1}}
  \xrightarrow{\psi_{N} \circ \Omega_{\delta_{N}}}  
  g_* K_C \circ \Omega_{\delta_N}
  \end{multline*}
  of sheaf homomorphisms.
  Applying the functor $\kappa_I$
  we obtain an injective map
  \begin{equation}
    \label{eq:supposedInjection}
    \kappa_I (f_* K_C \circ \Omega_{\delta_0}) \longrightarrow
    \kappa_I (g_* K_C \circ \Omega_{\delta_N})
    .
  \end{equation}
  Now let
  \begin{equation*}
    I' := \Omega_{\delta_0} (I) = (-1, 1)
    \quad
    \text{and} \quad
    I'' := \Omega_{\delta_N} (I) = (-2\varepsilon, 2\varepsilon).
  \end{equation*}
  Then we have
  \begin{equation*}
    \kappa_I (f_* K_C \circ \Omega_{\delta_0}) =
    \kappa_{I'} (f_* K_C)
    \quad \text{and} \quad
    \kappa_I (g_* K_C \circ \Omega_{\delta_N}) =
    \kappa_{I''} (g_* K_C)
    .
  \end{equation*}
  Moreover,
  evaluating $f_* K_C$ at the corresponding intervals,
  we obtain the commutative diagram
  \begin{equation*}
    \begin{tikzcd}
      (f_* K_C)(-\infty, 1 )
      \arrow[r, "\sim"]
      \arrow[dd]
      &
      K
      \arrow[d, "{\begin{pmatrix} 1 \\ 1 \end{pmatrix}}"]
      &
      (f_* K_C)(-1 , \infty)
      \arrow[l, "\sim"']
      \arrow[dd]
      \\[5ex]
      &
      K^2
      \\
      (f_* K_C)(I')
      \arrow[ur, "\sim" rotate=20]
      \arrow[rr, equal]
      &
      &
      (f_* K_C)(I'),
      \arrow[ul, "\sim"' rotate=-20]
    \end{tikzcd}
  \end{equation*}
  hence ${\kappa_{I'} (f_* K_C) \cong K}$.
  On the other hand, we have
  ${\kappa_{I''} (g_* K_C) \cong \{0\}}$,
  contradicting injectivity of the map
  in \eqref{eq:supposedInjection}.
\end{proof}

\begin{cor}
  \label{cor:factorIntrinsic}
  The interleaving distance of sheaves on the reals
  taking values in the category of $K$-vector spaces
  and the induced intrinsic interleaving distance
  differ globally by at least a factor of $2$.
\end{cor}

\begin{remark}
  In the above argument we did not use that $K$ was a field,
  and so the proof applies to any ring, in particular the integers.
  Thus, \cref{cor:factorIntrinsic} also applies to sheaves
  valued in the category of abelian groups.
\end{remark}

We use \cref{prp:intrinsicDistSheaves}
and \cref{lem:ReebToSheavesInterleaving}
to show that the interleaving distance
of Reeb graphs,
or more specifically $\R$-graphs as defined in \cite[Section 2.3]{MR3505333},
is not intrinsic.
This answers a question raised in \cite[Section 4]{MR3685697}.

\begin{cor}
  The induced intrinsic interleaving distance
  of Reeb cosheaves $\mathrm{RcS}(f)$ and $\mathrm{RcS}(g)$ from the example in \eqref{eq:exampleReeb}
is $1$.
\end{cor}

\begin{proof}
  Since ${|f(p) - g(p)| \leq 1}$ for all ${p \in C}$,
  the induced intrinsic interleaving distance
  of $\mathrm{RcS}(f)$ and $\mathrm{RcS}(g)$ is at most $1$.
  Moreover,
  by \cref{prp:intrinsicDistSheaves}
  and \cref{lem:ReebToSheavesInterleaving}
  the induced intrinsic interleaving distance
  of $\mathrm{RcS}(f)$ and $\mathrm{RcS}(g)$ is at least $1$.  
\end{proof}

\begin{cor}
  \label{cor:factorIntrinsicReebCosheaves}
  The interleaving distance of\, $\mathrm{Set}$-valued cosheaves on $\R$
  and the induced intrinsic interleaving distance
  differ globally by at least a factor of $2$.
\end{cor}

\begin{cor}
  \label{cor:factorIntrinsicReebGraphs}
  The interleaving distance of Reeb graphs
  in the sense of \mbox{\cite[Section 4.5]{MR3505333}}
  and the induced intrinsic interleaving distance
  differ globally by at least a factor of $2$.
\end{cor}

\begin{proof}
  By \cite[Section 4.5]{MR3505333}
  two $\R$-graphs are $\delta$-interleaved
  iff their corresponding
  Reeb cosheaves are $\delta$-interleaved.
\end{proof}

\subsection*{Acknowledgements}

This research has been partially supported by the DFG Collaborative Research Center SFB/TRR 109 \emph{Discretization in Geometry and Dynamics}.
BF thanks Justin Curry for useful discussions on the  analogues to the Eilenberg--Steenrod axioms for a homology theory  stated in \cref{sec:props}.
BF thanks the authors of the Haskell library \href{https://diagrams.github.io/}{\texttt{diagrams}}, which
has been indispensable for creating several of the graphics in this document.

\bibliographystyle{alpha}
\bibliography{bib/tomDieck08.bib,bib/dey-2007.bib,bib/cohen-steiner-2007.bib,bib/cohen-steiner-2009.bib,bib/carlsson-2009.bib,bib/carlsson-2019.bib,bib/mcCleary-2020.bib,bib/bendich-2013.bib,bib/lesnick-2015.bib,bib/klein-2005.bib,bib/berkouk-ginot-oudot-2019.bib,bib/spanier-1981.bib,bib/carlsson-deSilva-kalisnik-morozov-2019.bib,bib/bubenik-2017.bib,bib/levanger-2019.bib,bib/bubenik-scott-2014.bib,bib/berkouk-2018.bib,bib/curry14.bib,bib/edelsbrunner-2002.bib,bib/deSilva-munch-patel-2016.bib,bib/botnan-lesnick-2018.bib,bib/barratt-whitehead-1956.bib,bib/crawley-boevey-2015.bib,bib/kashiwara-schapira-1990.bib,bib/bauer-botnan-fluhr-2021.bib,bib/bauer-fluhr-2022.bib,bib/berkouk-ginot-2022.bib,bib/bubenik-deSilva-scott-2015.bib,bib/bauer-landi-memoli-2021.bib,bib/kashiwara-schapira-2018.bib,bib/DeSha-2021.bib,bib/Catanzaro-2021.bib,bib/scoccola-2020.bib,bib/bauer-bjerkevik-fluhr-2022.bib,bib/cardona-curry-lam-lesnick-2022.bib,bib/carriere-oudot-2017.bib,bib/guillermou-2016.bib,bib/gabriel-1972.bib,bib/stacks-project.bib}

\newcommand{\etalchar}[1]{$^{#1}$}
\begin{thebibliography}{CdSKM19}

\bibitem[BBF22a]{2021arXiv210809298B}
Ulrich {Bauer}, Magnus {Bakke Botnan}, and Benedikt {Fluhr}.
\newblock {Structure and Interleavings of Relative Interlevel Set Cohomology}.
\newblock \url{https://arxiv.org/abs/2108.09298v3}, 2022.

\bibitem[BBF22b]{MR4470893}
Ulrich Bauer, H\aa vard~Bakke Bjerkevik, and Benedikt Fluhr.
\newblock Quasi-universality of {R}eeb graph distances.
\newblock In {\em 38th {I}nternational {S}ymposium on {C}omputational
  {G}eometry ({S}o{CG} 2022)}, volume 224 of {\em LIPIcs. Leibniz Int. Proc.
  Inform.}, pages Paper No. 14, 18. Schloss Dagstuhl. Leibniz-Zent. Inform.,
  Wadern, 2022.

\bibitem[BdSS15]{MR3413628}
Peter Bubenik, Vin de~Silva, and Jonathan Scott.
\newblock Metrics for generalized persistence modules.
\newblock {\em Found. Comput. Math.}, 15(6):1501--1531, 2015.

\bibitem[BdSS17]{2017arXiv170706288B}
Peter Bubenik, Vin de~Silva, and Jonathan Scott.
\newblock Interleaving and gromov-hausdorff distance, 2017.
\newblock Preprint.

\bibitem[BEMP13]{Bendich-2013}
Paul Bendich, Herbert Edelsbrunner, Dmitriy Morozov, and Amit Patel.
\newblock Homology and robustness of level and interlevel sets.
\newblock {\em Homology Homotopy Appl.}, 15(1):51--72, 2013.

\bibitem[BF22]{2022arXiv220515275B}
Ulrich {Bauer} and Benedikt {Fluhr}.
\newblock {Relative Interlevel Set Cohomology Categorifies Extended Persistence
  Diagrams}.
\newblock {\em arXiv e-prints}, May 2022.

\bibitem[BG22]{MR4355732}
Nicolas Berkouk and Gr\'{e}gory Ginot.
\newblock A derived isometry theorem for sheaves.
\newblock {\em Adv. Math.}, 394:Paper No. 108033, 39, 2022.

\bibitem[BGO19]{2019arXiv190709759B}
Nicolas {Berkouk}, Gr{\'e}gory {Ginot}, and Steve {Oudot}.
\newblock {Level-sets persistence and sheaf theory}.
\newblock {\em arXiv e-prints}, Jul 2019.

\bibitem[BLM21]{MR4323622}
Ulrich Bauer, Claudia Landi, and Facundo M\'{e}moli.
\newblock The {R}eeb graph edit distance is universal.
\newblock {\em Found. Comput. Math.}, 21(5):1441--1464, 2021.

\bibitem[CCF{\etalchar{+}}20]{MR4130976}
Michael~J. Catanzaro, Justin~M. Curry, Brittany~Terese Fasy, J\={a}nis
  Lazovskis, Greg Malen, Hans Riess, Bei Wang, and Matthew Zabka.
\newblock Moduli spaces of {M}orse functions for persistence.
\newblock {\em J. Appl. Comput. Topol.}, 4(3):353--385, 2020.

\bibitem[CCLL22]{MR4470903}
Robert Cardona, Justin Curry, Tung Lam, and Michael Lesnick.
\newblock The universal {$\ell^p$}-metric on merge trees.
\newblock In {\em 38th {I}nternational {S}ymposium on {C}omputational
  {G}eometry ({S}o{CG} 2022)}, volume 224 of {\em LIPIcs. Leibniz Int. Proc.
  Inform.}, pages Paper No. 24, 20. Schloss Dagstuhl. Leibniz-Zent. Inform.,
  Wadern, 2022.

\bibitem[CdM09]{Carlsson:2009:ZPH:1542362.1542408}
Gunnar Carlsson, Vin {de Silva}, and Dmitriy Morozov.
\newblock Zigzag persistent homology and real-valued functions.
\newblock In {\em Proceedings of the Twenty-fifth Annual Symposium on
  Computational Geometry}, SCG '09, pages 247--256, New York, NY, USA, 2009.
  ACM.

\bibitem[CdSKM19]{MR3924175}
Gunnar Carlsson, Vin de~Silva, Sara Kali\v{s}nik, and Dmitriy Morozov.
\newblock Parametrized homology via zigzag persistence.
\newblock {\em Algebr. Geom. Topol.}, 19(2):657--700, 2019.

\bibitem[CO17]{MR3685697}
Mathieu Carri\`ere and Steve Oudot.
\newblock Local equivalence and intrinsic metrics between {R}eeb graphs.
\newblock In {\em 33rd {I}nternational {S}ymposium on {C}omputational
  {G}eometry}, volume~77 of {\em LIPIcs. Leibniz Int. Proc. Inform.}, pages
  Art. No. 25, 15. Schloss Dagstuhl. Leibniz-Zent. Inform., Wadern, 2017.

\bibitem[CSEH07]{MR2279866}
David Cohen-Steiner, Herbert Edelsbrunner, and John Harer.
\newblock Stability of persistence diagrams.
\newblock {\em Discrete Comput. Geom.}, 37(1):103--120, 2007.

\bibitem[CSEH09]{Cohen-Steiner2009}
David Cohen-Steiner, Herbert Edelsbrunner, and John Harer.
\newblock Extending persistence using {P}oincar\'{e} and {L}efschetz duality.
\newblock {\em Found. Comput. Math.}, 9(1):79--103, 2009.

\bibitem[Cur14]{MR3259939}
Justin~Michael Curry.
\newblock {\em Sheaves, cosheaves and applications}.
\newblock PhD thesis, University of Pennsylvania, 2014.

\bibitem[DeS21]{DeSha-2021}
Jordan DeSha.
\newblock {\em Inverse Problems for Topological Summaries in Topological Data
  Analysis}.
\newblock PhD thesis, University at Albany, State University of New York, 2021.

\bibitem[dSMP16]{MR3505333}
Vin de~Silva, Elizabeth Munch, and Amit Patel.
\newblock Categorified {R}eeb graphs.
\newblock {\em Discrete Comput. Geom.}, 55(4):854--906, 2016.

\bibitem[DW07]{MR2352705}
Tamal~K. Dey and Rephael Wenger.
\newblock Stability of critical points with interval persistence.
\newblock {\em Discrete Comput. Geom.}, 38(3):479--512, 2007.

\bibitem[ELZ02]{MR1949898}
Herbert Edelsbrunner, David Letscher, and Afra Zomorodian.
\newblock Topological persistence and simplification.
\newblock {\em Discrete Comput. Geom.}, 28(4):511--533, 2002.

\bibitem[Gab72]{MR332887}
Peter Gabriel.
\newblock Unzerlegbare {D}arstellungen. {I}.
\newblock {\em Manuscripta Math.}, 6:71--103; correction, ibid. 6 (1972), 309,
  1972.

\bibitem[{Gui}16]{2016arXiv160307876G}
St{\'e}phane {Guillermou}.
\newblock {The three cusps conjecture}.
\newblock {\em arXiv e-prints}, March 2016.

\bibitem[HKLM19]{MR3988214}
Shaun Harker, Miroslav Kram\'{a}r, Rachel Levanger, and Konstantin Mischaikow.
\newblock A comparison framework for interleaved persistence modules.
\newblock {\em J. Appl. Comput. Topol.}, 3(1-2):85--118, 2019.

\bibitem[KS90]{MR1074006}
Masaki Kashiwara and Pierre Schapira.
\newblock {\em Sheaves on manifolds}, volume 292 of {\em Grundlehren der
  mathematischen Wissenschaften [Fundamental Principles of Mathematical
  Sciences]}.
\newblock Springer-Verlag, Berlin, 1990.
\newblock With a chapter in French by Christian Houzel.

\bibitem[KS18]{MR3873181}
Masaki Kashiwara and Pierre Schapira.
\newblock Persistent homology and microlocal sheaf theory.
\newblock {\em J. Appl. Comput. Topol.}, 2(1-2):83--113, 2018.

\bibitem[Les15]{Lesnick2015}
Michael Lesnick.
\newblock The theory of the interleaving distance on multidimensional
  persistence modules.
\newblock {\em Found. Comput. Math.}, 15(3):613--650, 2015.

\bibitem[MP20]{MR4099800}
Alex McCleary and Amit Patel.
\newblock Bottleneck stability for generalized persistence diagrams.
\newblock {\em Proc. Amer. Math. Soc.}, 148(7):3149--3161, 2020.

\bibitem[Sco20]{scoccola-2020}
Luis Scoccola.
\newblock {\em Locally Persistent Categories And Metric Properties Of
  Interleaving Distances}.
\newblock PhD thesis, Western University, 2020.

\bibitem[Spa81]{MR666554}
Edwin~H. Spanier.
\newblock {\em Algebraic topology}.
\newblock Springer-Verlag, New York-Berlin, 1981.
\newblock Corrected reprint.

\bibitem[{Sta}24]{stacks-project}
The {Stacks project authors}.
\newblock The stacks project.
\newblock \url{https://stacks.math.columbia.edu}, 2024.

\bibitem[{tom}08]{MR2456045}
Tammo {tom Dieck}.
\newblock {\em Algebraic topology}.
\newblock EMS Textbooks in Mathematics. European Mathematical Society (EMS),
  Z\"urich, 2008.

\end{thebibliography}

\appendix

\newpage

\section{Stable Functors on the Strip $\M$}
\label{sec:stableFun}

In this \cref{sec:stableFun}
we state a categorical result,
that we use in \cref{sec:constr_pers}
to assemble the Mayer--Vietoris pyramids into a single strip-shaped diagram.
For a proof of this \cref{prp:fundaExt}
we refer to \mbox{\cite[Proposition A.14]{2021arXiv210809298B}}.

Let $\mathcal{A}$ be an additive (or pointed) category (the relevant case for us being $\mathcal{A}=\mathrm{vect}_K^{\Z}$, the graded vector spaces)
and
let $\Sigma$ be an automorphism of $\mathcal{A}$.

\begin{definition}
  A functor
  $F \colon \M \rightarrow \mathcal{A}$
  vanishing on $\partial \M$
  is \emph{strictly stable} if
  $F \circ T = \Sigma \circ F$.
\end{definition}

Now let $F \colon \M \rightarrow \mathcal{A}$
be a strictly stable functor vanishing on $\partial \M$,
let $D$ be a convex subposet of $\M$ that is a fundamental domain
with respect to the action of $\langle T \rangle$,
and let $F' := F |_D$.
We set
\begin{equation}
  \label{dfn:R_D}
  R_D :=
  \set[\big]{
  (v, w) \in D \times T(D) \given
   v \preceq w \preceq T(v)
  }
  .
\end{equation}
If we view $R_D$ as a subposet of $D \times T(D)$ with the product order,
we obtain the two functors
$F' \circ \op{pr}_1 = F \circ \op{pr}_1$ and
$\Sigma \circ F' \circ T^{-1} \circ \op{pr}_2 = F \circ \op{pr}_2$,
where $\op{pr}_1 \colon R_D \rightarrow D$
and $\op{pr}_2 \colon R_D \rightarrow T(D)$
are the projections to the first and the second component, respectively.
We consider a natural transformation $\partial (F, D)$
as in the diagram
\begin{equation}
  \label{eq:strictlyStableRestrictionBoundary}
  \begin{tikzcd}
    R_D
    \arrow[rrr, "\op{pr}_1"]
    \arrow[ddd, "\op{pr}_2"']
    &[-23pt] & &[-18pt]
    D
    \arrow[ddd, "F"]
    \\[-17pt]
    & &
    {}
    \arrow[ld, Rightarrow, "{\partial (F, D)}"]
    \\
    &
    {}
    \\[-20pt]
    T(D)
    \arrow[rrr, "F"']
    & & &
    \mathcal{A}
  \end{tikzcd}
\end{equation}
by letting
\begin{equation}
  \partial (F, D)
  \colon
  F \circ \op{pr}_1 \Rightarrow F \circ \op{pr}_2,
  (v, w) \mapsto F(v \preceq w)
  .
\end{equation}

As it turns out, $F$ is determined by
its restriction $F |_D$ and
the natural transformation $\partial(F, D)$.
To give an explicit expression, we will use
${w \preceq T(v) \in T^{n+1}(D)}$
as a shorthand for
${w, T(v) \in T^{n+1}(D)}$, $n \in \Z$,
and $w \preceq T(v)$.

\begin{lem}
  \label{lem:reduction}
  Suppose $\partial' := \partial (F, D)$.
  Then
  \begin{equation}
    \label{eq:constr_ext}
    F(v \preceq w) =
    \begin{cases}
      (\Sigma^n \circ F' \circ T^{-n})(v \preceq w)
      & v, w \in T^n(D)
      \\
      (\Sigma^n \circ \partial')_{
        \left(T^{-n} (v), T^{-n} (w)\right)
      }
      &
      w \preceq T(v) \in T^{n+1}(D)
      \\
      0
      & \text{otherwise}
    \end{cases}
  \end{equation}
  for all $v \preceq w \in \M$.
\end{lem}

By the following \cref{prp:fundaExt},
which we prove in 
\mbox{\cite[Proposition A.14]{2021arXiv210809298B}},
there is an inverse construction
to this restriction of a strictly stable functor
${F \colon \M \rightarrow \mathcal{A}}$
to the functor
${F |_D \colon \M \rightarrow \mathcal{A}}$
defined on the fundamental domain ${D \subset \M}$
together with a natrual transformation
as in the diagram \eqref{eq:strictlyStableRestrictionBoundary}.

\begin{prp}[{\cite[Proposition A.14]{2021arXiv210809298B}}]
  \label{prp:fundaExt}
  For any functor $F' \colon D \rightarrow \mathcal{A}$
  vanishing on $D \cap \partial \M$
  together with a natural transformation
  \begin{equation*}
    \begin{tikzcd}
      R_D
      \arrow[rrr, "\op{pr}_1"]
      \arrow[ddd, "\op{pr}_2"']
      &[-23pt] & &[-18pt]
      D
      \arrow[ddd, "F'"]
      \\[-17pt]
      & &
      {}
      \arrow[ld, Rightarrow, "{\partial'}"]
      \\
      &
      {}
      \\[-20pt]
      T(D)
      \arrow[rrr, "\Sigma \circ F' \circ T^{-1}"']
      & & &
      \mathcal{A}
      ,
    \end{tikzcd}
  \end{equation*}
  there is a unique strictly stable functor
  $F \colon \M \rightarrow \mathcal{A}$
  with
  \[F |_D = F', \quad
    F |_{\partial \M} = 0, \quad \text{and} \quad
    \partial (F, D) = \partial' .\]
\end{prp}

\section{Constructing Relative Interlevel Set Homology}
\label{sec:constr_pers}

In this section, we use the formalism of strictly stable functors on $\M$ developed in \cref{sec:stableFun} to construct relative interlevel set homology in a way that yields a functor by construction.
For computations using long exact sequences in homology, a generalization
of the relative interlevel set homology for a function $f \colon X \to \R$ considered on a pair $A \subseteq X$ of spaces will be useful.
For such considerations we will use an intersection-like operation on pairs $A \subseteq X$ of spaces.
We note that this operation turns out to be different from the componentwise intersection, which is the meet operation in the lattice $(2^X \times 2^X, \subseteq)$.

\begin{definition}
We say that a pair of sets $(X, A)$ is \emph{admissible}
if $A \subseteq X$.
For two admissible pairs of sets $(X, A)$ and $(Y, B)$, we define
  \(
  (X, A) \smashcap (Y, B)
  \)
  to be the pair
  \(
  {(X \cap Y, (A \cap Y) \cup (X \cap B))}
  .
  \)
\end{definition}

Substituting intersections by products,
the above definition
becomes very similar to a well known notion of a product of pairs of spaces
used in a relative version of the Künneth Theorem,
see for example \cite[Section 5.3]{MR666554}.
Moreover, if we view the partial order on sets given by inclusions
as the structure of a category, then the intersection is
the categorical product.

\begin{lem}[Closedness and Commutativity]
  For two admissible pairs of sets $(X, A)$ and $(Y, B)$, the pair $(X, A) \smashcap (Y, B)$ is admissible,
  and we have
  $(X, A) \smashcap (Y, B) =
  (Y, B) \smashcap (X, A).$
\end{lem}

Let $(Z, C)$ be another admissible pair.

\begin{lem}
  \label{lem:smashcap_dist}
  Consider an admissible pair of sets $(X, A)$.
  The operation $(X, A) \smashcap -$ preserves
  componentwise unions and intersections
  of admissible pairs.
  In other words, for two further admissible pairs of sets $(Y, B)$ and $(Z, C)$, we have the equation
  \begin{equation*}
    (X, A) \smashcap (Y \cup Z, B \cup C)
    =
    ((X, A) \smashcap (Y, B)) \cup ((X, A) \smashcap (Z, C)) ,
  \end{equation*}
  and similarly for intersections.
  Here, the union on the right hand side of the above equation
  is to be understood as a componentwise union.
\end{lem}

Now let $(X, A)$ be a finite simplicial pair and
$f \colon X \rightarrow \R$ be a piecewise linear map.
In the following we construct the
\emph{relative interlevel set homology}
$h(X, A; f)$
of $f \colon A \subseteq X \rightarrow \R$.
As we now have a relative part $A$ as well,
this is an extension of our construction from \cref{sec:InterlevelHom}.
In order to construct the functor $h(X, A; f)$,
we use a similar procedure as in \cite[Section 2]{2021arXiv210809298B}.
To this end, we set
\begin{equation*}
  F' \colon D \rightarrow \mathrm{vect}_K^{\Z},
  u \mapsto H_{\bullet} ((X, A) \smashcap f^{-1}(\rho(u)); K)
  ,
\end{equation*}
where $\mathrm{vect}_K^{\Z}$ is the category
of finite-dimensional graded vector space over $K$.
We construct $h(X, A; f) \colon \M \rightarrow \mathrm{vect}_K$
in such a way that it carries the same (and more) information as $F'$
with precomposition by $T$ taking the place of the degree shift
\begin{equation*}
  \Sigma \colon \mathrm{vect}_K^{\Z} \rightarrow \mathrm{vect}_K^{\Z},
  M_{\bullet} \mapsto M_{\bullet - 1}
  .
\end{equation*}
As an intermediate step,
we extend $F'$ to a functor
\begin{equation*}
  F \colon \M \rightarrow \mathrm{vect}_K^{\Z}
  ,
\end{equation*}
which is $\Z$-equivariant or strictly stable in the sense that
\begin{equation}
  \label{eq:strictlyStable}
  F \circ T = \Sigma \circ F
  .
\end{equation}
Now as a map into the objects of $\mathrm{vect}_K^{\Z}$
such a functor $F$ carries no new information
in comparison to $F'$.
Moreover, by \eqref{eq:strictlyStable} most of the information
carried by $F$ is redundant, and we discard all redundant information
by post-composition with the projection
\begin{equation*}
  \op{pr}_0 \colon \mathrm{vect}_K^{\Z} \rightarrow \mathrm{vect}_K, \,
  M_{\bullet} \mapsto M_0
  .
\end{equation*}

\begin{figure}[t]
  \centering
\begingroup%
  \makeatletter%
  \providecommand\color[2][]{%
    \errmessage{(Inkscape) Color is used for the text in Inkscape, but the package 'color.sty' is not loaded}%
    \renewcommand\color[2][]{}%
  }%
  \providecommand\transparent[1]{%
    \errmessage{(Inkscape) Transparency is used (non-zero) for the text in Inkscape, but the package 'transparent.sty' is not loaded}%
    \renewcommand\transparent[1]{}%
  }%
  \providecommand\rotatebox[2]{#2}%
  \newcommand*\fsize{\dimexpr\f@size pt\relax}%
  \newcommand*\lineheight[1]{\fontsize{\fsize}{#1\fsize}\selectfont}%
  \ifx\svgwidth\undefined%
    \setlength{\unitlength}{300bp}%
    \ifx\svgscale\undefined%
      \relax%
    \else%
      \setlength{\unitlength}{\unitlength * \real{\svgscale}}%
    \fi%
  \else%
    \setlength{\unitlength}{\svgwidth}%
  \fi%
  \global\let\svgwidth\undefined%
  \global\let\svgscale\undefined%
  \makeatother%
  \begin{picture}(1,0.5)%
    \lineheight{1}%
    \setlength\tabcolsep{0pt}%
    \put(0,0){\includegraphics[width=\unitlength,page=1]{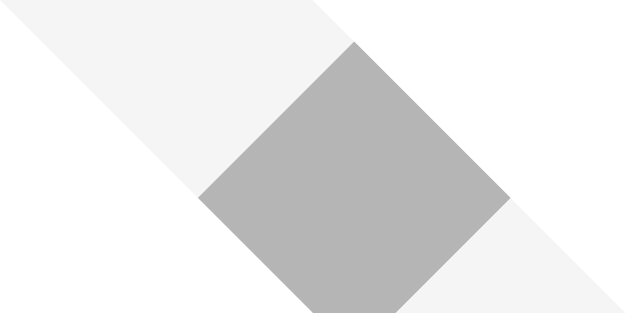}}%
    \put(0.3625,0.35833325){\makebox(0,0)[t]{\lineheight{1.25}\smash{\begin{tabular}[t]{c}$\hat{u}$\end{tabular}}}}%
    \put(0.26875,0.4125){\makebox(0,0)[t]{\lineheight{1.25}\smash{\begin{tabular}[t]{c}$\Sigma(D)$\end{tabular}}}}%
    \put(0.5125,0.225){\makebox(0,0)[t]{\lineheight{1.25}\smash{\begin{tabular}[t]{c}$w$\end{tabular}}}}%
    \put(0.68333325,0.25){\makebox(0,0)[t]{\lineheight{1.25}\smash{\begin{tabular}[t]{c}$v_2$\end{tabular}}}}%
    \put(0.49166675,0.09166675){\makebox(0,0)[t]{\lineheight{1.25}\smash{\begin{tabular}[t]{c}$v_1$\end{tabular}}}}%
    \put(0.67083325,0.09583325){\makebox(0,0)[t]{\lineheight{1.25}\smash{\begin{tabular}[t]{c}$u$\end{tabular}}}}%
    \put(0,0){\includegraphics[width=\unitlength,page=2]{constr-boundary.pdf}}%
    \put(0.7809475,0.2590525){\makebox(0,0)[t]{\lineheight{1.25}\smash{\begin{tabular}[t]{c}$l_1$\end{tabular}}}}%
    \put(0.226675,0.223325){\makebox(0,0)[t]{\lineheight{1.25}\smash{\begin{tabular}[t]{c}$l_0$\end{tabular}}}}%
  \end{picture}%
\endgroup%

  \caption{
    The axis-aligned rectangle $u \preceq v_1, v_2 \preceq w \in D$
    determined by $(w, \hat{u}) \in R_D$.
  }
  \label{fig:constrBoundary}
\end{figure}

Now in order to obtain such a strictly stable functor $F$
from $F'$ we need to glue consecutive layers
using connecting homomorphisms.
To this end, let
\[R_D :=
  \{
  (w, \hat{u}) \in D \times T(D) \mid
  w \preceq \hat{u} \preceq T(w)
  \}
\]
as in \cref{dfn:R_D}. %
As shown in \cref{fig:constrBoundary}, any pair $(w, \hat{u}) \in R_D$
determines an axis-aligned rectangle $u \preceq v_1, v_2 \preceq w \in D$
with $T(u) = \hat{u}$.
Moreover,
as this rectangle is contained in $D$,
the corresponding join $w = v_1 \vee v_2$ and meet $u = v_1 \wedge v_2$
are preserved by $\rho$ by \cref{prp:rho}.3.
Furthermore, since taking preimages is a homomorphism of boolean algebras,
$f^{-1}$ also preserves joins and meets,
which in this case are the componentwise unions and intersections.
Finally, by \cref{lem:smashcap_dist}, joins and meets are also preserved by $(X, A) \smashcap -$.
This means
that $(X, A) \smashcap f^{-1} (\rho(u))$ is the componentwise intersection
of $(X, A) \smashcap f^{-1} (\rho(v_1))$
and $(X, A) \smashcap f^{-1} (\rho(v_2))$,
while $(X, A) \smashcap f^{-1} (\rho(w))$ is their union.
As $f$ is piecewise linear and $A$ a subcomplex of $X$, the triad
$(X, A) \smashcap f^{-1} (\rho(w); \rho(v_1), \rho(v_2))$
of pairs is excisive in each component.
Thus, we have the boundary map
\begin{equation*}
  \partial'_{(w, \hat{u})} \colon
  H_{\bullet} ((X, A) \smashcap f^{-1}(\rho(w)); K)
  \rightarrow
  H_{\bullet-1} ((X, A) \smashcap f^{-1}(\rho(u)); K)
\end{equation*}
of the corresponding Mayer--Vietoris sequence
as described by \cite[Theorem 10.7.7]{MR2456045}.
Now let
$\op{pr}_1 \colon R_D \rightarrow D$
and $\op{pr}_2 \colon R_D \rightarrow T(D)$
be the projections to the first and the second component, respectively.
Then we have functors
$F' \circ \op{pr}_1$ and
$\Sigma \circ F' \circ T^{-1} \circ \op{pr}_2$
from $R_D$ to $\mathrm{vect}_K^{\Z}$
and $\partial'$ is a natural transformation
\begin{equation*}
  \partial' \colon
  F' \circ \op{pr}_1 \Rightarrow
  \Sigma \circ F' \circ T^{-1} \circ \op{pr}_2
  .
\end{equation*}
Thus,
the functor $F' \colon D \rightarrow \mathrm{vect}_K^{\Z}$
and $\partial'$ determine a unique strictly stable functor
$F \colon \M \rightarrow \mathrm{vect}_K^{\Z}$
by \cref{prp:fundaExt}. %
\sloppy
To obtain a functor of type
${\M \rightarrow \mathrm{vect}_K}$ from $F$ we post-compose
${F \colon \M \rightarrow \mathrm{vect}_K^{\Z}}$
with the projection
$\op{pr}_0 \colon \mathrm{vect}_K^{\Z} \rightarrow \mathrm{vect}_K$
and we define
\begin{equation*}
  h(X, A; f) := h(X, A; f; K) := \op{pr}_0 \circ \, F
  .
\end{equation*}

\section{Structure of Relative Interlevel Set Homology}
\label{sec:struct}

In the following we describe the structure
of relative interlevel set homology
in order to show that our \cref{dfn:diagramFunction}
of the extended persistence diagram
is indeed equivalent to the original definition from \cite{Cohen-Steiner2009}.
This relation to extended persistence is also discussed in
\cite{Carlsson:2009:ZPH:1542362.1542408,Bendich-2013}.
Moreover, the authors of
\cite{Bendich-2013} have shown
that the restrictions of relative interlevel set homology
to each individual tile (called \emph{pyramid} in that article)
can be decomposed into restricted blocks.
We show that the relative interlevel set homology
decomposes into the following type
of rectangle-shaped indecomposables, see also \cref{fig:block}.

\begin{definition}[Block]
  \label{dfn:block}
  For $v \in \op{int} \M$ we define
  \begin{equation*}
    B^v \colon \M \rightarrow \mathrm{Vect}_K,\,
    w \mapsto
    \begin{cases}
      K & w \in (\uparrow v) \cap \op{int} (\downarrow T(v))
      \\
      \{0\} & \text{otherwise}
      ,
    \end{cases}
  \end{equation*}
  where $\op{int} (\downarrow T(v))$ is the interior
  of the downset of $T(v)$ in $\M$.
  The internal maps are identities whenever both domain and codomain are $K$,
  otherwise they are zero.
\end{definition}

\begin{figure}[t]
  \centering
\begingroup%
  \makeatletter%
  \providecommand\color[2][]{%
    \errmessage{(Inkscape) Color is used for the text in Inkscape, but the package 'color.sty' is not loaded}%
    \renewcommand\color[2][]{}%
  }%
  \providecommand\transparent[1]{%
    \errmessage{(Inkscape) Transparency is used (non-zero) for the text in Inkscape, but the package 'transparent.sty' is not loaded}%
    \renewcommand\transparent[1]{}%
  }%
  \providecommand\rotatebox[2]{#2}%
  \newcommand*\fsize{\dimexpr\f@size pt\relax}%
  \newcommand*\lineheight[1]{\fontsize{\fsize}{#1\fsize}\selectfont}%
  \ifx\svgwidth\undefined%
    \setlength{\unitlength}{270bp}%
    \ifx\svgscale\undefined%
      \relax%
    \else%
      \setlength{\unitlength}{\unitlength * \real{\svgscale}}%
    \fi%
  \else%
    \setlength{\unitlength}{\svgwidth}%
  \fi%
  \global\let\svgwidth\undefined%
  \global\let\svgscale\undefined%
  \makeatother%
  \begin{picture}(1,0.5)%
    \lineheight{1}%
    \setlength\tabcolsep{0pt}%
    \put(0,0){\includegraphics[width=\unitlength,page=1]{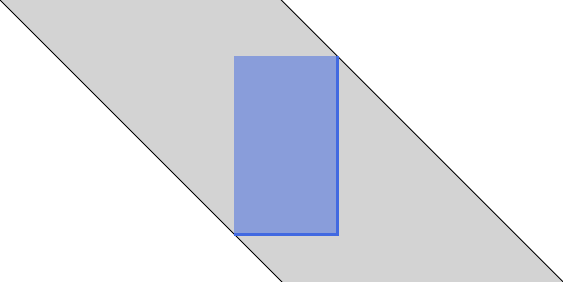}}%
    \put(0.30416667,0.33333333){\makebox(0,0)[t]{\lineheight{1.25}\smash{\begin{tabular}[t]{c}$\{0\}$\end{tabular}}}}%
    \put(0.725,0.125){\makebox(0,0)[t]{\lineheight{1.25}\smash{\begin{tabular}[t]{c}$\{0\}$\end{tabular}}}}%
    \put(0.50833333,0.22916667){\makebox(0,0)[t]{\lineheight{1.25}\smash{\begin{tabular}[t]{c}$K$\end{tabular}}}}%
    \put(0.3875,0.42083333){\makebox(0,0)[t]{\lineheight{1.25}\smash{\begin{tabular}[t]{c}$T(v)$\end{tabular}}}}%
    \put(0.61666667,0.05833333){\makebox(0,0)[t]{\lineheight{1.25}\smash{\begin{tabular}[t]{c}$v$\end{tabular}}}}%
    \put(0,0){\includegraphics[width=\unitlength,page=2]{block.pdf}}%
    \put(0.77297806,0.26035528){\makebox(0,0)[t]{\lineheight{1.25}\smash{\begin{tabular}[t]{c}$l_1$\end{tabular}}}}%
    \put(0.29691611,0.16141722){\makebox(0,0)[t]{\lineheight{1.25}\smash{\begin{tabular}[t]{c}$l_0$\end{tabular}}}}%
  \end{picture}%
\endgroup%

  \caption{
    The indecomposable $B^v \colon \M \rightarrow \mathrm{Vect}_K$.
  }
  \label{fig:block}
\end{figure}

To this end, we make use of a structure theorem
from \mbox{\cite[Theorem 3.21]{2021arXiv210809298B}},
which applies to certain functors from $\M$ to $\mathrm{vect}_K$
and in particular to the relative interlevel set homology
$h(f) \colon \M \rightarrow \mathrm{vect}_K$
of a piecewise linear function $f \colon X \rightarrow \R$
on a finite simplicial complex $X$.
More specifically, our structure theorem assumes
that the following two notions apply to
$h(f) \colon \M \rightarrow \mathrm{vect}_K$.

\begin{definition}
  \label{dfn:cont}
  We say that a functor
  $F \colon \M \rightarrow \mathrm{Vect}_K$
  is \emph{sequentially continuous}
  if for any decreasing sequence
  $(u_k)_{k=1}^{\infty}$ in $\M$
  converging to $u$
  the natural map
  \begin{equation*}
    F(u) \rightarrow \varprojlim_{k} F(u_k)
  \end{equation*}
  is an isomorphism;
  see also \cite[Definition 2.4]{2021arXiv210809298B}.
\end{definition}

As $f \colon X \rightarrow \R$
is simplexwise linear on some finite triangulation of $X$
and as singular homology is homotopy invariant,
the relative interlevel set homology
$h(f) \colon \M \rightarrow \mathrm{vect}_K$
is indeed sequentially continuous.

\begin{definition}
  \label{dfn:homological}
  We say that a functor $F \colon \M \rightarrow \mathrm{Vect}_K$
  vanishing on $\partial \M$ is \emph{homological}
  if for any axis-aligned rectangle
  with one corner lying on $l_1$ and the other corners
  ${u \preceq v \preceq w \in \M}$,
  the long sequence
  \begin{equation*}
    \qquad
    \begin{tikzcd}
      &
      \cdots
      \arrow[r]
      &
      F(T^{-1}(w))
      \arrow[dll, out=0, in=180, looseness=2, overlay]
      \\
      F(u)
      \arrow[r]
      &
      F(v)
      \arrow[r]
      &
      F(w)
      \arrow[dll, out=0, in=180, looseness=2, overlay]
      \\
      F(T(u))
      \arrow[r]
      &
      \cdots
    \end{tikzcd}
  \end{equation*}
  is exact;
  see also
  \cref{fig:homological} or
  \cite[Definition C.3]{2021arXiv210809298B}.
\end{definition}

\begin{figure}[t]
  \centering
  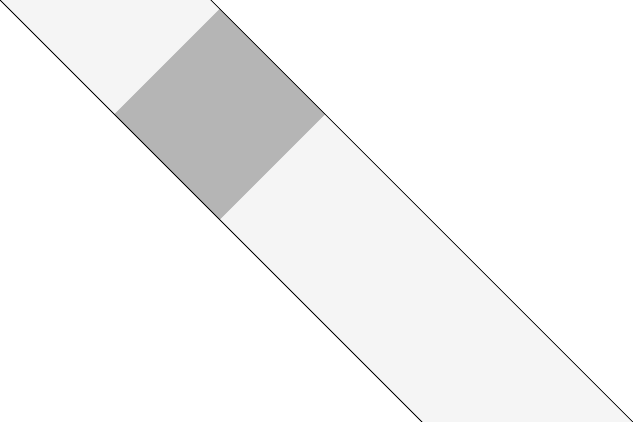
  \caption{
    The linear subposet given by the orbits of $u$, $v$, and $w$.
  }
  \label{fig:homological}
\end{figure}

In \cite[Proposition C.4]{2021arXiv210809298B} we provide
several equivalent characterizations of homological functors.
We will shortly use the following one.

\begin{prp}
\label{prp:homological}
A functor $F \colon \M \rightarrow \mathrm{Vect}_K$
vanishing on $\partial \M$
is homological iff
for any axis-aligned rectangle
$u \preceq v_1, v_2 \preceq w \in D$
as shown in \cref{fig:constrBoundary}
the long sequence
\begin{equation}
\label{eq:charHomological}
\qquad ~~
\begin{tikzcd}
  &
  \cdots
  \arrow[r]
  &
  F(T^{-1}(w))
  \arrow[dll, out=0, in=180, looseness=2, overlay]
  \\
  F(u)
  \arrow[r]
  &
  F(v_1) \oplus F(v_2)
  \arrow[r, "(1 ~~ -1)"]
  &
  F(w)
  \arrow[dll, out=0, in=180, looseness=2, overlay]
  \\
  F(T(u))
  \arrow[r]
  &
  \cdots
\end{tikzcd}  
\end{equation}
is exact.
\end{prp}

Now if we set $F := h(f) \colon \M \rightarrow \mathrm{vect}_K$,
then the long sequence 
\eqref{eq:charHomological}
is a Mayer--Vietoris sequence by the construction
of $h(f)$ and hence exact.
Thus, $h(f) \colon \M \rightarrow \mathrm{vect}_K$ is homological.
All in all, the relative interlevel set homology
$h(f) \colon \M \rightarrow \mathrm{vect}_K$
satisfies the assumptions of the following theorem
from \cite[Theorem 3.21]{2021arXiv210809298B}.

\begin{thm}
  \label{thm:decomp}
  Any sequentially continuous homological functor
  $F \colon \M \rightarrow \mathrm{vect}_K$
  decomposes as
  \begin{equation*}
    F \cong \bigoplus_{v \in \op{int} \M} (B^v)^{\oplus \nu(v)}
    ,
  \end{equation*}
  where $\nu := \op{Dgm}(F)$.
\end{thm}

In particular
the relative interlevel set homology of $f \colon X \rightarrow \R$
is completely classified by the extended persistence diagram
$\op{Dgm}(f)$.

\subsection{Equivalence of Definitions for Extended Persistence Diagrams}
\label{sec:equivalenceOfDefs}
  
We now explain how this implies
that our definition of the extended persistence diagram
is equivalent to the definition from \cite{Cohen-Steiner2009}.
This discussion is closely related to \cite[Section 3.2]{2021arXiv210809298B},
where we also describe a connection to the level set barcode.
We consider \cref{fig:relExtended}
and the restriction of $h(f) \colon \M \rightarrow \mathrm{vect}_K$
to the subposet of $\M$,
which is shaded in blue in this figure.
Here each point on the vertical blue line segment
to the upper left is assigned the homology space in degree $0$
of a sublevel set of $f$, of $X$,
or of a pair with $X$ as the first component and
a superlevel set as the second component.
Up to isomorphism of posets,
this is the extended persistent homology
of $f \colon X \rightarrow \R$ in degree $0$.
Similarly, any point on the horizontal blue line segment in the center
is assigned the homology space of some pair of preimages in degree $1$
and any point on the vertical blue line at the lower right
is assigned the homology of some pair in degree $2$.
By \cref{thm:decomp} the relative interlevel set homology $h(f)$ decomposes
into blocks as in \cref{dfn:block}.
Now the support of each such block
intersects exactly one of these blue line segments.
We focus on the horizontal blue line segment in the center
of \cref{fig:relExtended},
which carries the extended persistent homology
of $f \colon X \rightarrow \R$ in degree $1$.
Any choice of decomposition of $h(f) \colon \M \rightarrow \mathrm{vect}_K$
yields a decomposition of its restriction
to this line segment and thus of persistent cohomology in degree $1$.
Moreover, the support of the contravariant block
assigned to any of the black dots
in \cref{fig:relExtended}
intersects this horizontal line segment.

\begin{figure}[t]
  \centering
\begingroup%
  \makeatletter%
  \providecommand\color[2][]{%
    \errmessage{(Inkscape) Color is used for the text in Inkscape, but the package 'color.sty' is not loaded}%
    \renewcommand\color[2][]{}%
  }%
  \providecommand\transparent[1]{%
    \errmessage{(Inkscape) Transparency is used (non-zero) for the text in Inkscape, but the package 'transparent.sty' is not loaded}%
    \renewcommand\transparent[1]{}%
  }%
  \providecommand\rotatebox[2]{#2}%
  \newcommand*\fsize{\dimexpr\f@size pt\relax}%
  \newcommand*\lineheight[1]{\fontsize{\fsize}{#1\fsize}\selectfont}%
  \ifx\svgwidth\undefined%
    \setlength{\unitlength}{300.51638031bp}%
    \ifx\svgscale\undefined%
      \relax%
    \else%
      \setlength{\unitlength}{\unitlength * \real{\svgscale}}%
    \fi%
  \else%
    \setlength{\unitlength}{\svgwidth}%
  \fi%
  \global\let\svgwidth\undefined%
  \global\let\svgscale\undefined%
  \makeatother%
  \begin{picture}(1,0.60396042)%
    \lineheight{1}%
    \setlength\tabcolsep{0pt}%
    \put(0,0){\includegraphics[width=\unitlength,page=1]{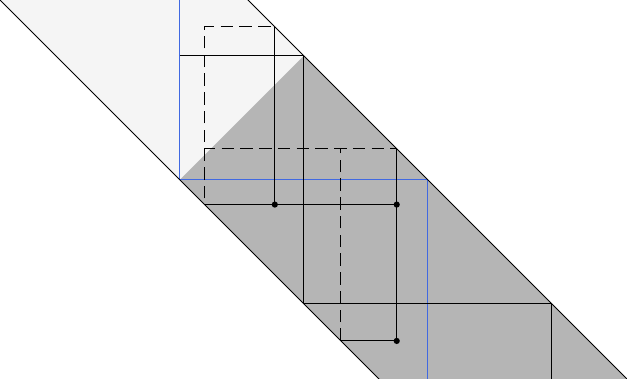}}%
    \put(0.60433228,0.42867089){\makebox(0,0)[t]{\lineheight{1.25}\smash{\begin{tabular}[t]{c}$l_1$\end{tabular}}}}%
    \put(0.16381952,0.39888691){\makebox(0,0)[t]{\lineheight{1.25}\smash{\begin{tabular}[t]{c}$l_0$\end{tabular}}}}%
  \end{picture}%
\endgroup%

  \caption{
    The subposet of $\M$ corresponding to extended persistence
    shaded in blue
    as well as three vertices contained in the domains corresponding to
    $1$-dimensional relative, extended, and ordinary persistence;
    see also \cref{fig:subdiagrams}.
  }
  \label{fig:relExtended}
\end{figure}

First assume that
the black dot on the lower right appears in
$\op{Dgm}(f)$.
Then the restriction of the associated block
to the horizontal blue line segment
is a direct summand of the persistent homology
of $f \colon X \rightarrow \R$ in degree $1$
and the intersection of its support with the blue line segment
is the life span of the corresponding feature
in the sense that the point of intersection of the right edge
marks the birth of a homology class that dies
as soon as it is mapped into the sublevel set corresponding
to the point of intersection of the left edge.
Moreover, this life span is encoded by the position of this black dot.
Now this particular black dot on the lower right
of \cref{fig:relExtended}
is contained in the triangular region labeled as $\mathrm{Ord}_1$
in \cref{fig:subdiagrams}.
Furthermore, any vertex of the extended persistence diagram $\op{Dgm}(f)$
contained in the triangular region labeled $\mathrm{Ord}_1$
describes a feature of $f \colon X \rightarrow \R$,
which is born at some sublevel set
and also dies at some sublevel set.
Thus, up to reparametrization,
the ordinary persistence diagram
of $f \colon X \rightarrow \R$ in degree $1$
is the restriction of $\op{Dgm}(f) \colon \M \rightarrow \N_0$
to the region labeled $\mathrm{Ord}_1$ in \cref{fig:subdiagrams}.

Now suppose that the black dot on the upper left in \cref{fig:relExtended}
appears in $\op{Dgm}(f)$.
Then the intersection of the support of the associated
contravariant block with the horizontal blue line segment
describes the life span of a feature which is born
at the homology of $X$ relative to some superlevel set in degree $1$
and also dies at some relative homology space.
Moreover, this is true for any vertex of $\op{Dgm}(f)$
contained in the triangular region labeled $\mathrm{Rel}_1$
in \cref{fig:subdiagrams}.
Thus, up to reparametrization,
the relative subdiagram of $f \colon X \rightarrow \R$ in degree $1$
is the restriction of $\op{Dgm}(f) \colon \M \rightarrow \N_0$
to the region labeled $\mathrm{Rel}_1$ in \cref{fig:subdiagrams}.

Finally, the black dot to the upper right in \cref{fig:relExtended}
(if in $\op{Dgm}(f)$),
or any other vertex of $\op{Dgm}(f)$
in the square region labeled $\mathrm{Ext}_1$ in \cref{fig:subdiagrams},
describes a feature,
which is born
at the homology of some sublevel set of $f \colon X \rightarrow \R$
and dies at the homology of $X$ relative to some superlevel set.
Thus,
the extended subdiagram of $f \colon X \rightarrow \R$ in degree $1$
is the restriction of $\op{Dgm}(f) \colon \M \rightarrow \N_0$
to the square region labeled $\mathrm{Ext}_1$ in \cref{fig:subdiagrams}.

As we have analogous correspondences for each line segment
of the subposet of $\M$ shaded in blue
in \cref{fig:relExtended},
we obtain a partition of the lower right part of the strip $\M$ into regions,
corresponding to ordinary, relative, and extended subdiagrams
as labeled in \cref{fig:subdiagrams}
analogous to \cite{Cohen-Steiner2009}.

\section{Canceling Boundaries from the Zig-zag Lemma}
\label{sec:anticommutingBoundaries}

In order to show \cref{lem:MVseq},
we use a result from homological algebra,
which is \cref{lem:anticommutingBoundaries} below.
Before we get to \cref{lem:anticommutingBoundaries},
we collect some basic facts on abelian categories.

\begin{lem}
  \label{lem:pushoutDirectSum}
  Suppose we have a commutative diagram
  \begin{equation*}
    \begin{tikzcd}
      &
      0
      \arrow[d]
      &[3ex]
      0
      \arrow[d]
      &[3ex]
      0
      \arrow[d]
      \\
      0
      \arrow[r]
      &
      A
      \arrow[r]
      \arrow[d]
      \arrow[dr, phantom, "\ulcorner", very near end]
      &
      B
      \arrow[r]
      \arrow[d]
      &
      C
      \arrow[r]
      \arrow[dd, dashed, bend right=20]
      &
      0
      \\[3ex]
      0
      \arrow[r]
      &
      A'
      \arrow[r]
      \arrow[d]
      &
      P
      \arrow[ru]
      \arrow[dl]
      \arrow[dr, "{\mathrm{pr}}", two heads]
      \\[3ex]
      &
      A''
      \arrow[rr, dashed, bend left=20]
      \arrow[d]
      &
      &
      P/A
      \arrow[ll, dashed, bend left=20]
      \arrow[uu, dashed, bend right=20]
      \\
      &
      0
    \end{tikzcd}
  \end{equation*}
  with the straight and the buckled rows and columns exact,
  $P$ a pushout,
  and
  ${\mathrm{pr} \colon P \rightarrow P / A}$ a cokernel as indicated.
  Moreover,
  we assume the dashed arrows are induced by the universal property
  of cokernels.
  Then the dashed arrows exhibit ${P / A}$ as a direct sum
  ${A'' \oplus C}$.
\end{lem}

\begin{proof}
  Here the essential observation is
  that the two pushout induced homomorphisms
  in diagrams
  \begin{equation*}
    \begin{tikzcd}
      A
      \arrow[r]
      \arrow[d]
      \arrow[dr, phantom, "\ulcorner", very near end]
      &
      B
      \arrow[d]
      \arrow[ddr, bend left, "0"]
      \\
      A'
      \arrow[r]
      \arrow[drr, bend right, "{\mathrm{pr} |_{A'}}"']
      &
      P
      \arrow[dr, dashed]
      \\
      &
      &
      P/A
    \end{tikzcd}
    \quad \text{and}
    \quad
    \begin{tikzcd}
      A
      \arrow[r]
      \arrow[d]
      \arrow[dr, phantom, "\ulcorner", very near end]
      &
      B
      \arrow[d]
      \arrow[ddr, bend left, "{\mathrm{pr} |_{B}}"]
      \\
      A'
      \arrow[r]
      \arrow[drr, bend right, "0"']
      &
      P
      \arrow[dr, dashed]
      \\
      &
      &
      P/A
    \end{tikzcd}
  \end{equation*}
  add up to the canonical projection
  ${\mathrm{pr} \colon P \rightarrow P / A}$.
\end{proof}

\begin{lem}
  \label{lem:sesPushoutPullback}
  Suppose we have the commutative diagram
  \begin{equation*}
    \begin{tikzcd}
      0
      \arrow[r]
      &
      A
      \arrow[rr]
      \arrow[dd]
      \arrow[dr, phantom, "\ulcorner", very near end]
      &
      &
      B
      \arrow[r]
      \arrow[dd]
      \arrow[dl]
      &
      C
      \arrow[r]
      \arrow[dd]
      &
      0
      \\[2ex]
      &
      &[4ex]
      P
      \arrow[urr, dashed]
      \arrow[dr, dashed]
      \\
      0
      \arrow[r]
      &
      A'
      \arrow[rr]
      \arrow[ur]
      \arrow[uurrr, bend right=15, "0"', near end]
      &
      &
      B'
      \arrow[r]
      &
      C'
      \arrow[r]
      &
      0
    \end{tikzcd}
  \end{equation*}
  with exact rows and $P$ a pushout as indicated.
  Then the two dashed arrows induced by the universal property of the pushout
  exhibit $P$ as a pullback:
  \begin{equation}
    \label{eq:pullbackPCB'C'}
    \begin{tikzcd}
      P
      \arrow[d]
      \arrow[r]
      \arrow[dr, phantom, "\lrcorner", very near start]
      &
      C
      \arrow[d]
      \\
      B'
      \arrow[r]
      &
      C'.
    \end{tikzcd}
  \end{equation}
\end{lem}

\begin{proof}
  This follows from the five lemma.
\end{proof}

\begin{lem}
  \label{lem:sesPushout3by3}
  Suppose we have a commutative diagram
  \begin{equation}
    \label{eq:3by3ses}
    \begin{tikzcd}
      &
      0
      \arrow[d]
      &
      &
      0
      \arrow[d]
      &
      0
      \arrow[d]
      \\
      0
      \arrow[r]
      &
      A
      \arrow[rr]
      \arrow[dd]
      \arrow[dr, phantom, "\ulcorner", very near end]
      &
      &
      B
      \arrow[r]
      \arrow[dd]
      \arrow[dl]
      &
      C
      \arrow[r]
      \arrow[dd]
      &
      0
      \\[2ex]
      &
      &[4ex]
      P
      \arrow[dr, dashed]
      \\
      0
      \arrow[r]
      &
      A'
      \arrow[rr]
      \arrow[d]
      \arrow[ur]
      &
      &
      B'
      \arrow[r]
      \arrow[d]
      &
      C'
      \arrow[r]
      \arrow[d]
      &
      0
      \\
      0
      \arrow[r]
      &
      A''
      \arrow[rr]
      \arrow[d]
      &
      &
      B''
      \arrow[r]
      \arrow[d]
      &
      C''
      \arrow[r]
      \arrow[d]
      &
      0
      \\
      &
      0
      &
      &
      0
      &
      0
    \end{tikzcd}
  \end{equation}
  with exact rows and columns and $P$ a pushout as indicated.
  Then the sequence
  \begin{equation*}
    0 \rightarrow P \rightarrow B' \rightarrow C \rightarrow 0
  \end{equation*}
  is exact.
\end{lem}

\begin{proof}
  This follows from \cref{lem:sesPushoutPullback}
  and the fact that the dashed arrow in \eqref{eq:3by3ses}
  is one of the two legs of the corresponding pullback diagram
  \eqref{eq:pullbackPCB'C'}.
\end{proof}

\begin{lem}
  \label{lem:anticommutingBoundaries}
  Suppose we have the commutative diagram
  \begin{equation}
    \label{eq:3x3chainComplexes}
    \begin{tikzcd}
      &
      0
      \arrow[d]
      &
      0
      \arrow[d]
      &
      0
      \arrow[d]
      \\
      0
      \arrow[r]
      &
      A
      \arrow[r]
      \arrow[d]
      &
      B
      \arrow[r]
      \arrow[d]
      &
      C
      \arrow[r]
      \arrow[d]
      &
      0
      \\
      0
      \arrow[r]
      &
      A'
      \arrow[r]
      \arrow[d]
      &
      B'
      \arrow[r]
      \arrow[d]
      &
      C'
      \arrow[r]
      \arrow[d]
      &
      0
      \\
      0
      \arrow[r]
      &
      A''
      \arrow[r]
      \arrow[d]
      &
      B''
      \arrow[r]
      \arrow[d]
      &
      C''
      \arrow[r]
      \arrow[d]
      &
      0
      \\
      &
      0
      &
      0
      &
      0
    \end{tikzcd}
  \end{equation}
  of chain complexes of vector spaces over $K$
  (or modules over some ring)
  with levelwise exact rows and columns.
  Moreover,
  let ${n \in \Z}$ and let
  \begin{align*}
    \partial \colon H_n(C) & \rightarrow H_{n-1}(A),
    \\
    \partial_A \colon H_n(A'') & \rightarrow H_{n-1}(A),
    \\
    \partial'' \colon H_{n+1}(C'') & \rightarrow H_{n}(A''),
    \\
    \text{and}
    \quad
    \partial_C \colon H_{n+1}(C'') & \rightarrow H_{n}(C)
  \end{align*}
  be the corresponding boundary maps obtained from the zig-zag lemma
  when applied to the outermost exact rows and columns
  of \eqref{eq:3x3chainComplexes}.
  Then we have ${\partial_A \circ \partial'' = -\partial \circ \partial_C}$;
  or in other words,
  the square
  \begin{equation*}
    \begin{tikzcd}[row sep=2ex, column sep=2ex]
      H_{n+1}(C'')
      \arrow[dd, "\partial''"']
      \arrow[rr, "\partial_C"]
      &
      &
      H_{n}(C)
      \arrow[dd, "\partial"]
      \\
      &
      -
      \\
      H_{n}(A'')
      \arrow[rr, "\partial_A"']
      &
      &
      H_{n-1}(A)
    \end{tikzcd}
  \end{equation*}
  is anti-commutative.
\end{lem}

\begin{proof}
  First we form the pushout
  \begin{equation*}
    \begin{tikzcd}
      A
      \arrow[r]
      \arrow[d]
      \arrow[dr, phantom, "\ulcorner", very near end]
      &
      B
      \arrow[d]
      \\
      A'
      \arrow[r]
      &
      P
      .
    \end{tikzcd}
  \end{equation*}
  Now the universal property of the pushout $P$ yields
  commutative diagrams
  \begin{equation}
    \label{eq:extDown}
    \begin{tikzcd}
      0
      \arrow[r]
      &
      A
      \arrow[r]
      \arrow[d]
      &
      B
      \arrow[r]
      \arrow[d]
      &
      C
      \arrow[r]
      \arrow[d, equal]
      &
      0
      \\
      0
      \arrow[r]
      &
      A'
      \arrow[r]
      \arrow[rr, "0"', bend right]
      &
      P
      \arrow[r]
      &
      C
      \arrow[r]
      &
      0
    \end{tikzcd}
  \end{equation}
  and
  \begin{equation}
    \label{eq:extRight}
    \begin{tikzcd}
      0
      \arrow[d]
      &
      0
      \arrow[d]
      \\
      A
      \arrow[r]
      \arrow[d]
      &
      B
      \arrow[d]
      \arrow[dd, "0", bend left]
      \\
      A'
      \arrow[r]
      \arrow[d]
      &
      P
      \arrow[d]
      \\
      A''
      \arrow[r, equal]
      \arrow[d]
      &
      A''
      \arrow[d]
      \\
      0
      &
      0
      .
    \end{tikzcd}
  \end{equation}
  Moreover,
  it is well known that the newly formed row of \eqref{eq:extDown}
  and the newly formed column of \eqref{eq:extRight}
  are both exact;
  see for example
  \mbox{\cite[\href{https://stacks.math.columbia.edu/tag/010I}{Tag 010I}]{stacks-project}}.
  So on the one hand,
  we obtain naturally induced chain maps as in
  \begin{equation}
    \label{eq:pushoutDirectSum}
    \begin{tikzcd}
      A''
      &
      &
      C
      \\
      &
      P/A
      \arrow[lu, dashed, two heads]
      \arrow[ru, dashed, two heads]
      \\
      A''
      \arrow[ru, tail, dashed]
      \arrow[uu, equal]
      &
      &
      C
      \arrow[lu, tail, dashed]
      \arrow[uu, equal]
    \end{tikzcd}
  \end{equation}
  exhibiting ${P / A}$ as a direct sum
  ${A'' \oplus C}$
  in conjunction with
  \cref{lem:pushoutDirectSum},
  and on the other hand
  we obtain the short exact sequence
  \begin{equation}
    \label{eq:sesPB'C}
    0 \rightarrow P \rightarrow B' \rightarrow C \rightarrow 0
  \end{equation}
  from \cref{lem:sesPushout3by3}.
  Now let
  ${\partial_P \colon H_{n+1}(C'') \rightarrow H_n(P)}$
  be the boundary map obtained from the zig-zag lemma
  in degree ${n+1}$
  when applied to the short exact sequence \eqref{eq:sesPB'C}.
  Similarly,
  we obtain the exact sequence
  \begin{equation}
    \label{eq:boundaryZigZagQuotient}
    \begin{tikzcd}
      H_n(P)
      \arrow[r, "H_n(\mathrm{pr})"]
      &[3ex]
      H_n(P/A)
      \arrow[r, "\partial'"]
      &
      H_{n-1}(A)
    \end{tikzcd}
  \end{equation}
  from the zig-zag lemma applied to the short exact sequence
  \begin{equation*}
    \begin{tikzcd}
      0
      \arrow[r]
      &
      A
      \arrow[r]
      &
      P
      \arrow[r, "\mathrm{pr}"]
      &
      P/A
      \arrow[r]
      &
      0
      .
    \end{tikzcd}
  \end{equation*}
  Combining these induced maps in homology with
  \eqref{eq:pushoutDirectSum}
  we obtain the commutative hexagon
  \begin{equation*}
    \begin{tikzcd}
      &
      H_{n+1}(C'')
      \arrow[dd, "{H_n(\mathrm{pr}) \circ \partial_P}" description]
      \arrow[dl, "\partial''"']
      \arrow[dr, "\partial_C"]
      \\
      H_n(A'')
      &
      &
      H_n(C)
      \\
      &
      H_n(P/A)
      \arrow[lu, two heads]
      \arrow[ru, two heads]
      \arrow[dd, "\partial'"]
      \\
      H_n(A'')
      \arrow[ru, tail]
      \arrow[uu, equal]
      \arrow[dr, "\partial_A"']
      &
      &
      H_n(C)
      \arrow[lu, tail]
      \arrow[uu, equal]
      \arrow[dl, "\partial"]
      \\
      &
      H_{n-1}(A)
      .
    \end{tikzcd}
  \end{equation*}
  In conjunction with
  \mbox{\cite[Hexagon Lemma 11.1.3]{MR2456045}}
  and the exactness of \eqref{eq:boundaryZigZagQuotient}
  we obtain the result.
\end{proof}

\section{Constructing Lifts of Points}
\label{sec:formal_lifts_points}

In this appendix we provide formal definitions
for the construction of lifts of points $u \in \mathbb{I}$
from \cref{sec:liftingPoints}.
More specifically,
we provide
formal definitions
of the simplicial complex $X_u$
as well as the maps
\begin{align*}
  r_u \colon X_u & \rightarrow [0, 1]
                   ,
  \\
  b_u \colon [0, 1] & \rightarrow \R
                      ,
  \\
  \text{and} \quad
  j_u \colon [0, 1] & \xrightarrow{\cong} A_u \subset X_u
\end{align*}
for each $u \in \mathbb{I}$.

\begin{figure}[t]
  \centering
  \includegraphics{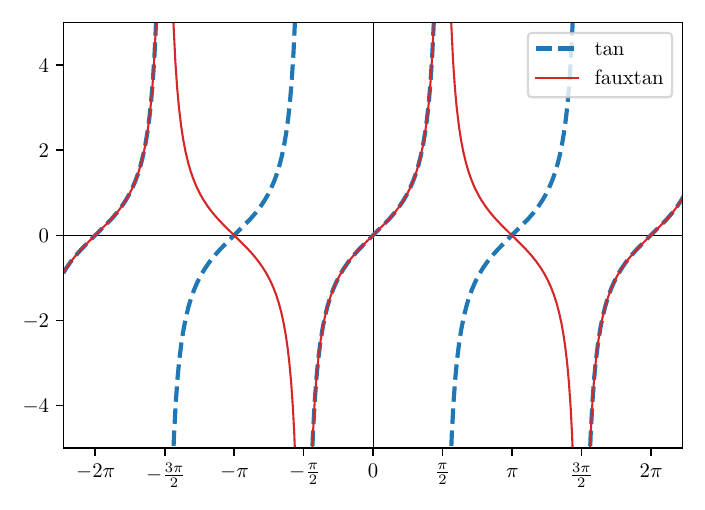}
  \caption{The modified tangent function $\op{fauxtan}\colon \R
  \rightarrow \overline\R$.}
  \label{fig:fauxtan}
\end{figure}

We start with defining the affine function
$b_u \colon [0, 1] \rightarrow \R$
for each $u = (u_1, u_2) \in \mathbb{S}$
by the formula
\begin{equation}
  \label{eq:bu}
  b_u \colon %
  [0, 1] \rightarrow \R,
  \,
  \begin{cases}
    0 \mapsto \min\, \op{fauxtan} (\{u_1, u_2\})
    \\
    1 \mapsto \max \op{fauxtan} (\{u_1, u_2\})
    ,
  \end{cases}
\end{equation}
where
\({\op{fauxtan} \colon \R
  \rightarrow \overline\R},\)
shown in \cref{fig:fauxtan},
is the unique continuous
extension
of the restricted ordinary tangent function
\(\tan \colon \left(-\frac{\pi}{2}, \frac{\pi}{2}\right) \rightarrow \R\)
that is symmetric with respect to translation by $2 \pi$ and reflection at $\frac{\pi}{2}$:
$\op{fauxtan}(x)=\op{fauxtan}(x+2\pi)=\op{fauxtan}(\pi-x)$ for all $x \in \R$, and $\op{fauxtan}(\pm\frac{\pi}{2})=\pm\infty$.
As a side note, for each $u \in R$ we have
$\op{Dgm}(b_u) = \mathbf{1}_u$ and hence $b_u = f_u$.

Next we construct the simplicial complex $X_u$
as well as the two simplicial maps
$j_u \colon [0, 1] %
\rightarrow X_u$ and
$r_u \colon X_u \rightarrow [0, 1] %
$
in terms of the corresponding vertex maps
for each $u \in \mathbb{I}$.
To this end,
let $Q := \uparrow (\frac{\pi}{2}, -\frac{\pi}{2})$ be the principal upset
of $(\frac{\pi}{2}, -\frac{\pi}{2})$.
Then $E := Q \setminus T (Q)$ is a fundamental domain
for the $\Z$-action on $\M$, and the map
\begin{equation*}
  \N_{>0} \times E \rightarrow \M \setminus Q,
  (n, v) \mapsto T^{-n} (v)
\end{equation*}
is a bijection%
.
Now the intersection
\begin{equation*}
  E
  \cap
  \left(\left(-\frac{\pi}{2}, \frac{\pi}{2}\right)^2 + \pi \Z^2\right)
\end{equation*}
has three connected components;
namely
\begin{align*}
  E_1 & :=
        E \cap
        \left(-\frac{3 \pi}{2}, -\frac{\pi}{2}\right) \times
        \left(-\frac{\pi}{2}, \frac{\pi}{2}\right)
        ,
        \\
  E_2 & :=
        E \cap
        \left(-\frac{\pi}{2}, \frac{\pi}{2}\right) \times
        \left( \frac{\pi}{2}, \frac{3 \pi}{2}\right)
        , \quad \text{and}
  \\
  E_3 & :=
        {E} \mathbin{{\cap}}
        \left(-\frac{\pi}{2}, \frac{\pi}{2}\right)^2
\end{align*}
as shown in \cref{fig:regionsFormalPointLifts}.
\begin{figure}[t]
  \centering
\begingroup%
  \makeatletter%
  \providecommand\color[2][]{%
    \errmessage{(Inkscape) Color is used for the text in Inkscape, but the package 'color.sty' is not loaded}%
    \renewcommand\color[2][]{}%
  }%
  \providecommand\transparent[1]{%
    \errmessage{(Inkscape) Transparency is used (non-zero) for the text in Inkscape, but the package 'transparent.sty' is not loaded}%
    \renewcommand\transparent[1]{}%
  }%
  \providecommand\rotatebox[2]{#2}%
  \newcommand*\fsize{\dimexpr\f@size pt\relax}%
  \newcommand*\lineheight[1]{\fontsize{\fsize}{#1\fsize}\selectfont}%
  \ifx\svgwidth\undefined%
    \setlength{\unitlength}{290.37734985bp}%
    \ifx\svgscale\undefined%
      \relax%
    \else%
      \setlength{\unitlength}{\unitlength * \real{\svgscale}}%
    \fi%
  \else%
    \setlength{\unitlength}{\svgwidth}%
  \fi%
  \global\let\svgwidth\undefined%
  \global\let\svgscale\undefined%
  \makeatother%
  \begin{picture}(1,0.49074076)%
    \lineheight{1}%
    \setlength\tabcolsep{0pt}%
    \put(0,0){\includegraphics[width=\unitlength,page=1]{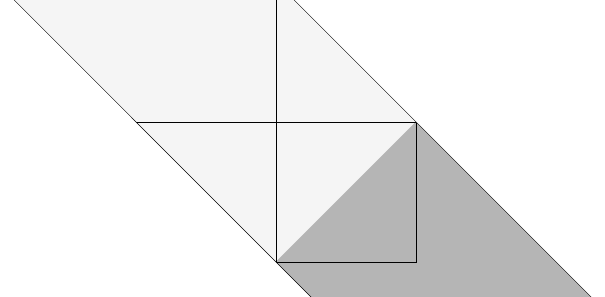}}%
    \put(0.31272743,0.17361112){\makebox(0,0)[rt]{\lineheight{1.25}\smash{\begin{tabular}[t]{r}\(l_0\)\end{tabular}}}}%
    \put(0.82045973,0.17361112){\makebox(0,0)[lt]{\lineheight{1.25}\smash{\begin{tabular}[t]{l}\(l_1\)\end{tabular}}}}%
    \put(0.39814815,0.22115011){\makebox(0,0)[t]{\lineheight{1.25}\smash{\begin{tabular}[t]{c}{\fontsize{13.5pt}{13.5pt}\selectfont \(E_1\)}\end{tabular}}}}%
    \put(0.5138889,0.33689086){\makebox(0,0)[t]{\lineheight{1.25}\smash{\begin{tabular}[t]{c}{\fontsize{13.5pt}{13.5pt}\selectfont \(E_2\)}\end{tabular}}}}%
    \put(0.57175927,0.16327974){\makebox(0,0)[t]{\lineheight{1.25}\smash{\begin{tabular}[t]{c}{\fontsize{13.5pt}{13.5pt}\selectfont \(E_3\)}\end{tabular}}}}%
  \end{picture}%
\endgroup%

  \caption{The three connected components $E_1$, $E_2$, and $E_3$.}
  \label{fig:regionsFormalPointLifts}
\end{figure}
Altogether we obtain the bijection
\begin{equation*}
  \N_{>0} \times (E_1 \cup E_2 \cup E_3) \xrightarrow{~\cong~} \mathbb{I},~
  (n, v) \mapsto T^{-n} (v)
  .
\end{equation*}

In the following we denote the standard $n$-simplex by $\Delta^n$
for $n \in \N_0$.
Its vertex set is given by the standard basis
${\{e_0, \dots, e_n\} \subset \R^{n+1}}$.
Now let ${u \in \mathbb{I}}$ and let ${n \in \N_{>0}}$ such that
${v := T^n(u) \in E}$.
We distinguish between three cases
depending on which of the three connected components of $E$ contains $v$.
\begin{description}
\item[For $v \in E_1,$] we set
  \begin{align*}
    X_u & := \Lambda^{n+1}_{n+1}, \\
    j_u & :\,
          [0, 1] \rightarrow X_u,\,
          \begin{cases}
            0 \mapsto e_0 \\
            1 \mapsto e_{n+1},
          \end{cases}
    \\
    \text{and} \quad
    r_u & :\,
          X_u \rightarrow [0, 1],~
          e_k \mapsto
          \begin{cases}
            0
            &
            0 \leq k \leq n
            \\
            1
            &
            k = n+1.
          \end{cases}
  \end{align*}
  Here $\Lambda^{n+1}_{n+1}$ is the simplicial complex we obtain
  from $\partial \Delta^{n+1}$
  after removing the facet opposite to the $(n+1)$-st vertex,
  keeping all other simplices.
\item[For $v \in E_2,$] we set
  \begin{align*}
    X_u & := \Lambda^{n+1}_{0}, \\
    j_u & :\,
          [0, 1] \rightarrow X_u,\,
          \begin{cases}
            0 \mapsto e_0 \\
            1 \mapsto e_{n+1},
          \end{cases}
    \\
    \text{and} \quad
    r_u & :\,
          X_u \rightarrow [0, 1],~
          e_k \mapsto
          \begin{cases}
            0
            &
            k = 0
            \\
            1
            &
            1 \leq k \leq n+1.
          \end{cases}
  \end{align*}
  Here $\Lambda^{n+1}_{0}$ is the simplicial complex we obtain
  from $\partial \Delta^{n+1}$
  after removing the facet opposite to the $0$-th vertex.  
\item[For $v \in E_3$] the construction is not as direct
  as in the previous two cases.
  We start by setting
  \begin{equation*}
    X_u := \left(\partial \Delta^{n+1}\right) \times [0, 1]
    \subset \R^{n+2} \times \R
  \end{equation*}
  with the following triangulation (which
  coincides with the triangulation induced
    by the corresponding simplicial set).
  The vertex set of $X_u$ is
  $\{e_0, \dots, e_{n+1}\} \times \{0, 1\}$
  and a set of pairs in $\{e_0, \dots, e_{n+1}\} \times \{0, 1\}$
  spans a simplex in $X_u$ iff the index of each basis vector
  that appears in a pair with $0$ is at most the index
  of any basis vector that appears in a pair with $1$
  and there is one basis vector not appearing in any pair.
  (The second condition ensures that we are not adding a simplex to $X_u$
  that is not even contained in the proclaimed underlying space of $X_u$.)
  With this triangulation, we define the simplicial map
  \begin{align*}
    j_u & :\,
          [0, 1] \rightarrow X_u,\,
          \begin{cases}
            0 \mapsto (e_0, 0) \\
            1 \mapsto (e_{n+1}, 1).
          \end{cases}
  \end{align*}
  To define the retraction $r_u \colon X_u \rightarrow [0, 1]$
  we have to distinguish between two cases
  depending on which of the two coordinates $v_1$ and $v_2$ of $v = (v_1, v_2)$
  is larger than or smaller than the other.
  For $v_1 \leq v_2$ we set
  \begin{equation*}
    r_u \colon X_u \rightarrow [0, 1],~
    (e_k, l) \mapsto
    \begin{cases}
      0
      &
      0 \leq k \leq n
      \\
      1
      &
      k = n+1
    \end{cases}    
  \end{equation*}
  and if ${v_1 > v_2}$,
  then we define $r_u$ to be the projection
  \[r_u := \pr_2 \colon
  \left(\partial \Delta^{n+1}\right) \times [0, 1]
  \rightarrow [0, 1],~
  (e_k, l) \mapsto l\]
  onto the second factor.
  Here the intuition is the following.
  When $v_1 = v_2$,
  then $b_u \colon [0, 1] \rightarrow \R$ is constant,
  so we can \enquote{flip} the retraction $r_u \colon X_u \rightarrow [0, 1]$
  without changing the composed map
  $f_u = b_u \circ r_u$.
\end{description}

\section{Sections of Sheaves on Intersections of Open Sets}
\label{sec:sectIntersect}

The sheaf condition essentially says that any compatible family of sections
on a cover
uniquely glues to a global section.
Given a cover by two open subsets,
we may conversely ask,
which are the sections on the intersection,
that can be obtained as the restriction of a section
on either of the two open subsets.
To this end, let $X$ be a topological space,
let $U_1 \cup U_2 = X$ be an open cover of $X$,
and let $F$ be a sheaf on $X$ with values in the category of abelian groups.
We consider the group of sections
$F(U_1 \cap U_2)$ of $F$ on the intersection of $U_1$ and $U_2$.
Some of these sections can be obtained by restriction
from sections of $F$ on $U_1$ or $U_2$.
Thus, there is an induced group homomorphism
\begin{equation*}
  \nabla_{U_1, U_2} (F) \colon F(U_1) \oplus F(U_2) \rightarrow F(U_1 \cap U_2)
  .
\end{equation*}
Now suppose we have open subsets $V_1, V_2 \subseteq X$
with $U_i \subseteq V_i$ for $i = 1,2$.
Then we have the commutative diagram
\begin{equation}
  \begin{tikzcd}
    \label{eq:inducedMapOnCokers}
    F(V_1) \oplus F(V_2)
    \arrow[r]
    \arrow[d, "\nabla_{V_1, V_2} (F)"']
    &
    F(U_1) \oplus F(U_2)
    \arrow[d, "\nabla_{U_1, U_2} (F)"]
    \\[2ex]
    F(V_1 \cap V_2)
    \arrow[r]
    \arrow[d]
    &
    F(U_1 \cap U_2)
    \arrow[d]
    \\
    \op{coker}(\nabla_{V_1, V_2} (F))
    \arrow[r, "\alpha"']
    &
    \op{coker}(\nabla_{U_1, U_2} (F))
  \end{tikzcd}
\end{equation}
as well as the induced map on cokernels
$\alpha \colon
\mathrm{coker}(\nabla_{V_1, V_2}) (F) \rightarrow
\op{coker}(\nabla_{U_1, U_2} (F))$.

\begin{lem}
  \label{lem:cokerSuperCover}
  The map
  $\alpha \colon
  \mathrm{coker}(\nabla_{V_1, V_2} (F)) \rightarrow
  \op{coker}(\nabla_{U_1, U_2} (F))$
  in the above diagram \eqref{eq:inducedMapOnCokers} is injective.
\end{lem}

\begin{proof}
  Identifying $H^0(U; F)$ with $F(U)$ for any open subset $U \subseteq X$
  and using the Mayer-Vietoris sequence for sheaf cohomology (see for example \cite[Remark 2.6.10]{MR1074006})
  we obtain the commutative diagram
  \begin{equation*}
    \begin{tikzcd}
      0 \ar[d]
      &
      0 \ar[d]
      \\
      H^0 (X; F)
      \ar[r, equal]
      \ar[d, "\begin{pmatrix} 1 \\ -1 \end{pmatrix}"']
      &
      H^0 (X; F)
      \ar[d, "\begin{pmatrix} 1 \\ -1 \end{pmatrix}"]
      \\[5ex]
      H^0 (V_1; F) \oplus H^0(V_2; F)
      \ar[r]
      \ar[d, "\nabla_{V_1, V_2} (F)"']
      &
      H^0 (U_1; F) \oplus H^0 (U_2; F)
      \ar[d, "\nabla_{U_1, U_2} (F)"]
      \\[2ex]
      H^0 (V_1 \cap V_2; F)
      \ar[r]
      \ar[d]
      &
      H^0 (U_1 \cap U_2; F)
      \ar[d]
      \\
      H^1 (X; F)
      \ar[r, equal]
      &
      H^1 (X; F)    
    \end{tikzcd}
  \end{equation*}
  with exact columns.
  By the exactness of these two columns
  at ${H^0 (V_1 \cap V_2; F)}$ and ${H^0 (U_1 \cap U_2; F)}$ respectively
  we obtain the commutative diagram
  \begin{equation*}
    \begin{tikzcd}
      \op{coker}(\nabla_{V_1, V_2} (F))
      \arrow[r, "\alpha"]
      \ar[d, tail]
      &
      \op{coker}(\nabla_{U_1, U_2} (F))
      \ar[d, tail]
      \\
      H^1 (X; F)
      \ar[r, equal]
      &
      H^1 (X; F)      
    \end{tikzcd}
  \end{equation*}
  with the two vertical maps monomorphisms of abelian groups as indicated.
  Thus, the map
  $\alpha \colon
  \mathrm{coker}(\nabla_{V_1, V_2} (F)) \rightarrow
  \op{coker}(\nabla_{U_1, U_2} (F))$
  is a monomorphism as well.
\end{proof}

\end{document}